\documentclass[12pt]{article}

\makeatletter
\newcommand{\vast}{\bBigg@{4}}
\newcommand{\Vast}{\bBigg@{5}}
\makeatother

\usepackage{amssymb,amsmath}

\usepackage{url,graphicx,subfig}  
\usepackage{mathrsfs}  
\usepackage{url}
\usepackage{color}

\usepackage{float}

\usepackage{tikz}
\usepackage{caption}

\usepackage{mathtools}
\usepackage{graphicx}
\usepackage{amssymb,amsmath,wasysym}
\usepackage{amsfonts}
\usepackage{comment}
\usepackage{enumerate}
\usepackage{color}

\usepackage[margin=1.5cm]{geometry}

\newtheorem{Theorem}{Theorem}
\newtheorem{Corollary}{Corollary}
\newtheorem{Lemma}{Lemma}

\newtheorem{Notes}{Note}

\newcommand{\bxi}{\mbox{\boldmath$\xi$}}
\newcommand{\bpi}{\mbox{\boldmath$\pi$}}
\newcommand{\bPsi}{\mbox{\boldmath$\Psi$}}

\newcommand{\balpha}{\mbox{\boldmath$\alpha$}}

\newcommand{\bnu}{\mbox{\boldmath$\nu$}}
\newcommand{\btheta}{\mbox{\boldmath$\theta$}}
\newcommand{\bXi}{\mbox{\boldmath$\Xi$}}

\title{Sensitivity analysis of Stochastic Fluid Models: \\Stationary and transient quantities with applications}

\date{\normalsize \today}

\author{
Anna Aksamit\thanks{School of Mathematics and Statistics, The University of Sydney, NSW 2006, Australia, email:
	anna.aksamit@sydney.edu.au
}
\thanks{Anna Aksamit was supported by the Australian Research Council Discovery Project DP220103106.}
\and	
Ma{\l}gorzata M. O'Reilly\thanks{University of Tasmania, Hobart, TAS 7001, Australia, email: malgorzata.oreilly@utas.edu.au}
\thanks{Ma{\l}gorzata M. O'Reilly was supported by the Australian Research Council Discovery Project DP180100352.}
\and
Zbigniew Palmowski\thanks{Faculty of Pure and Applied Mathematics, Wroc{\l}aw University of Science and Technology, ul. Wybrze\.{z}e Wyspia\'{n}skiego 27,
50-370 Wroc{\l}aw, Poland email: zbigniew.palmowski@pwr.edu.pl}
\thanks{Zbigniew Palmowski was partially supported by the National Science Centre (Poland) under the grant 2021/41/B/HS4/00599.}
}

\begin{document}

\maketitle

\begin{abstract}
We establish results for the first sensitivity analysis of the stochastic fluid models (SFMs). We derive expressions for the sensitivity analysis of the key stationary and transient (time-dependent) quantities of this class of models. We also construct numerical examples to demonstrate the application potential of our methodology in queueing systems, such as deteriorating systems and insurance risk processes. This work forms foundation for the sensitivity analysis of other Markovian modulated models, which are generalisations of the SFMs, and have widespread applications.
\end{abstract}

\medskip



\section{Introduction}

A stochastic fluid model (SFM), also referred to as a Markov-Modulated Fluid Flow, is one of the fundamental models in matrix-analytic methods (MAMs). Pioneered by Neuts, MAMs introduced transformative algorithmic approaches to queueing theory and applied probability~\cite{Neuts}. Because SFMs can elegantly capture the behavior of continuous quantities driven by a discrete background environment, they have become an indispensable tool across a variety of fields, ranging from telecommunications and operations research to risk theory. For a comprehensive overview of the results in the analysis of standard SFMs, the reader is referred to Ramaswami~\cite{Rama}, Ahn and Ramaswami~\cite{AR}, da Silva Soares~\cite{DaSilva}, and Bean, O'Reilly and Taylor~\cite{BOT05a,BOT05,BOT08}.

The state space of a standard SFM $\{(X(t),\varphi(t)):t\ge0\}$ consists of the continuous level variable $X(t)\ge0$, a phase variable $\varphi(t)\in\mathcal{S}=\{1,...,m\}$, and fluid rates $c_{i}\in\mathbb{R}$ for all $i\in\mathcal{S}$. The level $X(t)$ is modulated by a continuous-time Markov chain (CTMC) $\{\varphi(t):t\ge0\}$ with generator $T=[T_{ij}]_{i,j\in\mathcal{S}}$, so that
\begin{equation}\label{SFM1}
\frac{dX(t)}{dt}=c_{\varphi(t)}I(X(t)>0)+\max\{0,c_{\varphi(t)}\}I(X(t)=0),
\end{equation}
where $I(\cdot)$ denotes an indicator function. That is, whenever $X(t)$ is positive, the level changes at the rate $c_{\varphi(t)}$. However, if $X(t)$ is zero, then the level changes at the rate $c_{\varphi(t)}$ provided $c_{\varphi(t)}>0$, and remains constant if $c_{\varphi(t)}\le0$. In this sense, the CTMC models the evolution of some underlying environment that in turn drives the evolution of the continuous level variable within the SFM. We also consider the SFM without the zero boundary defined by
\begin{equation}\label{SFM2}
\frac{d \widetilde{X}(t)}{dt}=c_{\varphi(t)}.
\end{equation}

While the fundamental evaluation of SFMs has been extensively studied, real-world applications often demand an understanding of how robust these models are to changes in their underlying parameters. Sensitivity analysis provides a rigorous mathematical framework to quantify the impact of small perturbations or uncertainties in system parameters, such as the transition rates of the environmental Markov chain or the state-dependent fluid rates, on the overall performance metrics in stationary and transient regimes. Historically, the perturbation and sensitivity analysis of finite Markov chains has been a rich area of research, with foundational matrix-theoretic bounds and derivatives established by Schweitzer \cite{schweitzer1968} and Meyer \cite{meyer1980}. This was later expanded into a broader sensitivity-based optimization framework by Cao \cite{cao2007} and Taylor series expansions for Markov processes by Heidergott et al. \cite{heidergott2006}.

In the context of matrix-analytic methods, so far the sensitivity analysis has predominantly focused on models with a discrete level variable, namely on
Quasi-Birth-and-Death processes; for more information about this family of the processes see Neuts~\cite{Neuts}, Ramaswami and Latouche~\cite{GLVR_1999}, Ramaswami~\cite{ram1997}, Joyner and Fralix~\cite{joyner2016new}, and Phung-Duc et al.~\cite{phung2010simple}. This sensitivity analysis is developed in Aksamit, O'Reilly and Palmowski~\cite{doi:10.1080/15326349.2024.2325448}; see also references therein.

Despite these advancements in discrete-level models, there remains a critical gap in the literature regarding continuous-state models. To the best of our knowledge, this paper establishes the first sensitivity analysis of stochastic fluid models. In this paper, we employ matrix derivative calculus to derive explicit and computationally tractable expressions for the sensitivity of the key quantities of SFMs with respect to a general parameter vector $\theta$.

Our main contributions are multi-fold. First, we establish the sensitivity derivatives for the foundational building blocks of SFM analysis, specifically focusing on the Laplace transforms of the first return to level $0$, $\bPsi(s;\theta)$ and $\bXi(s;\theta)$ for the SFM $X(t)$ and $\widetilde{X}(t)$, respectively. These matrices act as continuous-state analogues to the ${\bf G}$ matrix of QBDs. Building upon these intermediate results, we derive analytical sensitivity measures for the stationary probability density vectors. Furthermore, we extend our methodology to transient (time-dependent) quantities.

To demonstrate the practical applicability of our methodology, we construct numerical examples that show how this theoretical framework can be deployed to used in real-world systems. Specifically, we apply our results to queueing systems present in the management of hydro power generation systems. We also consider the insurance risk processes, where understanding the sensitivity of ruin probabilities to premium rates or claim intensities is vital. Ultimately, this work forms the essential mathematical foundation for the sensitivity analysis that might be applied for other Markovian modulated models generalizing SFMs.

The rest of the paper is structured as follows. In Section \ref{sec:SenAn}, we summarize the relevant expressions from matrix calculus and develop the results for the quantities that are the key building blocks in the sensitivity analysis of the SFMs. Next, the stationary and transient analysis results are developed in Sections \ref{sec:StaQua} and \ref{sec:TraQua}, respectively. This is followed by numerical examples in Section \ref{sec:Exa}, where we illustrate the application of the theory and demonstrate the application potential of our methodology to queueing systems and risk theory. Finally, in Section \ref{sec:Conclusion}, we provide our concluding remarks.

\section{Sensitivity analysis}\label{sec:SenAn}

Suppose $\btheta$ is a vector of some parameters of the SFM $\{(X(t),\varphi(t)):t\ge0\}$, which has an effect on the values of the generator ${\bf T}=[T_{ij}]_{i,j\in\mathcal{S}}$ or fluid rates ${\bf C}=\mbox{diag}(c_i)_{i\in\mathcal{S}}$. Throughout the analysis that follows, we write the expressions for transient and stationary quantities of interest as functions of $\btheta$ and evaluate the derivatives $\frac{\partial}{\partial\btheta}$ of these functions with respect to $\btheta$. We denote drift \[\mu=\sum_ic_i\nu_i,\] where $\bnu=[\nu_i]$ is the stationary distribution of the Markov chain $\{\varphi(t):t\geq 0\}$.

\subsection{Matrix derivative calculus review}
In our analysis, we apply the matrix derivative calculus,
\begin{align}
\frac{\partial}{\partial\btheta}
\left(
{\bf A}(\btheta)\times {\bf B}( \btheta)
\right)
&=
\frac{\partial {\bf A}(\btheta)}{ \partial \btheta}
\times\left(
{\bf I}\otimes {\bf B}(\btheta)
\right)
+
{\bf A}(\btheta)
\times
\frac{\partial {\bf B}(\btheta)}{\partial \btheta}
,
\label{eq:mproduct}\\
\frac{\partial ({\bf A}(\btheta))^{-1}}{\partial \btheta}
&=
-({\bf A}(\btheta))^{-1}
\times
\frac{\partial {\bf A}(\btheta)}{\partial \btheta}
\times
\left(
{\bf I}
\otimes ({\bf A}(\btheta))^{-1}
\right)
,
\label{eq:minverse}
\end{align}
where ${\bf I}$ denotes an identity matrix of an appropriate size.

We will also be dealing with expressions in the form of a matrix exponential, and so we are interested in evaluating
\begin{eqnarray*}
\frac{\partial
}
{ \partial \btheta}
e^{
{\bf A}\left(\btheta\right)x
}
&=&
\frac{\partial }
{ \partial \btheta}
\sum_{n=0}^{\infty}
({\bf A}(\btheta))^n
\frac{x^n}{n!}
=
\frac{\partial }
{ \partial \btheta}
\sum_{n=1}^{\infty}
({\bf A}(\btheta))^n
\frac{x^n}{n!}.
\end{eqnarray*}
Since for $n\geq 1$, we have
\begin{eqnarray*}
\frac{\partial
\left(
{\bf A}(\btheta)
\right)^n
}
{ \partial \btheta}
&=&
\frac{\partial
\left(
{\bf A}(\btheta)
\right)^{n-1}
}
{ \partial \btheta}
\left(
{\bf I}\otimes
{\bf A}(\btheta)
\right)
+
\left(
{\bf A}(\btheta)
\right)^{n-1}
\frac{\partial
{\bf A}(\btheta)
}
{ \partial \btheta},
\end{eqnarray*}
it follows by mathematical induction that
\begin{eqnarray*}
\frac{\partial
\left(
{\bf A}(\btheta)
\right)^n
}
{ \partial \btheta}
&=&
\sum_{k=0}^{n-1}
\left(
{\bf A}(\btheta)
\right)^k
\frac{\partial
{\bf A}(\btheta)
}
{ \partial \btheta}
\left(
{\bf A}(\btheta)
\right)^{(n-1)-k}
,
\end{eqnarray*}
and so
\begin{eqnarray*}
\frac{\partial
}
{ \partial \btheta}
e^{
{\bf A}\left(\btheta\right)x
}
&=&
\sum_{n=1}^{\infty}
\sum_{k=0}^{n-1}
\left(
{\bf A}(\btheta)
\right)^k
\frac{\partial
{\bf A}(\btheta)
}
{ \partial \btheta}
\left(
{\bf A}(\btheta)
\right)^{(n-1)-k}
\frac{x^n}{n!}.
\end{eqnarray*}

Nevertheless, a convenient approach to evaluating $\frac{\partial
}
{ \partial \btheta}
e^{
{\bf A}\left(\btheta\right)x
}$, provided the inverse $-({\bf A}(\btheta)-v{\bf I})^{-1}$ exists, is to numerically invert the Laplace transform given by
\begin{eqnarray}
\int_{x=0}^{\infty}e^{-vx}
\frac{\partial
}
{ \partial \btheta}
e^{
{\bf A}\left(\btheta\right)x
}
dx
&=&
\frac{\partial
}
{ \partial \btheta}
\int_{x=0}^{\infty}e^{-vx}
e^{
{\bf A}\left(\btheta\right)x
}dx
\nonumber
\\
&=&
-
\frac{\partial
}
{ \partial \btheta}
\left(
{\bf A}(\btheta)-v{\bf I}
\right)^{-1}
\nonumber
\\
&=&
({\bf A}(\btheta)-v{\bf I})^{-1}
\times
\frac{\partial {\bf A}(\btheta)}{\partial \btheta}
\times
\left(
{\bf I}
\otimes ({\bf A}(\btheta)-v{\bf I})^{-1}
\right),
\label{eq:exptrick}
\end{eqnarray}
and so we will adopt this approach in our analysis whenever suitable.

\subsection{Key matrices: $\frac{\partial}{\partial\btheta}\bPsi(s;\btheta)$ and $\frac{\partial}{\partial\btheta}\bXi(s;\btheta)$}

The key quantity in the analysis of the SFMs is the matrix $\bPsi(s;\btheta)=[\Psi(s;\btheta)_{ij}]_{i\in\mathcal{S}_+,j\in\mathcal{S}_-}$ recording the Laplace transform of the distribution of the busy period, such that for $s$ with $Re(s)>0$, $\Psi(s;\btheta)_{ij}=\mathbb{E}(e^{-s\delta(0)} I(\varphi(\delta(0))=j)\ |\ X(0)=0,\varphi(0)=i)$, where $\delta(x)=\inf\{t>0:X(t)=x\}$ is the first time to hit level $x$ after time zero.

We are interested in evaluating $\frac{\partial}{\partial\btheta}\bPsi(s;\btheta)$ first, since  the remaining quantities of interest may be expressed using this key derivative.

Denote $\mathcal{S}_+=\{i\in\mathcal{S}:c_i>0\}$, $\mathcal{S}_-=\{i\in\mathcal{S}:c_i<0\}$, $\mathcal{S}_0=\{i\in\mathcal{S}:c_i=0\}$.
Let ${\bf C}_+=\mbox{diag}(c_i)_{i\in\mathcal{S}_+}$ and ${\bf C}_-=\mbox{diag}(c_i)_{i\in\mathcal{S}_-}$.
We also need the following notation ${\bf T}_{+-}=[T_{ij}]_{i\in\mathcal{S}_+,j\in\mathcal{S}_-}$, ${\bf T}_{-+}=[T_{ij}]_{i\in\mathcal{S}_-,j\in\mathcal{S}_+}$,
${\bf T}_{++}=[T_{ij}]_{i\in\mathcal{S}_+,j\in\mathcal{S}_+}$, ${\bf T}_{--}=[T_{ij}]_{i\in\mathcal{S}_-,j\in\mathcal{S}_-}$. Similarly, we denote ${\bf T}_{0+}$, ${\bf T}_{+0}$, ${\bf T}_{00}$, ${\bf T}_{0-}$, ${\bf T}_{-0}$.
We now consider the fluid generator ${\bf Q}(s;\btheta)$ introduced in Bean, O'Reilly and Taylor~\cite{BOT05},
\begin{eqnarray*}
{\bf Q}(s;\btheta)&=&
\left[
\begin{array}{cc}
{\bf Q}_{++}(s;\btheta)&{\bf Q}_{+-}(s;\btheta)\\
{\bf Q}_{++}(s;\btheta)&{\bf Q}_{+-}(s;\btheta)
\end{array}
\right]
,
\\
{\bf Q}_{++}(s;\btheta)
&=&
({\bf C}_+(s;\btheta))^{-1}
\left(
{\bf T}_{++}(\btheta)-s{\bf I}
-{\bf T}_{+0}
\left(
{\bf T}_{00}(\btheta)-s{\bf I}
\right)^{-1}
{\bf T}_{0+}(\btheta)
\right)
,\\
{\bf Q}_{+-}(s;\btheta)
&=&
({\bf C}_+(s;\btheta))^{-1}
\left(
{\bf T}_{+-}(\btheta)
-{\bf T}_{+0}
\left(
{\bf T}_{00}(\btheta)-s{\bf I}
\right)^{-1}
{\bf T}_{0-}(\btheta)
\right)
,\\
{\bf Q}_{--}(s;\btheta)
&=&
|{\bf C}_-(s;\btheta)|^{-1}
\left(
{\bf T}_{--}(\btheta)-s{\bf I}
-{\bf T}_{-0}
\left(
{\bf T}_{00}(\btheta)-s{\bf I}
\right)^{-1}
{\bf T}_{0-}(\btheta)
\right)
,\\
{\bf Q}_{-+}(s;\btheta)
&=&
|{\bf C}_-(s;\btheta)|^{-1}
\left(
{\bf T}_{-+}(\btheta)
-{\bf T}_{-0}
\left(
{\bf T}_{00}(\btheta)-s{\bf I}
\right)^{-1}
{\bf T}_{0+}(\btheta)
\right),
\end{eqnarray*}
which is the key generator used in the expressions of the quantities in the analysis of the SFMs.

Denote
\begin{eqnarray*}
{\bf D}(s;\btheta)
&=&
{\bf Q}_{--}(s;\btheta)
+
{\bf Q}_{-+}(s;\btheta)\bPsi(s;\btheta)
,\\
{\bf U}(s;\btheta)
&=&
{\bf Q}_{++}(s;\btheta)
+
{\bf Q}_{+-}(s;\btheta)\bXi(s;\btheta)
,\\
{\bf K}(s;\btheta)
&=&
{\bf Q}_{++}(s;\btheta)
+
\bPsi(s;\btheta){\bf Q}_{-+}(s;\btheta)
,\\
{\bf J}(s;\btheta)
&=&
{\bf Q}_{--}(s;\btheta)
+
\bXi(s;\btheta){\bf Q}_{+-}(s;\btheta)
,
\end{eqnarray*}
and let ${\bf D}(\btheta)={\bf D}(0;\btheta)$, ${\bf U}(\btheta)={\bf U}(0;\btheta)$, ${\bf K}(\btheta)={\bf K}(0;\btheta)$, ${\bf J}(\btheta)={\bf J}(0;\btheta)$.

\begin{Theorem}\label{th:PsiDer}
For $s$ with $Re(s)>0$, the matrix $\frac{\partial}{\partial\btheta}\bPsi(s;\btheta)$ is the solution of the equation
\begin{eqnarray*}
{\bf A}{\bf X}+{\bf X}{\bf B}&=&-{\bf D},
\end{eqnarray*}
where
\begin{eqnarray*}
{\bf A}&=&
{\bf K}(s;\btheta)
,\\
{\bf B}&=&
{\bf I}\otimes {\bf D}(s;\btheta)
,\\
{\bf D}&=&
\frac{\partial
{\bf Q}_{+-}(s;\btheta)
}{\partial \btheta}
+
\frac{\partial {\bf Q}_{++}(s;\btheta)}
{ \partial \btheta}
\times\left(
{\bf I}\otimes \bPsi(s;\btheta)
\right)
+
\bPsi(s;\btheta)
\times
\frac{\partial  {\bf Q}_{--}(s;\btheta)}
{\partial \btheta}
\nonumber
\\
&&
+
\bPsi(s;\btheta)
\times
\left\{
\frac{\partial {\bf Q}_{-+}(s;\btheta)}
{ \partial \btheta}
\times\left(
{\bf I}\otimes \bPsi(s;\btheta)
\right)
\right\}
.
\end{eqnarray*}
\end{Theorem}

{\bf Proof:} By Bean, O'Reilly and Taylor~\cite{BOT05,BOT08}, for $s$ with $Re(s)>0$, $\bPsi(s;\btheta)$ is the minimum nonnegative solution of
\begin{eqnarray}\label{eq:PsiRic}
{\bf 0}
&=&
{\bf Q}_{+-}(s;\btheta)
+
{\bf Q}_{++}(s;\btheta)\bPsi(s;\btheta)
+
\bPsi(s;\btheta){\bf Q}_{--}(s;\btheta)
+
\bPsi(s;\btheta)
{\bf Q}_{-+}(s;\btheta)
\bPsi(s;\btheta)
.
\end{eqnarray}

Therefore, by the matrix derivative calculus~\eqref{eq:mproduct}-\eqref{eq:minverse},
\begin{eqnarray*}
{\bf 0}
&=&
\frac{\partial
{\bf Q}_{+-}(s;\btheta)
}{\partial \btheta}
\\
&&
+
\frac{\partial {\bf Q}_{++}(s;\btheta)}
{ \partial \btheta}
\times\left(
{\bf I}\otimes \bPsi(s;\btheta)
\right)
+
{\bf Q}_{++}(s;\btheta)
\times
\frac{\partial  \bPsi(s;\btheta)}
{\partial \btheta}
\\
&&
+
\frac{\partial \bPsi(s;\btheta)}
{ \partial \btheta}
\times\left(
{\bf I}\otimes {\bf Q}_{--}(s;\btheta)
\right)
+
\bPsi(s;\btheta)
\times
\frac{\partial  {\bf Q}_{--}(s;\btheta)}
{\partial \btheta}
\\
&&
+
\frac{\partial \bPsi(s;\btheta)}
{\partial \btheta}
\times
\left\{
{\bf I}\otimes
{\bf Q}_{-+}(s;\btheta)
\bPsi(s;\btheta)
\right\}
\\
&&
+
\bPsi(s;\btheta)
\times
\left\{
\frac{\partial {\bf Q}_{-+}(s;\btheta)}
{ \partial \btheta}
\times\left(
{\bf I}\otimes \bPsi(s;\btheta)
\right)
+
{\bf Q}_{-+}(s;\btheta)
\times
\frac{\partial  \bPsi(s;\btheta)}
{\partial \btheta}
\right\}
,
\end{eqnarray*}
which gives
\begin{eqnarray*}
\lefteqn{
\Big(
{\bf Q}_{++}(s;\btheta)
+
\bPsi(s;\btheta)
{\bf Q}_{-+}(s;\btheta)
\Big)
\times
\frac{\partial  \bPsi(s;\btheta)}
{\partial \btheta}
}
\\
&&
+
\frac{\partial  \bPsi(s;\btheta)}
{\partial \btheta}
\times
\Big(
{\bf I}\otimes {\bf Q}_{--}(s;\btheta)
+
{\bf I}\otimes {\bf Q}_{-+}(s;\btheta)
\bPsi(s;\btheta)
\Big)
\\
&=&
-
\frac{\partial
{\bf Q}_{+-}(s;\btheta)
}{\partial \btheta}
-
\frac{\partial {\bf Q}_{++}(s;\btheta)}
{ \partial \btheta}
\times\left(
{\bf I}\otimes \bPsi(s;\btheta)
\right)
-
\bPsi(s;\btheta)
\times
\frac{\partial  {\bf Q}_{--}(s;\btheta)}
{\partial \btheta}
\\
&&
-
\bPsi(s;\btheta)
\times
\left\{
\frac{\partial {\bf Q}_{-+}(s;\btheta)}
{ \partial \btheta}
\times\left(
{\bf I}\otimes \bPsi(s;\btheta)
\right)
\right\},
\end{eqnarray*}
and the result follows. \rule{9pt}{9pt}

For the SFM $\{(\widetilde{X}(t),\varphi(t)):t\geq 0\}$ defined in \eqref{SFM2} (without the zero boundary), we also consider the key quantity $\bXi(s;\btheta)=[\Xi(s;\btheta)_{ij}]_{i\in\mathcal{S}_-,j\in\mathcal{S}_+}$ symmetrical to $\bPsi(s;\btheta)$, such that for $s$ with $Re(s)>0$, $\Xi(s;\btheta)_{ij}=\mathbb{E}(e^{-s\widetilde\delta(0)} I(\varphi(\widetilde\delta(0))=j)\ |\ \widetilde X(0)=0,\varphi(0)=i)$, where $\widetilde\delta(x)=\inf\{t>0:\widetilde X(t)=x\}$ . Then we have the following result, by symmetry.

\begin{Theorem}\label{th:KsiDer}
For $s$ with $Re(s)>0$, the matrix $\frac{\partial}{\partial\btheta}\bXi(s;\btheta)$ is the solution of the equation
\begin{eqnarray*}
{\bf A}{\bf X}+{\bf X}{\bf B}&=&-{\bf D},
\end{eqnarray*}
where
\begin{eqnarray*}
{\bf A}&=&
{\bf J}(s;\btheta)
,\\
{\bf B}&=&
{\bf I}\otimes {\bf U}(s;\btheta)
,\\
{\bf D}&=&
\frac{\partial
{\bf Q}_{-+}(s;\btheta)
}{\partial \btheta}
+
\frac{\partial {\bf Q}_{--}(s;\btheta)}
{ \partial \btheta}
\times\left(
{\bf I}\otimes \bXi(s;\btheta)
\right)
+
\bXi(s;\btheta)
\times
\frac{\partial  {\bf Q}_{++}(s;\btheta)}
{\partial \btheta}
\\
&&
+
\bXi(s;\btheta)
\times
\left\{
\frac{\partial {\bf Q}_{+-}(s;\btheta)}
{ \partial \btheta}
\times\left(
{\bf I}\otimes \bXi(s;\btheta)
\right)
\right\}
.
\end{eqnarray*}
\end{Theorem}

Since the quantities $\bPsi(s;\btheta)$ and $\bXi(s;\btheta)$ can be readily obtained using the existing algorithms (e.g. see Bean, O'Reilly and Taylor~\cite{BOT08}), the above expressions can be used to compute their derivatives $\frac{\partial}{\partial\btheta}$ with respect to $\btheta$.

\section{Stationary quantities: $\frac{\partial
}
{ \partial \btheta}
\bpi(x;\btheta)
$ and $\frac{\partial
}
{ \partial \btheta}
{\bf p}(\btheta)
$
}\label{sec:StaQua}

The stationary distribution of the SFM consists of
the probability density vectors $\bpi(x;\btheta)$, $x>0$, such that \[\bpi(x;\btheta)]_i=\frac{\partial}{\partial x}\lim_{t\to\infty}\mathbb{P}(X(t)\leq x,\varphi(t)=i),\] partitioned according to $\mathcal{S}_+\cup\mathcal{S}_-\cup\mathcal{S}_0$,
\begin{equation*}
\bpi(x;\btheta)=\left[
\begin{array}{ccc}
\bpi_{+}(x;\btheta) & \bpi_{-}(x;\btheta) & \bpi_{0}(x;\btheta)
\end{array}
\right]
,
\end{equation*}
and the probability vector ${\bf p}(\btheta)$, such that \[[{\bf p}(\btheta)]_i=\lim_{t\to\infty}\mathbb{P}(X(t)=0,\varphi(t)=i),\] partitioned according to $\mathcal{S}_{-}\cup \mathcal{S}_{0}$,
\begin{equation*}
{\bf p}(\btheta)=\left[
\begin{array}{cc}
{\bf p}_{-}(\btheta) & {\bf p}_{0}(\btheta)
\end{array}
\right].
\end{equation*}

In order to compute $\frac{\partial
}
{ \partial \btheta}
\bpi(x;\btheta)
$ we may numerically invert the Laplace transform $\mathcal{L}_{\partial\bpi}(v;\btheta)$ given by,
\begin{eqnarray*}
\mathcal{L}_{\partial\bpi}(v;\btheta)
&=&
\int_{x=0}^{\infty}e^{-vx}
\frac{\partial}{\partial\btheta}
\bpi(x;\btheta)dx
=
\frac{\partial}{\partial\btheta}
\int_{x=0}^{\infty}e^{-vx}
\bpi(x;\btheta)dx
=
\frac{\partial}{\partial\btheta}
\mathcal{L}_{\bpi}(v;\btheta),
\end{eqnarray*}
using standard techniques by Abate and Whitt~\cite{abate1995numerical}, Den Iseger~\cite{DenIseger_2006}, and Horv{\'a}th et al.~\cite{horvath2020numerical}.

For notational convenience, let
\begin{eqnarray*}
{\bf T}_{\ominus\ominus}(\btheta)&=&
\left[
\begin{array}{cc}
{\bf T}_{- -}(\btheta)& {\bf T}_{- 0}(\btheta)\\
{\bf T}_{0 -}(\btheta)& {\bf T}_{0 0}(\btheta)
\end{array}
\right]
\end{eqnarray*}
and
\begin{eqnarray*}
{\bf T}_{\ominus +}(\btheta)
=
\left[
\begin{array}{c}
{\bf T}_{- +}(\btheta)\\
{\bf T}_{0 +}(\btheta)
\end{array}
\right]
,
\quad
{\bf T}_{\ominus -}(\btheta)
=
\left[
\begin{array}{c}
{\bf T}_{- -}(\btheta)\\
{\bf T}_{0 -}(\btheta)
\end{array}
\right]
,\quad
{\bf T}_{\pm 0}(\btheta)
=
\left[
\begin{array}{c}
{\bf T}_{+ 0}(\btheta)\\
{\bf T}_{- 0}(\btheta)
\end{array}
\right].
\end{eqnarray*}

Below, we state the result that can be used to evaluate $
\mathcal{L}_{\partial\bpi}(v;\btheta)$ and $\frac{\partial
}
{ \partial \btheta}
{\bf p}(\btheta)
$.
\begin{Theorem}
We have,
\begin{eqnarray*}
\mathcal{L}_{\partial\bpi}(v;\btheta)
&=&
\left[
\begin{array}{ccc}
\mathcal{L}_{\partial\bpi}(v;\btheta)_{+} & \mathcal{L}_{\partial\bpi}(v;\btheta)_{-} &
 \mathcal{L}_{\partial\bpi}(v;\btheta)_{0}
\end{array}
\right]
,
\end{eqnarray*}
with
\begin{eqnarray*}
\lefteqn{
\left[
\begin{array}{cc}
\mathcal{L}_{\partial\bpi}(v;\btheta)_{+} & \mathcal{L}_{\partial\bpi}(v;\btheta)_{-}
\end{array}
\right]
}
\nonumber
\\
&=&
\left(
\frac{\partial}{\partial \btheta}
\left[
\begin{array}{cc}
{\bf p}_{-}(\btheta) & {\bf p}_{0}(\btheta)
\end{array}
\right]
\right)
\left(
{\bf I}
\otimes
{\bf T}_{\ominus +}(\btheta)
\left(-
{\bf K}(\btheta)
\right)^{-1}
\left[
\begin{array}{cc}
({\bf C}_{+}(\btheta))^{-1}
&{\bf \Psi}(\btheta)|{\bf C}_{-}(\btheta)|^{-1}
\end{array}
\right]
\right)
\nonumber
\\
&&
+
\left[
\begin{array}{cc}
{\bf p}_{-}(\btheta) & {\bf p}_{0}(\btheta)
\end{array}
\right]
\Bigg\{
\left(
\frac{\partial}{\partial\btheta}
{\bf T}_{\ominus +}(\btheta)
\right)
\left(
{\bf I}\otimes
\left(-
{\bf K}(\btheta)
\right)^{-1}
\left[
\begin{array}{cc}
({\bf C}_{+}(\btheta))^{-1}
&{\bf \Psi}(\btheta)|{\bf C}_{-}(\btheta)|^{-1}
\end{array}
\right]
\right)
\nonumber
\\
&&
\quad
+
{\bf T}_{\ominus +}(\btheta)
\left(
\frac{\partial}{\partial\btheta}
\left(-
{\bf K}(\btheta)
\right)^{-1}
\right)
\left(
{\bf I}\otimes
\left[
\begin{array}{cc}
({\bf C}_{+}(\btheta))^{-1}
&{\bf \Psi}(\btheta)|{\bf C}_{-}(\btheta)|^{-1}
\end{array}
\right]
\right)
\nonumber
\\
&&
\quad
+
{\bf T}_{\ominus +}(\btheta)
\left(-
{\bf K}(\btheta)
\right)^{-1}
\left[
\begin{array}{cc}
\frac{\partial}{\partial\btheta}({\bf C}_{+}(\btheta))^{-1}
&
\left(
\frac{\partial}{\partial\btheta}{\bf \Psi}(\btheta)
\right)
\left(
{\bf I}\otimes
|{\bf C}_{-}(\btheta)|^{-1}
\right)
+
{\bf \Psi}(\btheta)
\left(
\frac{\partial}{\partial\btheta}
|{\bf C}_{-}(\btheta)|^{-1}
\right)
\end{array}
\right]
\Bigg\}
\nonumber\\
\end{eqnarray*}
and
\begin{eqnarray*}
\lefteqn{
\mathcal{L}_{\partial\bpi}(v;\btheta)_{0}
=
\left[
\begin{array}{cc}
\mathcal{L}_{\partial\bpi}(v;\btheta)_{+} & \mathcal{L}_{\partial\bpi}(v;\btheta)_{-}
\end{array}
\right]
\left(
{\bf I}\otimes
{\bf T}_{\pm 0}(\btheta)
(-{\bf T}_{00}(\btheta))^{-1}
\right)
}
\nonumber
\\
&&
+
\left[
\begin{array}{cc}
\mathcal{L}_{\bpi}(v;\btheta)_{+} & \mathcal{L}_{\bpi}(v;\btheta)_{-}
\end{array}
\right]
\Bigg\{
\left(
\frac{\partial}{\partial\btheta}
{\bf T}_{\pm 0}(\btheta)
\right)
\left(
{\bf I}\otimes
(-{\bf T}_{00}(\btheta))^{-1}
\right)
 +
{\bf T}_{\pm 0}(\btheta)
\left(
\frac{\partial}{\partial\btheta}
(-{\bf T}_{00}(\btheta))^{-1}
\right)
\Bigg\},
\nonumber\\
\end{eqnarray*}
where
\begin{eqnarray*}
\frac{\partial}{\partial\btheta}(-{\bf K}(\btheta))^{-1}
&=&
(-{\bf K}(\btheta))^{-1}
\frac{\partial {\bf K}(\btheta)}{\partial \btheta}
\left(
{\bf I}
\otimes (-{\bf K}(\btheta))^{-1}
\right),
\\
\frac{\partial}{\partial\btheta}({\bf C}_{+}(\btheta))^{-1}
&=&
-({\bf C}_{+}(\btheta))^{-1}
\frac{\partial {\bf C}_{+}(\btheta)}{\partial \btheta}
\left(
{\bf I}
\otimes ({\bf C}_{+}(\btheta))^{-1}
\right),
\\
\frac{\partial}{\partial\btheta}
|{\bf C}_{-}(\btheta)|^{-1}
&=&
-(|{\bf C}_{-}(\btheta)|)^{-1}
\frac{\partial |{\bf C}_{-}(\btheta)|}{\partial \btheta}
\left(
{\bf I}
\otimes (|{\bf C}_{-}(\btheta)|)^{-1}
\right)
,\\
\frac{\partial}{\partial\btheta}
(-{\bf T}_{00}(\btheta))^{-1}
&=&
(-{\bf T}_{00}(\btheta))^{-1}
\frac{\partial {\bf T}_{00}(\btheta)}{\partial \btheta}
\left(
{\bf I}
\otimes (-{{\bf T}_{00}(\btheta)})^{-1}
\right)
.
\end{eqnarray*}

Next,
\begin{eqnarray*}
\lefteqn{
\frac{\partial
}
{ \partial \btheta}
\left[
\begin{array}{cc}
{\bf p}_{-}(\btheta) & {\bf p}_{0}(\btheta)
\end{array}
\right]
=
\left(
\frac{\partial
}
{ \partial \btheta}
\alpha(\btheta)
\right)
\left(
{\bf I}\otimes
\left[
\begin{array}{cc}
\bxi(\btheta)&{\bf 0}
\end{array}
\right]
\left(-
{\bf T}_{\ominus\ominus}(\btheta)
\right)^{-1}
\right)
}
\nonumber
\\
&&
+
\alpha(\btheta)
\Bigg\{
\left(
\frac{\partial}{\partial\btheta}
\left[
\begin{array}{cc}
\bxi(\btheta)&{\bf 0}
\end{array}
\right]
\right)
\left(
{\bf I}\otimes
\left(-
{\bf T}_{\ominus\ominus}(\btheta)
\right)^{-1}
\right)
+
\left[
\begin{array}{cc}
\bxi(\btheta)&{\bf 0}
\end{array}
\right]
\left(
\frac{\partial}{\partial\btheta}
\left(-
{\bf T}_{\ominus\ominus}(\btheta)
\right)^{-1}
\right)
\Bigg\}
,
\end{eqnarray*}
where $\alpha(\btheta)$ is a normalising constant, given by
\begin{align}\label{alpha}
\alpha(\btheta)&=
\Bigg\{
\left[
\begin{array}{cc}
\bxi(\btheta)&{\bf 0}\\
\end{array}
\right]
\left(-
{\bf T}_{\ominus\ominus}(\btheta)
\right)^{-1}
\Bigg(
{\bf 1}
+
{\bf T}_{\ominus +}(\btheta)
(-{\bf K}(\btheta))^{-1}
\nonumber\\
&
\quad \quad
\times
\left[
\begin{array}{cc}
({\bf C}_{+}(\btheta))^{-1}
&{\bf \Psi}(\btheta)|{\bf C}_{-}(\btheta)|^{-1}
\end{array}
\right]
\left(
{\bf 1}
+
{\bf T}_{\pm 0}(\btheta)
(-{\bf T}_{00}(\btheta))^{-1}
{\bf 1}
\right)
\Bigg)\Bigg\}^{-1}
,
\end{align}
and $\bxi(\btheta)$ is the unique solution of the set of equations
\begin{eqnarray}
\left[
\begin{array}{cc}
\bxi(\btheta)&{\bf 0}
\end{array}
\right]
\left(-
{\bf T}_{\ominus\ominus}(\btheta)
\right)^{-1}
{\bf T}_{\ominus +}(\btheta)
{\bf \Psi}(\btheta)
&=&\bxi(\btheta),\label{ksi1}\\
\bxi(\btheta){\bf 1}&=&1\label{ksi2}.
\end{eqnarray}

Further, $\frac{\partial
}
{ \partial \btheta}
\alpha(\btheta)$ is given by
\begin{eqnarray*}
\lefteqn{
\frac{\partial
}
{ \partial \btheta}
\alpha(\btheta)
=
\Bigg\{
\left(
\quad
\left(
\frac{\partial}
{ \partial \btheta}
\left[
\begin{array}{cc}
\bxi(\btheta)&{\bf 0}\\
\end{array}
\right]
\right)
\left(
{\bf I}\otimes
\left(-
{\bf T}_{\ominus\ominus}(\btheta)
\right)^{-1}
\right)
+
\left(
\left[
\begin{array}{cc}
\bxi(\btheta)&{\bf 0}\\
\end{array}
\right]
\right)
\left(
\frac{\partial}
{ \partial \btheta}
\left(-
{\bf T}_{\ominus\ominus}(\btheta)
\right)^{-1}
\right)
\quad
\right)
}
\nonumber\\
&&
\quad
\times
{\bf I}\otimes
\Bigg(
{\bf 1}
+
{\bf T}_{\ominus +}(\btheta)
(-{\bf K}(\btheta))^{-1}
\left[
\begin{array}{cc}
({\bf C}_{+}(\btheta))^{-1}
&{\bf \Psi}(\btheta)|{\bf C}_{-}(\btheta)|^{-1}
\end{array}
\right]
\left(
{\bf 1}
+
{\bf T}_{\pm 0}(\btheta)
(-{\bf T}_{00}(\btheta))^{-1}
{\bf 1}
\right)
\Bigg)\Bigg\}^{-1}
\nonumber\\
&&
+
\Bigg\{
\left(
\left[
\begin{array}{cc}
\bxi(\btheta)&{\bf 0}\\
\end{array}
\right]
\left(-
{\bf T}_{\ominus\ominus}(\btheta)
\right)^{-1}
\right)
\left(
\quad
\left(
\frac{\partial}{\partial\btheta}
{\bf T}_{\ominus +}(\btheta)
\right)
{\bf I}\otimes
(-{\bf K}(\btheta))^{-1}
+
{\bf T}_{\ominus +}(\btheta)
\left(
\frac{\partial}{\partial\btheta}
(-{\bf K}(\btheta))^{-1}
\right)
\quad
\right)
\nonumber\\
&&
\quad
\times
{\bf I}\otimes
\Bigg(
\left[
\begin{array}{cc}
({\bf C}_{+}(\btheta))^{-1}
&{\bf \Psi}(\btheta)|{\bf C}_{-}(\btheta)|^{-1}
\end{array}
\right]
\left(
{\bf 1}
+
{\bf T}_{\pm 0}(\btheta)
(-{\bf T}_{00}(\btheta))^{-1}
{\bf 1}
\right)
\Bigg)\Bigg\}^{-1}
\nonumber\\
&&
+
\Bigg\{
\left[
\begin{array}{cc}
\bxi(\btheta)&{\bf 0}\\
\end{array}
\right]
\left(-
{\bf T}_{\ominus\ominus}(\btheta)
\right)^{-1}
{\bf T}_{\ominus +}(\btheta)
(-{\bf K}(\btheta))^{-1}
\nonumber\\
&&
\quad \quad
\times
\Bigg(
\quad
\left[
\begin{array}{cc}
\frac{\partial}{\partial\btheta}({\bf C}_{+}(\btheta))^{-1}
&
\left(
\frac{\partial}{\partial\btheta}{\bf \Psi}(\btheta)
\right)
\left(
{\bf I}\otimes
|{\bf C}_{-}(\btheta)|^{-1}
\right)
+
{\bf \Psi}(\btheta)
\frac{\partial}{\partial\btheta}
|{\bf C}_{-}(\btheta)|^{-1}
\end{array}
\right]
\nonumber\\
&&
\quad \quad \quad \quad \quad
\times
\left(
{\bf I}\otimes
\left(
{\bf 1}
+
{\bf T}_{\pm 0}(\btheta)
(-{\bf T}_{00}(\btheta))^{-1}
{\bf 1}
\right)
\right)
\nonumber\\
&&
\quad \quad \quad \quad
+
\left[
\begin{array}{cc}
({\bf C}_{+}(\btheta))^{-1}
&{\bf \Psi}(\btheta)|{\bf C}_{-}(\btheta)|^{-1}
\end{array}
\right]
\nonumber\\
&&
\quad \quad \quad \quad \quad
\times
\left(
\left(
\frac{\partial}{\partial\btheta}
{\bf T}_{\pm 0}(\btheta)
\right)
\left(
{\bf I}\otimes
(-{\bf T}_{00}(\btheta))^{-1}
{\bf 1}
\right)
+
{\bf T}_{\pm 0}(\btheta)
\left(
\frac{\partial}{\partial\btheta}
(-{\bf T}_{00}(\btheta))^{-1}
\right)
\left(
{\bf I}\otimes{\bf 1}
\right)
\right)
\quad
\Bigg)\Bigg\}^{-1}
,
\nonumber\\
\end{eqnarray*}
with
\begin{eqnarray*}
\frac{\partial}{\partial\btheta}
\left(-
{\bf T}_{\ominus\ominus}(\btheta)
\right)^{-1}
&=&
\left(-
{\bf T}_{\ominus\ominus}(\btheta)
\right)^{-1}
\left(
\frac{\partial}{\partial \btheta}
{\bf T}_{\ominus\ominus}(\btheta)
\right)
\left(
{\bf I}
\otimes
\left(-
{\bf T}_{\ominus\ominus}(\btheta)
\right)^{-1}
\right).
\end{eqnarray*}

Finally, $\frac{\partial}{\partial\btheta}\bxi(\btheta)$ is the solution of the set of equations
\begin{eqnarray*}
\lefteqn{
\frac{\partial}{\partial\btheta}
\bxi(\btheta)
=
\left(
\frac{\partial}{\partial\btheta}\bxi(\btheta)
\right)
\left(
{\bf I}\otimes
\left[
\begin{array}{cc}
{\bf I}_-&{\bf 0}_{-0}
\end{array}
\right]
\left(-
{\bf T}_{\ominus\ominus}(\btheta)
\right)^{-1}
{\bf T}_{\ominus +}(\btheta)
{\bf \Psi}(\btheta)
\right)
}
\nonumber\\
&&
+
\bxi(\btheta)
\left[
\begin{array}{cc}
{\bf I}_-&{\bf 0}_{-0}
\end{array}
\right]
\Bigg\{
\left(
\frac{\partial}{\partial\btheta}
\left(-
{\bf T}_{\ominus\ominus}(\btheta)
\right)^{-1}
\right)
\left(
{\bf I}\otimes
{\bf T}_{\ominus +}(\btheta)
{\bf \Psi}(\btheta)
\right)
\nonumber
\\
&&
+
\left(-
{\bf T}_{\ominus\ominus}(\btheta)
\right)^{-1}
\left(
\frac{\partial}{\partial\btheta}
{\bf T}_{\ominus +}(\btheta)
\right)
\left(
{\bf I}\otimes
{\bf \Psi}(\btheta)
\right)
+
\left(-
{\bf T}_{\ominus\ominus}(\btheta)
\right)^{-1}
{\bf T}_{\ominus +}(\btheta)
\left(
\frac{\partial}{\partial\btheta}
{\bf \Psi}(\btheta)
\right)
\Bigg\}
\end{eqnarray*}
and
\begin{eqnarray*}
\frac{\partial}{\partial\btheta}\bxi(\btheta)
{\bf 1}&=&0.
\end{eqnarray*}
\end{Theorem}
{\bf Proof:} By O'Reilly and Scheinhardt~\cite[Lemma 1]{OW17},
\begin{eqnarray}
\left[
\begin{array}{cc}
{\bf p}_{-}(\btheta) & {\bf p}_{0}(\btheta)
\end{array}
\right]&=&
\alpha(\btheta)
\left[
\begin{array}{cc}
\bxi(\btheta)&{\bf 0}
\end{array}
\right]
\left(-
{\bf T}_{\ominus\ominus}(\btheta)
\right)^{-1},\label{mass}\\
\left[
\begin{array}{cc}
\bpi_{+}(x;\btheta) & \bpi_{-}(x;\btheta)
\end{array}
\right]
&=&
\left[
\begin{array}{cc}
{\bf p}_{-}(\btheta) & {\bf p}_{0}(\btheta)
\end{array}
\right]
{\bf T}_{\ominus +}(\btheta)
e^{{\bf K}\left(\btheta\right)x}
\nonumber \\
&&
\quad \quad
\times
\left[
\begin{array}{cc}
({\bf C}_{+}(\btheta))^{-1}
&{\bf \Psi}(\btheta)|{\bf C}_{-}(\btheta)|^{-1}
\end{array}
\right],\\
\bpi_{0}(x;\btheta)&=&
\left[
\begin{array}{cc}
\bpi_{+}(x;\btheta) & \bpi_{-}(x;\btheta)
\end{array}
\right]
{\bf T}_{\pm 0}(\btheta)
(-{\bf T}_{00}(\btheta))^{-1},
\label{bpi0}
\end{eqnarray}
where $\alpha(\btheta)$ is given by~\eqref{alpha} and $\bxi(\btheta)$ is the unique solution of~\eqref{ksi1}-\eqref{ksi2}.

Consider the Laplace transform
\begin{eqnarray*}
\mathcal{L}_{\bpi}(v;\btheta)
&=&
\int_{x=0}^{\infty}e^{-vx}\bpi(x;\btheta)dx
=
\left[
\begin{array}{ccc}
\mathcal{L}_{\bpi}(v;\btheta)_{+} & \mathcal{L}_{\bpi}(v;\btheta)_{-} &
 \mathcal{L}_{\bpi}(v;\btheta)_{0}
\end{array}
\right]
\end{eqnarray*}
which, by above, is given by
\begin{eqnarray*}
\left[
\begin{array}{cc}
\mathcal{L}_{\bpi}(v;\btheta)_{+} & \mathcal{L}_{\bpi}(v;\btheta)_{-}
\end{array}
\right]
&=&
\left[
\begin{array}{cc}
{\bf p}_{-}(\btheta) & {\bf p}_{0}(\btheta)
\end{array}
\right]
{\bf T}_{\ominus +}(\btheta)
\left(-
{\bf K}(\btheta)
\right)^{-1}
\nonumber \\
&&
\quad \quad
\times
\left[
\begin{array}{cc}
({\bf C}_{+}(\btheta))^{-1}
&{\bf \Psi}(\btheta)|{\bf C}_{-}(\btheta)|^{-1}
\end{array}
\right],\\
\mathcal{L}_{\bpi}(v;\btheta)_{0}&=&
\left[
\begin{array}{cc}
\mathcal{L}_{\bpi}(v;\btheta)_{+} & \mathcal{L}_{\bpi}(v;\btheta)_{-}
\end{array}
\right]
{\bf T}_{\pm 0}(\btheta)
(-{\bf T}_{00}(\btheta))^{-1}.
\end{eqnarray*}
The result then follows by the matrix derivative calculus~\eqref{eq:mproduct}-\eqref{eq:minverse}. with
\begin{align*}
\lefteqn{
\left[
\begin{array}{cc}
\mathcal{L}_{\partial\bpi}(v;\btheta)_{+} & \mathcal{L}_{\partial\bpi}(v;\btheta)_{-}
\end{array}
\right]
}
\nonumber
\\
=&
\Bigg(
\frac{\partial}{\partial \btheta}
\left[
\begin{array}{cc}
{\bf p}_{-}(\btheta) & {\bf p}_{0}(\btheta)
\end{array}
\right]
\Bigg)
\Bigg(
{\bf I}
\otimes
{\bf T}_{\ominus +}(\btheta)
\left(-
{\bf K}(\btheta)
\right)^{-1}
\left[
\begin{array}{cc}
({\bf C}_{+}(\btheta))^{-1}
&{\bf \Psi}(\btheta)|{\bf C}_{-}(\btheta)|^{-1}
\end{array}
\right]
\Bigg)
\nonumber
\\
&
+
\left[
\begin{array}{cc}
{\bf p}_{-}(\btheta) & {\bf p}_{0}(\btheta)
\end{array}
\right]
\Bigg(
\frac{\partial}{\partial\btheta}
{\bf T}_{\ominus +}(\btheta)
\left(-
{\bf K}(\btheta)
\right)^{-1}
\left[
\begin{array}{cc}
({\bf C}_{+}(\btheta))^{-1}
&{\bf \Psi}(\btheta)|{\bf C}_{-}(\btheta)|^{-1}
\end{array}
\right]
\Bigg)
\nonumber
\\
=&
\Bigg(
\frac{\partial}{\partial \btheta}
\left[
\begin{array}{cc}
{\bf p}_{-}(\btheta) & {\bf p}_{0}(\btheta)
\end{array}
\right]
\Bigg)
\Bigg(
{\bf I}
\otimes
{\bf T}_{\ominus +}(\btheta)
\left(-
{\bf K}(\btheta)
\right)^{-1}
\left[
\begin{array}{cc}
({\bf C}_{+}(\btheta))^{-1}
&{\bf \Psi}(\btheta)|{\bf C}_{-}(\btheta)|^{-1}
\end{array}
\right]
\Bigg)
\nonumber
\\
&
+
\left[
\begin{array}{cc}
{\bf p}_{-}(\btheta) & {\bf p}_{0}(\btheta)
\end{array}
\right]
\Bigg(
\frac{\partial}{\partial\btheta}
{\bf T}_{\ominus +}(\btheta)
\Bigg)
\Bigg(
{\bf I}\otimes
\left(-
{\bf K}(\btheta)
\right)^{-1}
\left[
\begin{array}{cc}
({\bf C}_{+}(\btheta))^{-1}
&{\bf \Psi}(\btheta)|{\bf C}_{-}(\btheta)|^{-1}
\end{array}
\right]
\Bigg)
\nonumber
\\
&
+
\left[
\begin{array}{cc}
{\bf p}_{-}(\btheta) & {\bf p}_{0}(\btheta)
\end{array}
\right]
{\bf T}_{\ominus +}(\btheta)
\Bigg(
\frac{\partial}{\partial\btheta}
\left(-
{\bf K}(\btheta)
\right)^{-1}
\Bigg)
\Bigg(
{\bf I}\otimes
\left[
\begin{array}{cc}
({\bf C}_{+}(\btheta))^{-1}
&{\bf \Psi}(\btheta)|{\bf C}_{-}(\btheta)|^{-1}
\end{array}
\right]
\Bigg)
\nonumber
\\
&
+
\left[
\begin{array}{cc}
{\bf p}_{-}(\btheta) & {\bf p}_{0}(\btheta)
\end{array}
\right]
{\bf T}_{\ominus +}(\btheta)
\left(-
{\bf K}(\btheta)
\right)^{-1}
\nonumber\\
&
\quad\quad
\times
\left[
\begin{array}{cc}
\frac{\partial}{\partial\btheta}({\bf C}_{+}(\btheta))^{-1}
&
\frac{\partial}{\partial\btheta}{\bf \Psi}(\btheta)
\left(
{\bf I}\otimes
|{\bf C}_{-}(\btheta)|^{-1}
\right)
+
{\bf \Psi}(\btheta)
\frac{\partial}{\partial\btheta}
|{\bf C}_{-}(\btheta)|^{-1}
\end{array}
\right]
\nonumber
\end{align*}
and the remaining quantities are evaluated in a similar manner. \rule{9pt}{9pt}

We may also write the expressions for the probability density vectors $\bpi(x)$ directly as follows.
\begin{Corollary}
For $x>0$, we have,
\begin{eqnarray*}
\lefteqn{
\left[
\begin{array}{cc}
\frac{\partial}{\partial\btheta}
\bpi_{+}(\btheta) &
\frac{\partial}{\partial\btheta}
\bpi_{-}(\btheta)
\end{array}
\right]
}
\nonumber
\\
&=&
\Bigg(
\frac{\partial}{\partial \btheta}
\left[
\begin{array}{cc}
{\bf p}_{-}(\btheta) & {\bf p}_{0}(\btheta)
\end{array}
\right]
\Bigg)
\Bigg(
{\bf I}
\otimes
{\bf T}_{\ominus +}(\btheta)
e^{{\bf K}\left(\btheta\right)x}
\left[
\begin{array}{cc}
({\bf C}_{+}(\btheta))^{-1}
&{\bf \Psi}(\btheta)|{\bf C}_{-}(\btheta)|^{-1}
\end{array}
\right]
\Bigg)
\nonumber
\\
&&
+
\left[
\begin{array}{cc}
{\bf p}_{-}(\btheta) & {\bf p}_{0}(\btheta)
\end{array}
\right]
\Bigg\{
\quad
\Bigg(
\frac{\partial}{\partial\btheta}
{\bf T}_{\ominus +}(\btheta)
\Bigg)
\Bigg(
e^{{\bf K}\left(\btheta\right)x}
\left[
\begin{array}{cc}
({\bf C}_{+}(\btheta))^{-1}
&{\bf \Psi}(\btheta)|{\bf C}_{-}(\btheta)|^{-1}
\end{array}
\right]
\Bigg)
\nonumber
\\
&&
\quad
+
{\bf T}_{\ominus +}(\btheta)
\Bigg(
\frac{\partial}{\partial\btheta}
e^{{\bf K}\left(\btheta\right)x}
\Bigg)
\Bigg(
{\bf I}\otimes
\left[
\begin{array}{cc}
({\bf C}_{+}(\btheta))^{-1}
&{\bf \Psi}(\btheta)|{\bf C}_{-}(\btheta)|^{-1}
\end{array}
\right]
\Bigg)
\nonumber
\\
&&
\quad
+
{\bf T}_{\ominus +}(\btheta)
e^{{\bf K}\left(\btheta\right)x}
\left[
\begin{array}{cc}
\frac{\partial}{\partial\btheta}({\bf C}_{+}(\btheta))^{-1}
&
\frac{\partial}{\partial\btheta}{\bf \Psi}(\btheta)
\left(
{\bf I}\otimes
|{\bf C}_{-}(\btheta)|^{-1}
\right)
+
{\bf \Psi}(\btheta)
\frac{\partial}{\partial\btheta}
|{\bf C}_{-}(\btheta)|^{-1}
\end{array}
\right]
\quad
\Bigg\}
,
\end{eqnarray*}
where
\begin{eqnarray*}
\int_{x=0}^{\infty}e^{-vx}
\frac{\partial
}
{ \partial \btheta}
e^{
{\bf K}(\btheta)x
}dx
&=&
({\bf K}(\btheta)-v{\bf I})^{-1}
\times
\frac{\partial {\bf K}(\btheta)}{\partial \btheta}
\times
\left(
{\bf I}
\otimes ({\bf K}(\btheta))^{-1}
\right).
\end{eqnarray*}
\end{Corollary}

\section{Transient quantities: $\frac{\partial
}
{ \partial \btheta}
{\bf f}(x,t;\btheta)
$ and $\frac{\partial}{\partial\btheta}
{\bf p}(t;\btheta)$}\label{sec:TraQua}

Suppose that the SFM starts from level $X(0)= z>0$ and does so in some phase $\varphi(0)\in\mathcal{S}_-$ according to distribution ${\bf g}_-=[g_i]_{i\in\mathcal{S}_-}$, where $g_i=\mathbb{P}(\varphi(0)=i)$. Consider the probability density vectors ${\bf f}(x,t;\btheta)$, $x>0$, such that
\begin{equation*}
[{\bf f}(x,t;\btheta)]_i=\frac{\partial}{\partial x}\mathbb{P}(X(t)\leq x,\varphi(t)=i),
\end{equation*}
partitioned according to $\mathcal{S}_+\cup\mathcal{S}_-\cup\mathcal{S}_0$,
\begin{equation}
\label{eq:f}
{\bf f}(x,t;\btheta)=\left[
\begin{array}{ccc}
{\bf f}_{+}(x,t;\btheta) & {\bf f}_{-}(x,t;\btheta) & {\bf f}_{0}(x,t;\btheta)
\end{array}
\right]
,
\end{equation}
and the probability vector ${\bf p}(t;\btheta)$, such that $[{\bf p}(t;\btheta)]_i=\mathbb{P}(X(t)=0,\varphi(t)=i)$, partitioned according to $\mathcal{S}_{-}\cup \mathcal{S}_{0}$,
\begin{equation*}
{\bf p}(t;\btheta)=\left[
\begin{array}{cc}
{\bf p}_{-}(t;\btheta) & {\bf p}_{0}(t;\btheta)
\end{array}
\right].
\end{equation*}

Also, consider the corresponding Laplace transforms,
\begin{eqnarray*}
\widetilde {\bf f}(x,s;\btheta)
&=&\int_{t=0}^{\infty}
e^{-st}{\bf f}(x,t;\btheta)dt
,\\
\widetilde {\bf p}(s;\btheta)
&=&\int_{t=0}^{\infty}
e^{-st}{\bf p}(t;\btheta)dt,
\end{eqnarray*}
and note the following useful equivalences,
\begin{eqnarray*}
\mathcal{L}_{\partial{\bf f}}(x,s;\btheta)
&=&
\int_{t=0}^{\infty}e^{-st}
\frac{\partial}{\partial\btheta}
{\bf f}(x,t;\btheta)dt
=
\frac{\partial}{\partial\btheta}
\int_{t=0}^{\infty}e^{-st}
{\bf f}(x,t;\btheta)dt
=
\frac{\partial}{\partial\btheta}
\widetilde {\bf f}(x,s;\btheta),
\\
\mathcal{L}_{\partial{\bf p}}(s;\btheta)
&=&
\int_{t=0}^{\infty}e^{-st}
\frac{\partial}{\partial\btheta}
{\bf p}(t;\btheta)dt
=
\frac{\partial}{\partial\btheta}
\int_{t=0}^{\infty}e^{-st}
{\bf p}(t;\btheta)dt
=
\frac{\partial}{\partial\btheta}
\widetilde {\bf p}(s;\btheta),
\end{eqnarray*}
so that we may evaluate $\frac{\partial}{\partial\btheta}
{\bf f}(x,t;\btheta)$ and $\frac{\partial}{\partial\btheta}
{\bf p}(t;\btheta)$ by numerically inverting the above Laplace transforms $\mathcal{L}_{\partial{\bf f}}(x,s;\btheta)$ and $\mathcal{L}_{\partial{\bf p}}(s;\btheta)$ respectively, by applying standard  techniques~\cite{abate1995numerical,DenIseger_2006,horvath2020numerical}.

In the transient analysis, we also need to consider the distributions of the first hitting times. We address this in the lemma below. For $0<x<y$, denote
\begin{eqnarray*}
{\bf G}^{(x,y)}(s;\btheta)
=
	\left[
	\begin{array}{cc}
		{\bf O}_{++}&
{\bf G}_{+-}^{(x,y)}(s;\btheta)
\\
		{\bf O}_{-+}&
  {\bf G}_{--}^{(x,y)}(s;\btheta)
	\end{array}
	\right],
 \quad
 {\bf H}^{(x,y)}(s;\btheta)
=
	\left[
	\begin{array}{cc}
 {\bf H}_{++}^{(x,y)}(s;\btheta)&{\bf O}_{+-}\\
  {\bf H}_{-+}^{(x,y)}(s;\btheta)&{\bf O}_{--}
	\end{array}
	\right],
\end{eqnarray*}
where $[{\bf G}^{(x,y)}(s;\btheta)]_{ij}
=\mathbb{E}\left( e^{-s\delta(0)} I(\delta(0)<\delta(y),\varphi(\delta(0))=j)\ |\ X(0)=x,\varphi(0)=i\right)$ is the Laplace transform of the distribution of the time to hit level zero before hitting level $y$ and do so in phase $j$, given start from level $x$ in phase $i$, and $[{\bf H}^{(x,y)}(s;\btheta)]_{ij}
=\mathbb{E}\left( e^{-s\delta(y)} I(\delta(y)<\delta(0),\varphi(\delta(y))=j)\ |\ X(0)=x,\varphi(0)=i \right)$ is the Laplace transform of the distribution of the time to hit level $y$ before hitting level zero and do so in phase $j$, given start from level $x$ in phase $i$.

Also, let
\begin{eqnarray*}
{\bf G}^{(x)}(s;\btheta)
=
	\left[
	\begin{array}{cc}
		{\bf O}_{++}&
{\bf G}_{+-}^{(x)}(s;\btheta)
\\
		{\bf O}_{-+}&
  {\bf G}_{--}^{(x)}(s;\btheta)
	\end{array}
	\right],
 \quad
 {\bf H}^{(x)}(s;\btheta)
=
	\left[
	\begin{array}{cc}
 {\bf H}_{++}^{(x)}(s;\btheta)&{\bf O}_{+-}\\
  {\bf H}_{-+}^{(x)}(s;\btheta)&{\bf O}_{--}
	\end{array}
	\right],
\end{eqnarray*}
where $[{\bf G}^{(x)}(s;\btheta)]_{ij}
=\mathbb{E}\left( e^{-s\delta(0)} I(\varphi(\delta(0))=j)\ |\ X(0)=x,\varphi(0)=i\right)$ is the Laplace transform of the distribution of the time to hit level zero and do so in phase $j$, given start from level $x$ in phase $i$, and, the SFM $\{(\widetilde{X}(t),\varphi(t)):t\geq 0\}$ defined in \eqref{SFM2} (without the zero boundary), $[{\bf H}^{(x)}(s;\btheta)]_{ij}
=\mathbb{E}\left( e^{-s\widetilde\delta(y)} I(\varphi(\widetilde\delta(y))=j)\ |\ \widetilde{X}(0)=x,\varphi(0)=i\right)$ is the Laplace transform of the distribution of the time to hit level $y$ and do so in phase $j$, given start from level $x$ in phase $i$.

\begin{Lemma}\label{lem:GHder}
For any $0<x<y$,
\begin{eqnarray*}
\lefteqn{
\left[
	\begin{array}{cc}
\frac{\partial}{\partial\btheta}{\bf G}^{(x,y)}(s;\btheta)&
 \frac{\partial}{\partial\btheta}{\bf H}^{(x,y)}(s;\btheta)
	\end{array}
	\right]
 }
 \nonumber
 \\
	&=&
\Bigg(
 \frac{\partial}{\partial\btheta}
 	\left[
	\begin{array}{cccc}
 {\bf 0}_{++}&
\bPsi(s;\btheta)e^{{\bf D}(s;\btheta)x}&
e^{{\bf U}(s;\btheta)(y-x)}&
{\bf 0}_{+-}
\\
{\bf 0}_{-+}&
e^{{\bf D}(s;\btheta)x}&
\bXi(s;\btheta)e^{{\bf U}(s;\btheta)(y-x)}&
{\bf 0}_{--}
	\end{array}
	\right]
\Bigg)
\nonumber
\\
&&
\quad
\times
\left(
{\bf I}\otimes
 \left[
	\begin{array}{cccc}
{\bf I}&{\bf 0}&
e^{{\bf U}(s;\btheta)y}&
{\bf 0}_{+-}
\\
{\bf 0}&{\bf I}
&
\bXi(s)e^{{\bf U}(s;\btheta)y}
&{\bf 0}_{--}
\\
{\bf 0}_{++}&
\bPsi(s;\btheta)e^{{\bf D}(s;\btheta)y}&
{\bf I}&{\bf 0}
\\
{\bf 0}_{-+}&e^{{\bf D}(s;\btheta)y}&
{\bf 0}&{\bf I}
	\end{array}
	\right]
\right)
	\nonumber\\
 &&
 +
 	\left[
	\begin{array}{cccc}
 {\bf 0}_{++}&
\bPsi(s;\btheta)e^{{\bf D}(s;\btheta)x}&
e^{{\bf U}(s;\btheta)(y-x)}&
{\bf 0}_{+-}
\\
{\bf 0}_{-+}&
e^{{\bf D}(s;\btheta)x}&
\bXi(s;\btheta)e^{{\bf U}(s;\btheta)(y-x)}&
{\bf 0}_{--}
	\end{array}
	\right]
 \nonumber
 \\
 &&
 \quad
 \times
 \left(
 \frac{\partial}{\partial\btheta}
 \left[
	\begin{array}{cccc}
{\bf I}&{\bf 0}&
e^{{\bf U}(s;\btheta)y}&
{\bf 0}_{+-}
\\
{\bf 0}&{\bf I}
&
\bXi(s)e^{{\bf U}(s;\btheta)y}
&{\bf 0}_{--}
\\
{\bf 0}_{++}&
\bPsi(s;\btheta)e^{{\bf D}(s;\btheta)y}&
{\bf I}&{\bf 0}
\\
{\bf 0}_{-+}&e^{{\bf D}(s;\btheta)y}&
{\bf 0}&{\bf I}
	\end{array}
	\right]^{-1}
 \right)
 ,
\end{eqnarray*}
where
\begin{eqnarray*}
\frac{\partial
}
{ \partial \btheta}
\left(
\bPsi(s;\btheta)e^{{\bf D}(s;\btheta)y}
\right)
&=&
\frac{\partial \bPsi(s;\btheta)
}
{ \partial \btheta}
\left(
{\bf I}\otimes e^{{\bf D}(s;\btheta)y}
\right)
+
\bPsi(s;\btheta)
\left(
\frac{\partial
} { \partial \btheta}
e^{{\bf D}(s;\btheta)y}
\right)
,
\\
\frac{\partial
}
{ \partial \btheta}
\left(
\bXi(s;\btheta)e^{{\bf U}(s;\btheta)y}
\right)
&=&
\frac{\partial \bXi(s;\btheta)
}
{ \partial \btheta}
\left(
{\bf I}\otimes e^{{\bf U}(s;\btheta)y}
\right)
+
\bXi(s;\btheta)
\left(
\frac{\partial
} { \partial \btheta}
e^{{\bf U}(s;\btheta)y}
\right)
,
\\
\int_{x=0}^{\infty}e^{-vx}
\frac{\partial
}
{ \partial \btheta}
e^{
{\bf D}(s,\btheta)x
}dx
&=&
\int_{y=0}^{\infty}e^{-vy}
\frac{\partial
}
{ \partial \btheta}
e^{
{\bf D}(s,\btheta)y
}dy
\nonumber\\
&=&
({\bf D}(s,\btheta)-v{\bf I})^{-1}
\times
\frac{\partial {\bf D}(s,\btheta)}{\partial \btheta}
\times
\left(
{\bf I}
\otimes ({\bf D}(s,\btheta)-v{\bf I})^{-1}
\right)
\\
\int_{y=0}^{\infty}e^{-vy}
\frac{\partial
}
{ \partial \btheta}
e^{
{\bf U}(s,\btheta)y
}
dy
&=&
\int_{y=x}^{\infty}e^{-v(y-x)}
\frac{\partial
}
{ \partial \btheta}
e^{
{\bf U}(s,\btheta)(y-x)
}
dy
\nonumber
\\
&=&
({\bf U}(s,\btheta)-v{\bf I})^{-1}
\times
\frac{\partial {\bf U}(s,\btheta)}{\partial \btheta}
\times
\left(
{\bf I}
\otimes ({\bf U}(s,\btheta)-v{\bf I})^{-1}
\right)
,
\end{eqnarray*}
with ${\bf G}^{(x,x)}(s;\btheta)=\lim_{y\downarrow x}{\bf G}^{(x,y)}(s;\btheta)$ and ${\bf H}^{(x,x)}(s;\btheta)=\lim_{y\downarrow x}{\bf H}^{(x,y)}(s;\btheta)$.

\end{Lemma}
{\bf Proof:} By Bean, O'Reilly and Taylor~\cite{BOT05}, for any $0<x<y$,
\begin{eqnarray*}
\lefteqn{
	\left[
	\begin{array}{cccc}
		{\bf O}_{++}&{\bf G}_{+-}^{(x,y)}(s)&
  {\bf H}_{++}^{(x,y)}(s)&{\bf O}_{+-}\\
		{\bf O}_{-+}&{\bf G}_{--}^{(x,y)}(s)&{\bf H}_{-+}^{(x,y)}(s)&{\bf O}_{--}
	\end{array}
	\right]
 =
 \left[
	\begin{array}{cc}
{\bf G}^{(x)}(s)&
{\bf H}^{(y-x)}(s)
	\end{array}
	\right]
	\left[
	\begin{array}{cc}
		{\bf I}&
  {\bf H}^{(y)}(s)
  \\
{\bf G}^{(y)}(s)
  &{\bf I}
	\end{array}
	\right]^{-1}
}
 \\
 &=&
	\left[
	\begin{array}{cccc}
 {\bf 0}_{++}&
\bPsi(s;\btheta)e^{{\bf D}(s;\btheta)x}&
e^{{\bf U}(s;\btheta)(y-x)}&
{\bf 0}_{+-}
\\
{\bf 0}_{-+}&
e^{{\bf D}(s;\btheta)x}&
\bXi(s;\btheta)e^{{\bf U}(s;\btheta)(y-x)}&
{\bf 0}_{--}
	\end{array}
	\right]
 \left[
	\begin{array}{cccc}
{\bf I}&{\bf 0}&
e^{{\bf U}(s;\btheta)y}&
{\bf 0}_{+-}
\\
{\bf 0}&{\bf I}
&
\bXi(s)e^{{\bf U}(s;\btheta)y}
&{\bf 0}_{--}
\\
{\bf 0}_{++}&
\bPsi(s;\btheta)e^{{\bf D}(s;\btheta)y}&
{\bf I}&{\bf 0}
\\
{\bf 0}_{-+}&e^{{\bf D}(s;\btheta)y}&
{\bf 0}&{\bf I}
	\end{array}
	\right]^{-1}
\end{eqnarray*}
and the result follows by the matrix derivative calculus~\eqref{eq:mproduct}-\eqref{eq:minverse} and~\eqref{eq:exptrick}. \rule{9pt}{9pt}

We may also apply an alternative result as follows.
\begin{Lemma}
For any $0<x<y$,
\begin{eqnarray*}
\lefteqn{
\frac{\partial}{\partial\btheta}
\left[
	\begin{array}{cccc}
		{\bf O}_{++}&{\bf G}_{+-}^{(x,y)}(s)&
  {\bf H}_{++}^{(x,y)}(s)&{\bf O}_{+-}\\
		{\bf O}_{-+}&{\bf G}_{--}^{(x,y)}(s)&{\bf H}_{-+}^{(x,y)}(s)&{\bf O}_{--}
	\end{array}
	\right]
}
\\
&=&
\Vast\{
\frac{\partial}{\partial\btheta}
	\left[
	\begin{array}{cccc}
 {\bf 0}_{++}&
\bPsi(s;\btheta)e^{{\bf D}(s;\btheta)x}&
e^{{\bf U}(s;\btheta)(y-x)}&
{\bf 0}_{+-}
\\
{\bf 0}_{-+}&
e^{{\bf D}(s;\btheta)x}&
\bXi(s;\btheta)e^{{\bf U}(s;\btheta)(y-x)}&
{\bf 0}_{--}
	\end{array}
	\right]	
    \\
    &&
-
\left[
	\begin{array}{cccc}
		{\bf O}_{++}&{\bf G}_{+-}^{(x,y)}(s)&
  {\bf H}_{++}^{(x,y)}(s)&{\bf O}_{+-}\\
		{\bf O}_{-+}&{\bf G}_{--}^{(x,y)}(s)&{\bf H}_{-+}^{(x,y)}(s)&{\bf O}_{--}
	\end{array}
	\right]
\left(
\frac{\partial}{\partial\btheta}
\left[
	\begin{array}{cccc}
{\bf I}&{\bf 0}&
e^{{\bf U}(s;\btheta)y}&
{\bf 0}_{+-}
\\
{\bf 0}&{\bf I}
&
\bXi(s)e^{{\bf U}(s;\btheta)y}
&{\bf 0}_{--}
\\
{\bf 0}_{++}&
\bPsi(s;\btheta)e^{{\bf D}(s;\btheta)y}&
{\bf I}&{\bf 0}
\\
{\bf 0}_{-+}&e^{{\bf D}(s;\btheta)y}&
{\bf 0}&{\bf I}
	\end{array}
	\right]
\right)
\Vast\}
\\
&&
\times
\left(
{\bf I}
\otimes
\left[
	\begin{array}{cccc}
{\bf I}&{\bf 0}&
e^{{\bf U}(s;\btheta)y}&
{\bf 0}_{+-}
\\
{\bf 0}&{\bf I}
&
\bXi(s)e^{{\bf U}(s;\btheta)y}
&{\bf 0}_{--}
\\
{\bf 0}_{++}&
\bPsi(s;\btheta)e^{{\bf D}(s;\btheta)y}&
{\bf I}&{\bf 0}
\\
{\bf 0}_{-+}&e^{{\bf D}(s;\btheta)y}&
{\bf 0}&{\bf I}
	\end{array}
	\right]
\right)^{-1}.
\end{eqnarray*}

\end{Lemma}
{\bf Proof:} By Bean, O'Reilly and Taylor~\cite{BOT05}, for any $0<x<y$,
\begin{eqnarray*}
\lefteqn{
	\left[
	\begin{array}{cccc}
 {\bf 0}_{++}&
\bPsi(s;\btheta)e^{{\bf D}(s;\btheta)x}&
e^{{\bf U}(s;\btheta)(y-x)}&
{\bf 0}_{+-}
\\
{\bf 0}_{-+}&
e^{{\bf D}(s;\btheta)x}&
\bXi(s;\btheta)e^{{\bf U}(s;\btheta)(y-x)}&
{\bf 0}_{--}
	\end{array}
	\right]	
}
 \\
 &=&
\left[
	\begin{array}{cccc}
		{\bf O}_{++}&{\bf G}_{+-}^{(x,y)}(s)&
  {\bf H}_{++}^{(x,y)}(s)&{\bf O}_{+-}\\
		{\bf O}_{-+}&{\bf G}_{--}^{(x,y)}(s)&{\bf H}_{-+}^{(x,y)}(s)&{\bf O}_{--}
	\end{array}
	\right]
\left[
	\begin{array}{cccc}
{\bf I}&{\bf 0}&
e^{{\bf U}(s;\btheta)y}&
{\bf 0}_{+-}
\\
{\bf 0}&{\bf I}
&
\bXi(s)e^{{\bf U}(s;\btheta)y}
&{\bf 0}_{--}
\\
{\bf 0}_{++}&
\bPsi(s;\btheta)e^{{\bf D}(s;\btheta)y}&
{\bf I}&{\bf 0}
\\
{\bf 0}_{-+}&e^{{\bf D}(s;\btheta)y}&
{\bf 0}&{\bf I}
	\end{array}
	\right]
\end{eqnarray*}
and so,
\begin{eqnarray*}
\lefteqn{
\frac{\partial}{\partial\btheta}
	\left[
	\begin{array}{cccc}
 {\bf 0}_{++}&
\bPsi(s;\btheta)e^{{\bf D}(s;\btheta)x}&
e^{{\bf U}(s;\btheta)(y-x)}&
{\bf 0}_{+-}
\\
{\bf 0}_{-+}&
e^{{\bf D}(s;\btheta)x}&
\bXi(s;\btheta)e^{{\bf U}(s;\btheta)(y-x)}&
{\bf 0}_{--}
	\end{array}
	\right]	
}
\\
&=&
\frac{\partial}{\partial\btheta}
\left[
	\begin{array}{cccc}
		{\bf O}_{++}&{\bf G}_{+-}^{(x,y)}(s)&
  {\bf H}_{++}^{(x,y)}(s)&{\bf O}_{+-}\\
		{\bf O}_{-+}&{\bf G}_{--}^{(x,y)}(s)&{\bf H}_{-+}^{(x,y)}(s)&{\bf O}_{--}
	\end{array}
	\right]
\left(
{\bf I}
\otimes
\left[
	\begin{array}{cccc}
{\bf I}&{\bf 0}&
e^{{\bf U}(s;\btheta)y}&
{\bf 0}_{+-}
\\
{\bf 0}&{\bf I}
&
\bXi(s)e^{{\bf U}(s;\btheta)y}
&{\bf 0}_{--}
\\
{\bf 0}_{++}&
\bPsi(s;\btheta)e^{{\bf D}(s;\btheta)y}&
{\bf I}&{\bf 0}
\\
{\bf 0}_{-+}&e^{{\bf D}(s;\btheta)y}&
{\bf 0}&{\bf I}
	\end{array}
	\right]
\right)
\\
&&+
\left[
	\begin{array}{cccc}
		{\bf O}_{++}&{\bf G}_{+-}^{(x,y)}(s)&
  {\bf H}_{++}^{(x,y)}(s)&{\bf O}_{+-}\\
		{\bf O}_{-+}&{\bf G}_{--}^{(x,y)}(s)&{\bf H}_{-+}^{(x,y)}(s)&{\bf O}_{--}
	\end{array}
	\right]
\left(
\frac{\partial}{\partial\btheta}
\left[
	\begin{array}{cccc}
{\bf I}&{\bf 0}&
e^{{\bf U}(s;\btheta)y}&
{\bf 0}_{+-}
\\
{\bf 0}&{\bf I}
&
\bXi(s)e^{{\bf U}(s;\btheta)y}
&{\bf 0}_{--}
\\
{\bf 0}_{++}&
\bPsi(s;\btheta)e^{{\bf D}(s;\btheta)y}&
{\bf I}&{\bf 0}
\\
{\bf 0}_{-+}&e^{{\bf D}(s;\btheta)y}&
{\bf 0}&{\bf I}
	\end{array}
	\right]
\right)
,
\end{eqnarray*}
and the result follows. \rule{9pt}{9pt}

\begin{Theorem}\label{th:der_LT_transient}
We have,
\begin{eqnarray*}
\lefteqn{
\frac{\partial}{\partial\btheta}
\widetilde {\bf p}(s;\btheta)
=
{\bf g}_-\times
\Bigg(
\frac{\partial}{\partial\btheta}
e^{
{\bf D}(s;\btheta)
z}
\Bigg)
}
\nonumber\\
&&
\times
{\bf I}\otimes
\left(
{\bf I}
-
\left[
\begin{array}{cc}
{\bf I}&{\bf 0}
\end{array}
\right]
\left(-
{\bf T}_{\ominus\ominus}(\btheta)
+s{\bf I}
\right)^{-1}
{\bf T}_{\ominus +}(\btheta)
\bPsi(s;\btheta)
\right)^{-1}
\left[
\begin{array}{cc}
{\bf I}&{\bf 0}
\end{array}
\right]
\left(-
{\bf T}_{\ominus\ominus}(\btheta)
+s{\bf I}
\right)^{-1}
\nonumber\\
&&
+
{\bf g}_-\times
e^{
{\bf D}(s;\btheta)
z}
\nonumber\\
&&
\times
\Bigg\{
\quad
\left(
\frac{\partial}{\partial\btheta}
\left(
{\bf I}
-
\left[
\begin{array}{cc}
{\bf I}&{\bf 0}
\end{array}
\right]
\left(-
{\bf T}_{\ominus\ominus}(\btheta)
+s{\bf I}
\right)^{-1}
{\bf T}_{\ominus +}(\btheta)
\bPsi(s;\btheta)
\right)^{-1}
\right)
\left( {\bf I}\otimes
\left[
\begin{array}{cc}
{\bf I}&{\bf 0}
\end{array}
\right]
\left(-
{\bf T}_{\ominus\ominus}(\btheta)
+s{\bf I}
\right)^{-1}
\right)
\nonumber\\
\nonumber\\
&&
\quad \quad
+
\left(
{\bf I}
-
\left[
\begin{array}{cc}
{\bf I}&{\bf 0}
\end{array}
\right]
\left(-
{\bf T}_{\ominus\ominus}(\btheta)
+s{\bf I}
\right)^{-1}
{\bf T}_{\ominus +}(\btheta)
\bPsi(s;\btheta)
\right)^{-1}
\left[
\begin{array}{cc}
{\bf I}&{\bf 0}
\end{array}
\right]
\left(
\frac{\partial}{\partial\btheta}
\left(-
{\bf T}_{\ominus\ominus}(\btheta)
+s{\bf I}
\right)^{-1}
\right)
\quad
\Bigg\}
,
\nonumber\\
\end{eqnarray*}
where
\begin{eqnarray*}
\lefteqn{
\quad
\frac{\partial}{\partial\btheta}
\left(
{\bf I}
-
\left[
\begin{array}{cc}
{\bf I}&{\bf 0}
\end{array}
\right]
\left(-
{\bf T}_{\ominus\ominus}(\btheta)
+s{\bf I}
\right)^{-1}
{\bf T}_{\ominus +}(\btheta)
\bPsi(s;\btheta)
\right)^{-1}
}
\nonumber
\\
&=&
\left(
{\bf I}
-
\left[
\begin{array}{cc}
{\bf I}&{\bf 0}
\end{array}
\right]
\left(-
{\bf T}_{\ominus\ominus}(\btheta)
+s{\bf I}
\right)^{-1}
{\bf T}_{\ominus +}(\btheta)
\bPsi(s;\btheta)
\right)^{-1}
\left[
\begin{array}{cc}
{\bf I}&{\bf 0}
\end{array}
\right]
\nonumber
\\
&&
\times
\Bigg(
\quad
\left(
\frac{\partial}{\partial\btheta}
\left(-
{\bf T}_{\ominus\ominus}(\btheta)
+s{\bf I}
\right)^{-1}
\right)
\left(
{\bf I}\otimes
{\bf T}_{\ominus +}(\btheta)
\bPsi(s;\btheta)
\right)
\nonumber
\\
&&
\quad \quad
+
\left(-
{\bf T}_{\ominus\ominus}(\btheta)
+s{\bf I}
\right)^{-1}
\left(
\left(
\frac{\partial}{\partial\btheta}
{\bf T}_{\ominus +}(\btheta)
\right)
\left(
{\bf I}\otimes
\bPsi(s;\btheta)
\right)
+
{\bf T}_{\ominus +}(\btheta)
\left(
\frac{\partial}{\partial\btheta}
\bPsi(s;\btheta)
\right)
\right)
\quad
\Bigg)
\nonumber
\\
&&
\times
\left(
{\bf I}\otimes
\left(
{\bf I}
-
\left[
\begin{array}{cc}
{\bf I}&{\bf 0}
\end{array}
\right]
\left(-
{\bf T}_{\ominus\ominus}(\btheta)
+s{\bf I}
\right)^{-1}
{\bf T}_{\ominus +}(\btheta)
\bPsi(s;\btheta)
\right)^{-1}
\right)
\end{eqnarray*}
and
\begin{eqnarray*}
\frac{\partial} { \partial\btheta}
(-{\bf T}_{00}(\btheta)+s{\bf I})^{-1}
&=&
(-{\bf T}_{00}(\btheta)+s{\bf I})^{-1}
\frac{\partial {\bf T}_{00}(\btheta)}{\partial \btheta}
\left(
{\bf I}
\otimes (-{{\bf T}_{00}(\btheta)}+s{\bf I})^{-1}
\right)
.
\end{eqnarray*}
Further,
\begin{eqnarray*}
\lefteqn{
\frac{\partial}{\partial\btheta}
\widetilde {\bf f}_+(x,s;\btheta)
=
\frac{\partial}{\partial\btheta}
\left(
\widetilde {\bf p}(s;\btheta)
{\bf T}_{\ominus +}(\btheta)
e^{{\bf K}(s;\btheta)x}
({\bf C}_+(\btheta))^{-1}
\right)
}
\nonumber\\
&&
\quad
+
1\{x\leq z\}\times
{\bf g}_-\times
\frac{\partial}{\partial\btheta}
\left(
e^{{\bf D}(s;\btheta)(z-x)}
\left(
{\bf I}-
{\bf H}_{-+}^{(x,x)}(s;\btheta)
\bPsi(s;\btheta)
\right)^{-1}
{\bf H}_{-+}^{(x,x)}(s;\btheta)
({\bf C}_+(\btheta))^{-1}
\right)
\nonumber\\
&&
\quad
+
1\{x>z\}\times
{\bf g}_-\times
\frac{\partial}{\partial\btheta}
\left(
{\bf H}_{-+}^{(z,z)}(s;\btheta)
{\bf H}_{++}^{(z,x)}(s;\btheta)
\left(
{\bf I}-\bPsi(s;\btheta){\bf H}_{-+}^{(x,x)}(s;\btheta)
\right)^{-1}
({\bf C}_+(\btheta))^{-1}
\right)
\nonumber\\
\end{eqnarray*}
and
\begin{eqnarray*}
\lefteqn{
\frac{\partial}{\partial\btheta}
\widetilde {\bf f}_-(x,s;\btheta)
=
\frac{\partial}{\partial\btheta}
\left(
\widetilde {\bf p}(s;\btheta)
{\bf T}_{\ominus +}(\btheta)
e^{{\bf K}(s;\btheta)x}
\bPsi(s;\btheta)
|{\bf C}_-(\btheta)|^{-1}
\right)
}
\nonumber\\
&&
\quad
+
1\{x\leq z\}\times
{\bf g}_-\times
\frac{\partial}{\partial\btheta}
\left(
e^{{\bf D}(s;\btheta)(z-x)}
\left(
{\bf I}-{\bf H}_{-+}^{(z,z)}(s;\btheta)\bPsi(s;\btheta)
\right)^{-1}
|{\bf C}_-(\btheta)|^{-1}
\right)
\nonumber\\
&&
\quad
+
1\{x>z\}\times
{\bf g}_-\times
\frac{\partial}{\partial\btheta}
\left(
{\bf H}_{-+}^{(z,z)}(s;\btheta)
{\bf H}_{++}^{(z,x)}(s;\btheta)
\left(
{\bf I}-
\bPsi(s;\btheta)
{\bf H}_{-+}^{(x,x)}(s;\btheta)
\right)^{-1}
\bPsi(s;\btheta)
|{\bf C}_-(\btheta)|^{-1}
\right)
\nonumber\\
\end{eqnarray*}
and
\begin{eqnarray*}
\lefteqn{
\frac{\partial}{\partial\btheta}
\widetilde {\bf f}_0(x,s;\btheta)
\Bigg(
\frac{\partial}{\partial\btheta}
\left[
\begin{array}{cc}
\widetilde {\bf f}_+(x,s;\btheta)
&
\widetilde {\bf f}_-(x,s;\btheta)
\end{array}
\right]
\Bigg)
\Bigg(
{\bf I}\otimes
{\bf T}_{\pm 0}
(-{\bf T}_{00}+s{\bf I})^{-1}
\Bigg)
}
\nonumber
\\
&&
+
\left[
\begin{array}{cc}
\widetilde {\bf f}_+(x,s;\btheta)
&
\widetilde {\bf f}_-(x,s;\btheta)
\end{array}
\right]
\Bigg(
\left(
\frac{\partial}{\partial\btheta}
{\bf T}_{\pm 0}
\right)
\left(
{\bf I}\otimes
(-{\bf T}_{00}+s{\bf I})^{-1}
\right)
+
\left(
\frac{\partial}{\partial\btheta}
(-{\bf T}_{00}+s{\bf I})^{-1}
\right)
\Bigg),
\nonumber\\
\end{eqnarray*}
where
\begin{eqnarray*}
\lefteqn{
\frac{\partial}{\partial\btheta}
\left(
\widetilde {\bf p}(s;\btheta)
{\bf T}_{\ominus +}(\btheta)
e^{{\bf K}(s;\btheta)x}
\bPsi(s;\btheta)
|{\bf C}_-(\btheta)|^{-1}
\right)
}
\nonumber
\\
&=&
\left(
\frac{\partial \widetilde {\bf p}(s;\btheta)}{\partial\btheta}
\left(
{\bf I}
\otimes
{\bf T}_{\ominus +}(\btheta)
\right)
+
 \widetilde {\bf p}(s;\btheta)
\frac{\partial
{\bf T}_{\ominus +}(\btheta)
}
{\partial\btheta}
\right)
\left(
{\bf I}\otimes
{\bf T}_{\ominus +}(\btheta)
e^{{\bf K}(s;\btheta)x}
\bPsi(s;\btheta)
|{\bf C}_-(\btheta)|^{-1}
\right)
\nonumber
\\
&&
+
\widetilde {\bf p}(s;\btheta)
{\bf T}_{\ominus +}(\btheta)
\Bigg\{
\left(
\frac{\partial}{\partial\btheta}
e^{{\bf K}(s;\btheta)x}
\right)
\left(
{\bf I}\otimes
\bPsi(s;\btheta)
|{\bf C}_-(\btheta)|^{-1}
\right)
\nonumber
\\
&&
\quad
+
e^{{\bf K}(s;\btheta)x}
\left(
\left(
\frac{\partial}{\partial\btheta}
\bPsi(s;\btheta)
\right)
\left(
{\bf I}\otimes
|{\bf C}_-(\btheta)|^{-1}
\right)
+
\bPsi(s;\btheta)
\left(
\frac{\partial}{\partial\btheta}
|{\bf C}_-(\btheta)|^{-1}
\right)
\right)
\Bigg\}
,
\end{eqnarray*}
and
\begin{eqnarray*}
\lefteqn{
\frac{\partial}{\partial\btheta}
\left(
e^{{\bf D}(s;\btheta)(z-x)}
\left(
{\bf I}-
{\bf H}_{-+}^{(x,x)}(s;\btheta)
\bPsi(s;\btheta)
\right)^{-1}
{\bf H}_{-+}^{(x,x)}(s;\btheta)
({\bf C}_+(\btheta))^{-1}
\right)
}
\nonumber\\
&=&
\left(
\frac{\partial}{\partial\btheta}
e^{{\bf D}(s;\btheta)(z-x)}
\right)
\left(
{\bf I}\oplus
\left(
{\bf I}-
{\bf H}_{-+}^{(x,x)}(s;\btheta)
\bPsi(s;\btheta)
\right)^{-1}
{\bf H}_{-+}^{(x,x)}(s;\btheta)
({\bf C}_+(\btheta))^{-1}
\right)
\nonumber\\
&&
+
e^{{\bf D}(s;\btheta)(z-x)}
\left(
\frac{\partial}{\partial\btheta}
\left(
{\bf I}-
{\bf H}_{-+}^{(x,x)}(s;\btheta)
\bPsi(s;\btheta)
\right)^{-1}
\right)
\left(
{\bf I}\oplus
{\bf H}_{-+}^{(x,x)}(s;\btheta)
({\bf C}_+(\btheta))^{-1}
\right)
\nonumber\\
&&
+
e^{{\bf D}(s;\btheta)(z-x)}
\left(
{\bf I}-
{\bf H}_{-+}^{(x,x)}(s;\btheta)
\bPsi(s;\btheta)
\right)^{-1}
\nonumber\\
&&
\quad
\times
\left(
\frac{\partial {\bf H}_{-+}^{(x,x)}(s;\btheta)}{\partial\btheta}
\left(
{\bf I}\oplus
({\bf C}_+(\btheta))^{-1}
\right)
+
{\bf H}_{-+}^{(x,x)}(s;\btheta)
\frac{\partial
({\bf C}_+(\btheta))^{-1}
}{\partial\btheta}
\right)
,
\end{eqnarray*}
and
\begin{eqnarray*}
\lefteqn{
\frac{\partial}{\partial\btheta}
\left(
{\bf H}_{-+}^{(z,z)}(s;\btheta)
{\bf H}_{++}^{(z,x)}(s;\btheta)
\left(
{\bf I}-\bPsi(s;\btheta){\bf H}_{-+}^{(x,x)}(s;\btheta)
\right)^{-1}
({\bf C}_+(\btheta))^{-1}
\right)
}
\nonumber\\
&=&
\left(
\frac{\partial}{\partial\btheta}
{\bf H}_{-+}^{(z,z)}(s;\btheta)
{\bf H}_{++}^{(z,x)}(s;\btheta)
\right)
\left(
{\bf I}\oplus
\left(
{\bf I}-
\bPsi(s;\btheta)
{\bf H}_{-+}^{(x,x)}(s;\btheta)
\right)^{-1}
\bPsi(s;\btheta)
|{\bf C}_-(\btheta)|^{-1}
\right)
\nonumber\\
&&
+
{\bf H}_{-+}^{(z,z)}(s;\btheta)
{\bf H}_{++}^{(z,x)}(s;\btheta)
\left(
\frac{\partial}{\partial\btheta}
\left(
{\bf I}-
\bPsi(s;\btheta)
{\bf H}_{-+}^{(x,x)}(s;\btheta)
\right)^{-1}
\right)
\left(
{\bf I}\oplus
\bPsi(s;\btheta)
|{\bf C}_-(\btheta)|^{-1}
\right)
\nonumber\\
&&
+
{\bf H}_{-+}^{(z,z)}(s;\btheta)
{\bf H}_{++}^{(z,x)}(s;\btheta)
\left(
{\bf I}-
\bPsi(s;\btheta)
{\bf H}_{-+}^{(x,x)}(s;\btheta)
\right)^{-1}
\nonumber\\
&&
\quad
\times
\left(
\frac{\partial
\bPsi(s;\btheta)
}{\partial\btheta}
\left(
{\bf I}\oplus
|{\bf C}_-(\btheta)|^{-1}
\right)
+
\bPsi(s;\btheta)
\frac{\partial
|{\bf C}_-(\btheta)|^{-1}
}{\partial\btheta}
\right)
,
\end{eqnarray*}
and
\begin{eqnarray*}
\lefteqn{
\frac{\partial}{\partial\btheta}
\left(
e^{{\bf D}(s;\btheta)(z-x)}
\left(
{\bf I}-{\bf H}_{-+}^{(z,z)}(s;\btheta)\bPsi(s;\btheta)
\right)^{-1}
|{\bf C}_-(\btheta)|^{-1}
\right)
}
\nonumber\\
&=&
\left(
\frac{\partial}{\partial\btheta}
e^{{\bf D}(s;\btheta)(z-x)}
\right)
\left(
{\bf I}\oplus
\left(
{\bf I}-
{\bf H}_{-+}^{(z,z)}(s;\btheta)
\bPsi(s;\btheta)
\right)^{-1}
|{\bf C}_-(\btheta)|^{-1}
\right)
\nonumber\\
&&
+
e^{{\bf D}(s;\btheta)(z-x)}
\left(
\frac{\partial}{\partial\btheta}
\left(
{\bf I}-
{\bf H}_{-+}^{(z,z)}(s;\btheta)
\bPsi(s;\btheta)
\right)^{-1}
\right)
\left(
{\bf I}\oplus
|{\bf C}_-(\btheta)|^{-1}
\right)
\nonumber\\
&&
+
e^{{\bf D}(s;\btheta)(z-x)}
\left(
{\bf I}-
{\bf H}_{-+}^{(z,z)}(s;\btheta)
\bPsi(s;\btheta)
\right)^{-1}
\frac{\partial}{\partial\btheta}
\left(
|{\bf C}_-(\btheta)|^{-1}
\right)
,
\end{eqnarray*}
and
\begin{eqnarray*}
\lefteqn{
\frac{\partial}{\partial\btheta}
\left(
{\bf H}_{-+}^{(z,z)}(s;\btheta)
{\bf H}_{++}^{(z,x)}(s;\btheta)
\left(
{\bf I}-
\bPsi(s;\btheta)
{\bf H}_{-+}^{(x,x)}(s;\btheta)
\right)^{-1}
\bPsi(s;\btheta)
|{\bf C}_-(\btheta)|^{-1}
\right)
}
\nonumber\\
&=&
\left(
\frac{\partial}{\partial\btheta}
{\bf H}_{-+}^{(z,z)}(s;\btheta)
{\bf H}_{++}^{(z,x)}(s;\btheta)
\right)
\left(
{\bf I}\oplus
\left(
{\bf I}-
\bPsi(s;\btheta)
{\bf H}_{-+}^{(x,x)}(s;\btheta)
\right)^{-1}
\bPsi(s;\btheta)
|{\bf C}_-(\btheta)|^{-1}
\right)
\nonumber\\
&&
+
{\bf H}_{-+}^{(z,z)}(s;\btheta)
{\bf H}_{++}^{(z,x)}(s;\btheta)
\left(
\frac{\partial}{\partial\btheta}
\left(
{\bf I}-
\bPsi(s;\btheta)
{\bf H}_{-+}^{(x,x)}(s;\btheta)
\right)^{-1}
\right)
\left(
{\bf I}\oplus
\bPsi(s;\btheta)
|{\bf C}_-(\btheta)|^{-1}
\right)
\nonumber\\
&&
+
{\bf H}_{-+}^{(z,z)}(s;\btheta)
{\bf H}_{++}^{(z,x)}(s;\btheta)
\left(
{\bf I}-
\bPsi(s;\btheta)
{\bf H}_{-+}^{(x,x)}(s;\btheta)
\right)^{-1}
\frac{\partial}{\partial\btheta}
\left(
\bPsi(s;\btheta)
|{\bf C}_-(\btheta)|^{-1}
\right)
.
\end{eqnarray*}
with
\begin{eqnarray*}
\lefteqn{
\frac{\partial}{\partial\btheta}
\left(
{\bf I}-
{\bf H}_{-+}^{(x,x)}(s;\btheta)
\bPsi(s;\btheta)
\right)^{-1}
=
-
\left(
{\bf I}-
{\bf H}_{-+}^{(x,x)}(s;\btheta)
\bPsi(s;\btheta)
\right)^{-1}
}
\nonumber\\
&&
\times
\left(
\frac{\partial {\bf H}_{-+}^{(x,x)}(s;\btheta)}
{\partial \btheta}
\left(
{\bf I}\oplus
\bPsi(s;\btheta)
\right)
+
{\bf H}_{-+}^{(x,x)}(s;\btheta)
\frac{\partial \bPsi(s;\btheta)}
{\partial \btheta}
\right)
\left(
{\bf I}
\otimes
\left(
{\bf I}-
{\bf H}_{-+}^{(x,x)}(s;\btheta)
\bPsi(s;\btheta)
\right)^{-1}
\right)
\nonumber
\\
\lefteqn{
\frac{\partial}{\partial\btheta}
\left(
{\bf I}-
\bPsi(s;\btheta)
{\bf H}_{-+}^{(x,x)}(s;\btheta)
\right)^{-1}
=
-
\left(
{\bf I}-
\bPsi(s;\btheta)
{\bf H}_{-+}^{(x,x)}(s;\btheta)
\right)^{-1}
}
\nonumber\\
&&
\times
\left(
\frac{\partial
\bPsi(s;\btheta)
}
{\partial \btheta}
\left(
{\bf I}\oplus
{\bf H}_{-+}^{(x,x)}(s;\btheta)
\right)
+
\bPsi(s;\btheta)
\frac{\partial
{\bf H}_{-+}^{(x,x)}(s;\btheta)
}
{\partial \btheta}
\right)
\left(
{\bf I}
\otimes
\left(
{\bf I}-
\bPsi(s;\btheta)
{\bf H}_{-+}^{(x,x)}(s;\btheta)
\right)^{-1}
\right)
.
\nonumber
\\
\end{eqnarray*}

\end{Theorem}
{\bf Proof:} First,
\begin{eqnarray*}
\widetilde {\bf p}(s;\btheta)
&=&
{\bf g}_-\times e^{
{\bf D}(s;\btheta)
z}
\nonumber\\
&&
\quad
\times
\left(
{\bf I}
-
\left[
\begin{array}{cc}
{\bf I}&{\bf 0}
\end{array}
\right]
\left(-
{\bf T}_{\ominus\ominus}(\btheta)
+s{\bf I}
\right)^{-1}
{\bf T}_{\ominus +}(\btheta)
\bPsi(s;\btheta)
\right)^{-1}
\left[
\begin{array}{cc}
{\bf I}&{\bf 0}
\end{array}
\right]
\left(-
{\bf T}_{\ominus\ominus}(\btheta)
+s{\bf I}
\right)^{-1},
\end{eqnarray*}
since assuming start from level $z$, for the process to be observed at level zero, it must
\begin{itemize}
\item first drain to level zero, according to the Laplace transform ${\bf g}_-\times e^{
{\bf D}(s;\btheta)
z}$,
\item next, the process may leave level zero, and then return to it, and do so any number of times including zero, according to the Laplace transform $\left(
{\bf I}
-
\left[
\begin{array}{cc}
{\bf I}&{\bf 0}
\end{array}
\right]
\left(-
{\bf T}_{\ominus\ominus}(\btheta)
+s{\bf I}
\right)^{-1}
{\bf T}_{\ominus +}(\btheta)
\bPsi(s;\btheta)
\right)^{-1}$,
\item and then it must remain at level zero for some time, according the Laplace transform $\left(-
{\bf T}_{\ominus\ominus}(\btheta)
+s{\bf I}
\right)^{-1}$.
\end{itemize}

Next,
\begin{eqnarray*}
\lefteqn{
\widetilde {\bf f}_+(x,s;\btheta)
=
\widetilde {\bf p}(s;\btheta)
{\bf T}_{\ominus +}(\btheta)
e^{{\bf K}(s;\btheta)x}
({\bf C}_+(\btheta))^{-1}
}
\nonumber\\
&&
\quad
+
1\{x\leq z\}\times
{\bf g}_-\times
e^{{\bf D}(s;\btheta)(z-x)}
{\bf H}_{-+}^{(x,x)}(s;\btheta)
\left(
{\bf I}-\bPsi(s;\btheta){\bf H}_{-+}^{(x,x)}(s;\btheta)
\right)^{-1}
({\bf C}_+(\btheta))^{-1}
\nonumber\\
&&
\quad
+
1\{x>z\}\times
{\bf g}_-\times
{\bf H}_{-+}^{(z,z)}(s;\btheta)
{\bf H}_{++}^{(z,x)}(s;\btheta)
\left(
{\bf I}-\bPsi(s;\btheta){\bf H}_{-+}^{(x,x)}(s;\btheta)
\right)^{-1}
({\bf C}_+(\btheta))^{-1}
\\
&=&
\widetilde {\bf p}(s;\btheta)
{\bf T}_{\ominus +}(\btheta)
e^{{\bf K}(s;\btheta)x}
({\bf C}_+(\btheta))^{-1}
\nonumber\\
&&
\quad
+
1\{x\leq z\}\times
{\bf g}_-\times
e^{{\bf D}(s;\btheta)(z-x)}
\left(
{\bf I}-
{\bf H}_{-+}^{(x,x)}(s;\btheta)
\bPsi(s;\btheta)
\right)^{-1}
{\bf H}_{-+}^{(x,x)}(s;\btheta)
({\bf C}_+(\btheta))^{-1}
\nonumber\\
&&
\quad
+
1\{x>z\}\times
{\bf g}_-\times
{\bf H}_{-+}^{(z,z)}(s;\btheta)
{\bf H}_{++}^{(z,x)}(s;\btheta)
\left(
{\bf I}-\bPsi(s;\btheta){\bf H}_{-+}^{(x,x)}(s;\btheta)
\right)^{-1}
({\bf C}_+(\btheta))^{-1}
\\
\end{eqnarray*}
since assuming start from level $z$, for the process to be observed at level $x$, it may
\begin{itemize}
\item visit level $x$ by first visiting level zero, according to the Laplace transform $\widetilde {\bf p}(s;\btheta)
{\bf T}_{\ominus +}(\btheta)
e^{{\bf K}(s;\btheta)x}
({\bf C}_+)^{-1}
$,
\end{itemize}
or it may visit level $x$ without first visiting level zero, in which case, if $x\leq z$,
\begin{itemize}
\item it must first drain to level $z$ and so so in some phase in $\mathcal{S}_-$, according to the Laplace transform $e^{{\bf D}(s;\btheta)(z-x)}
$,
\item then visit level $x$ without visiting level zero and do so in some phase in $\mathcal{S}_+$, according to the Laplace transform ${\bf H}_{-+}^{(z,z)}(s;\btheta)$,
\item and then the process may visit level $x$ in some phase in $\mathcal{S}_+$ any number of times including zero, according to the Laplace transform $\left(
{\bf I}-\bPsi(s;\btheta){\bf H}_{-+}^{(x,x)}(s;\btheta)
\right)^{-1}
({\bf C}_+(\btheta))^{-1}$,
\end{itemize}
with a similar argument in the case $x>z$.

Further, by similar conditioning on the decomposition of the sample path, we have
\begin{eqnarray*}
\lefteqn{
\widetilde {\bf f}_-(x,s;\btheta)
=
\widetilde {\bf p}(s;\btheta)
{\bf T}_{\ominus +}(\btheta)
e^{{\bf K}(s;\btheta)x}
\bPsi(s;\btheta)
|{\bf C}_-(\btheta)|^{-1}
}
\nonumber\\
&&
\quad
+
1\{x\leq z\}\times
{\bf g}_-\times
e^{{\bf D}(s;\btheta)(z-x)}
\left(
{\bf I}-{\bf H}_{-+}^{(z,z)}(s;\btheta)\bPsi(s;\btheta)
\right)^{-1}
|{\bf C}_-(\btheta)|^{-1}
\nonumber\\
&&
\quad
+
1\{x>z\}\times
{\bf g}_-\times
{\bf H}_{-+}^{(z,z)}(s;\btheta)
{\bf H}_{++}^{(z,x)}(s;\btheta)
\bPsi(s;\btheta)
\left(
{\bf I}-
{\bf H}_{-+}^{(x,x)}(s;\btheta)
\bPsi(s;\btheta)
\right)^{-1}
|{\bf C}_-(\btheta)|^{-1}
\\
&=&
\widetilde {\bf p}(s;\btheta)
{\bf T}_{\ominus +}(\btheta)
e^{{\bf K}(s;\btheta)x}
\bPsi(s;\btheta)
|{\bf C}_-(\btheta)|^{-1}
\nonumber\\
&&
\quad
+
1\{x\leq z\}\times
{\bf g}_-\times
e^{{\bf D}(s;\btheta)(z-x)}
\left(
{\bf I}-{\bf H}_{-+}^{(z,z)}(s;\btheta)\bPsi(s;\btheta)
\right)^{-1}
|{\bf C}_-(\btheta)|^{-1}
\nonumber\\
&&
\quad
+
1\{x>z\}\times
{\bf g}_-\times
{\bf H}_{-+}^{(z,z)}(s;\btheta)
{\bf H}_{++}^{(z,x)}(s;\btheta)
\left(
{\bf I}-
\bPsi(s;\btheta)
{\bf H}_{-+}^{(x,x)}(s;\btheta)
\right)^{-1}
\bPsi(s;\btheta)
|{\bf C}_-(\btheta)|^{-1}
\end{eqnarray*}
and
\begin{eqnarray*}
\widetilde {\bf f}_0(x,s;\btheta)
&=&
\left[
\begin{array}{cc}
\widetilde {\bf f}_+(x,s;\btheta)
&
\widetilde {\bf f}_-(x,s;\btheta)
\end{array}
\right]
{\bf T}_{\pm 0}
(-{\bf T}_{00}+s{\bf I})^{-1}
,
\end{eqnarray*}
so the expressions for $\frac{\partial}{\partial\btheta}
\widetilde {\bf p}(s;\btheta)$ and $\frac{\partial}{\partial\btheta}
\widetilde {\bf f}(x,s;\btheta)$ follow by the matrix derivative calculus~\eqref{eq:mproduct}-\eqref{eq:minverse}. \rule{9pt}{9pt}

We note that the transient analysis given the SFM starts from level $X(0)= z>0$ and does so in some phase $\varphi(0)\in\mathcal{S}_+$ according to distribution ${\bf g}_+=[g_i]_{i\in\mathcal{S}_+}$ can be obtained by analogous arguments, since then
\begin{eqnarray*}
\widetilde {\bf p}(s;\btheta)
&=&
{\bf g}_+ \times \widetilde\bPsi(s;\btheta)\times e^{
{\bf D}(s;\btheta)
z}
\nonumber\\
&&
\quad
\times
\left(
{\bf I}
-
\left[
\begin{array}{cc}
{\bf I}&{\bf 0}
\end{array}
\right]
\left(-
{\bf T}_{\ominus\ominus}(\btheta)
+s{\bf I}
\right)^{-1}
{\bf T}_{\ominus +}(\btheta)
\bPsi(s;\btheta)
\right)^{-1}
\left[
\begin{array}{cc}
{\bf I}&{\bf 0}
\end{array}
\right]
\left(-
{\bf T}_{\ominus\ominus}(\btheta)
+s{\bf I}
\right)^{-1}
\end{eqnarray*}
and
\begin{eqnarray*}
\lefteqn{
\widetilde {\bf f}_+(x,s;\btheta)
=
\widetilde {\bf p}(s;\btheta)
{\bf T}_{\ominus +}(\btheta)
e^{{\bf K}(s;\btheta)x}
({\bf C}_+(\btheta))^{-1}
}
\nonumber\\
&&
\quad
+
1\{x\leq z\}\times
{\bf g}_+ \times
\widetilde\bPsi(s;\btheta)
\times
e^{{\bf D}(s;\btheta)(z-x)}
\left(
{\bf I}-
{\bf H}_{-+}^{(x,x)}(s;\btheta)
\bPsi(s;\btheta)
\right)^{-1}
{\bf H}_{-+}^{(x,x)}(s;\btheta)
({\bf C}_+(\btheta))^{-1}
\nonumber\\
&&
\quad
+
1\{x>z\}\times
{\bf g}_+\times {\bf H}_{++}^{(z,x)}(s;\btheta)
\left(
{\bf I}-\bPsi(s;\btheta){\bf H}_{-+}^{(x,x)}(s;\btheta)
\right)^{-1}
({\bf C}_+(\btheta))^{-1}
\end{eqnarray*}
and
\begin{eqnarray*}
\lefteqn{
\widetilde {\bf f}_-(x,s;\btheta)
=
\widetilde {\bf p}(s;\btheta)
{\bf T}_{\ominus +}(\btheta)
e^{{\bf K}(s;\btheta)x}
\bPsi(s;\btheta)
|{\bf C}_-(\btheta)|^{-1}
}
\nonumber\\
&&
\quad
+
1\{x\leq z\}
\times {\bf g}_+ \times
\widetilde\bPsi(s;\btheta)
\times
e^{{\bf D}(s;\btheta)(z-x)}
\left(
{\bf I}-{\bf H}_{-+}^{(z,z)}(s;\btheta)\bPsi(s;\btheta)
\right)^{-1}
|{\bf C}_-(\btheta)|^{-1}
\nonumber\\
&&
\quad
+
1\{x>z\}\times
{\bf g}_+\times
{\bf H}_{++}^{(z,x)}(s;\btheta)
\left(
{\bf I}-
\bPsi(s;\btheta)
{\bf H}_{-+}^{(x,x)}(s;\btheta)
\right)^{-1}
\bPsi(s;\btheta)
|{\bf C}_-(\btheta)|^{-1}.
\end{eqnarray*}

\section{Examples}\label{sec:Exa}

We now construct numerical examples to illustrate the application of the theory. First, in Subsection~\ref{ex:small}, we construct a simple SFM to demonstrate the application of the above expressions for the sensitivity analysis of stationary and transient quantities.
Next, in Subsection~\ref{ex:appl}, we consider a SMF constructed by Bean, O'Reilly and Sargison in~\cite{hydropaper} for the management of hydro power generation systems, to demonstrate the application potential of our methodology to real-world systems.
Finally, in Subsection~\ref{ex:insurance} we consider application potential  to the classical insurance risk processes.
\subsection{Simple SFM example}
\label{ex:small}
Consider a simple SFM $\{(X(t),\varphi(t)):t\geq 0\}$ studied in Bean, O'Reilly and Palmowski~\cite{BOP2021}, with $\mathcal{S}=\{1,2\}$, $\mathcal{S}_+=\{1\}$, $\mathcal{S}_-=\{2\}$, $c_1=1>0$, $c_2=-1<0$,
\begin{eqnarray*}
{\bf T}=
\left[
\begin{array}{cc}
-a&a\\b&-b
\end{array}
\right]
=
\left[
\begin{array}{cc}
{\bf T}_{++}&{\bf T}_{+-}
\\{\bf T}_{-+}&{\bf T}_{--}
\end{array}
\right],
\end{eqnarray*}
and with $a>b$, so that the process is stable with negative drift $\mu=\sum_ic_i\nu_i=(b-a)/(a+b)<0$, where $\bnu=[\nu_1 \quad \nu_2]=[b\quad a ]/(a+b)$ is the stationary distribution of the Markov chain $\{\varphi(t):t\geq 0\}$.

Here, phases $1$ and $2$ may be interpreted as the `on' and `off' phase of a telecommunication buffer, and the rates $c_1$ and $c_2$ as the rates at which data enters or leaves the buffer, respectively. Then $X(t)\geq 0$ represents the volume of data in the buffer at time $t$, while $\varphi(t)\in\mathcal{S}$ represents the state of the underlying environment (the Markov chain), which modulates the evolution of the SFM.

Let $\btheta=[a\quad b]$. Then, we have $\bPsi(\btheta)=1$, $\bxi(\btheta)=1$, ${\bf K}(\btheta)=b-a$,
\begin{eqnarray*}
\alpha(\btheta)&=&
\Bigg\{
\bxi(\btheta)
\left(-
{\bf T}_{--}(\btheta)
\right)^{-1}
\Bigg(
{\bf 1}
+
{\bf T}_{-+}(\btheta)
(-{\bf K}(\btheta))^{-1}
\left[
\begin{array}{cc}
({\bf C}_{+}(\btheta))^{-1}
&{\bf \Psi}(\btheta)|{\bf C}_{-}(\btheta)|^{-1}
\end{array}
\right]
{\bf 1}\Bigg)\Bigg\}^{_1}
\\
&=&
\left\{
b^{-1}(1+b(a-b)^{-1}[1\quad 1]
{\bf 1}
)
\right\}^{-1}
=
\left\{
b^{-1}\frac{a-b+2b}{a-b}
\right\}^{-1}
=\frac{b(a-b)}{a+b}
,
\end{eqnarray*}
and
\begin{eqnarray*}
{\bf p}_{-}(\btheta)&=&
\alpha(\btheta)
\left(-
{\bf T}_{--}(\btheta)
\right)^{-1}
=
\frac{a-b}{a+b}
,
\\
\left[
\begin{array}{cc}
\bpi_{+}(x;\btheta) & \bpi_{-}(x;\btheta)
\end{array}
\right]
&=&
{\bf p}_{-}(\btheta)
{\bf T}_{-+}(\btheta)
e^{{\bf K}\left(\btheta\right)x}
\left[
\begin{array}{cc}
({\bf C}_{+}(\btheta))^{-1}
&{\bf \Psi}(\btheta)|{\bf C}_{-}(\btheta)|^{-1}
\end{array}
\right]
\\
&=&
\frac{a-b}{a+b}
b e^{(b-a)x}[1\quad 1]
.
\end{eqnarray*}

\begin{Notes}
We have,
\begin{eqnarray*}
{\bf p}_{-}(\btheta)
+
\int_{x=0}^{\infty}
\bpi_{-}(x;\btheta)
dx
&=&
\frac{a-b}{a+b}
+
\int_{x=0}^{\infty} \frac{a-b}{a+b} b  e^{(b-a)x}dx
=
\frac{a-b}{a+b}
+\frac{b}{a+b}
=
\frac{a}{a+b}
=
\nu_2
,
\\
\int_{x=0}^{\infty}
\bpi_{+}(x;\btheta)
dx
&=&
\int_{x=0}^{\infty}
\int_{x=0}^{\infty} \frac{a-b}{a+b} b  e^{(b-a)x}dx
=
\frac{b}{a+b}
=\nu_1
,
\end{eqnarray*}
as expected. \hfill $\square$
\end{Notes}

Therefore,
\begin{eqnarray*}
\frac{\partial}{\partial\btheta}
{\bf p}_-(\btheta)
&=&
\left[
\begin{array}{cc}
\frac{\partial}{\partial a}
\left(\frac{a-b}{a+b}\right)
&
\frac{\partial}{\partial b}
\left(\frac{a-b}{a+b}\right)
\end{array}
\right]
,
\\
\frac{\partial}{\partial a}
\left(\frac{a-b}{a+b}\right)
&=&
\frac{(a+b)-(a-b)}{(a+b)^2}
=
\frac{2b}{(a+b)^2}
>0
,
\\
\frac{\partial}{\partial b}
\left(\frac{a-b}{a+b}\right)
&=&
\frac{-(a+b)-(a-b)}{(a+b)^2}
=
\frac{-2a}{(a+b)^2}
<0
,
\end{eqnarray*}
and
\begin{eqnarray*}
\frac{\partial}{\partial\btheta}
\bpi_{+}(x;\btheta)
&=&
\left[
\begin{array}{cc}
\frac{\partial}{\partial a}
\bpi_{+}(x;\btheta)
&
\frac{\partial}{\partial b}
\bpi_{+}(x;\btheta)
\end{array}
\right]
,
\\
\frac{\partial}{\partial a}
\bpi_{+}(x;\btheta)
&=&
\frac{2b}{(a+b)^2} b e^{(b-a)x}
-
\frac{a-b}{a+b}b e^{(b-a)x}x
>0
\iff x<\frac{2}{a^2-b^2}
,
\\
\frac{\partial}{\partial b}
\bpi_{+}(x;\btheta)
&=&
\frac{-2a}{(a+b)^2} b e^{(b-a)x}
+
\frac{a-b}{a+b}
(e^{(b-a)x}+be^{(b-a)x}x)
>0
\iff
x>\frac{2a}{a^2-b^2}-\frac{1}{b}
,
\end{eqnarray*}
and, since here $({\bf C}_{+}(\btheta))^{-1}
={\bf \Psi}(\btheta)|{\bf C}_{-}(\btheta)|^{-1}=1$,
\begin{eqnarray*}
\frac{\partial}{\partial\btheta}
\bpi_{-}(x;\btheta)
&=&
\frac{\partial}{\partial\btheta}
\bpi_{+}(x;\btheta).
\end{eqnarray*}
We illustrate $\frac{\partial}{\partial\btheta}
{\bf p}_-(\btheta)$ and $\frac{\partial}{\partial\btheta}
\bpi_{+}(x;\btheta)$ in Figures~\ref{fig:ex1_stationary1}-\ref{fig:ex1_stationary2}.

Next,
\begin{eqnarray*}
\bPsi(s;\btheta)
&=&
\frac{(a+b+2s)-\sqrt{(a+b+2s)^2-4ab}}{2b}
,\\
{\bf D}(s;\btheta)
&=&
-b-s
+
\frac{(a+b+2s)-\sqrt{(a+b+2s)^2-4ab}}{2}
,\\
\bXi(s;\btheta)
&=&
\frac{(a+b+2s)-\sqrt{(a+b+2s)^2-4ab}}{2a}
,\\
{\bf U}(s;\btheta)
&=&
-a-s+
\frac{(a+b+2s)-\sqrt{(a+b+2s)^2-4ab}}{2}
,
\end{eqnarray*}
and with (scalar) ${\bf g}_-=1$, noting that here $\left(
\frac{\partial}{\partial \btheta}
 e^{
{\bf D}(s;\btheta)
z}
\right)=
e^{
{\bf D}(s;\btheta)z}
\left(
\frac{\partial}{\partial \btheta}
 {\bf D}(s;\btheta)
\right)z$, we have
\begin{eqnarray*}
\widetilde {\bf p}(s;\btheta)
&=&
{\bf g}_-\times e^{
{\bf D}(s;\btheta)
z}
\left(
{\bf I}
-
\left(-
{\bf T}_{--}(\btheta)
+s{\bf I}
\right)^{-1}
{\bf T}_{- +}(\btheta)
\bPsi(s;\btheta)
\right)^{-1}
\left(-
{\bf T}_{--}(\btheta)
+s{\bf I}
\right)^{-1}
\\
&=&
 e^{
{\bf D}(s;\btheta)
z}
\left(
1-(b+s)^{-1}b\bPsi(s;\btheta)
\right)^{-1}
(b+s)^{-1}
=
e^{
{\bf D}(s;\btheta)
z}
\left(
\frac{b+s-b\bPsi(s;\btheta)}{b+s}
\right)^{-1}
(b+s)^{-1}
\\
&=&
e^{
{\bf D}(s;\btheta)
z}
\times
\left(b+s-b\bPsi(s;\btheta)\right)^{-1}
,
\\
\frac{\partial}{\partial a}
\widetilde {\bf p}(s;\btheta)
&=&
\left(
e^{
{\bf D}(s;\btheta)z}
\left(
\frac{\partial}{\partial a}
 {\bf D}(s;\btheta)
\right)z
\right)
\left(b+s-b\bPsi(s;\btheta)\right)^{-1}
\\
&&
-
e^{
{\bf D}(s;\btheta)z}
\left(b+s-b\bPsi(s;\btheta)\right)^{-2}
\left(
-b
\frac{\partial}{\partial a}
 \bPsi(s;\btheta)
 \right)
,\\
\frac{\partial}{\partial b}
\widetilde {\bf p}(s;\btheta)
&=&
\left(
e^{
{\bf D}(s;\btheta)z}
\left(
\frac{\partial}{\partial b}
 {\bf D}(s;\btheta)
\right)z
\right)
\left(b+s-b\bPsi(s;\btheta)\right)^{-1}
\\
&&
-
e^{
{\bf D}(s;\btheta)z}
\left(b+s-b\bPsi(s;\btheta)\right)^{-2}
\left(
1-\bPsi(s;\btheta)
-b
\frac{\partial}{\partial b}
 \bPsi(s;\btheta)
\right)
,
\end{eqnarray*}
where
\begin{eqnarray*}
\frac{\partial}{\partial a}\bPsi(s;\btheta)
&=&
\frac{1-(1/2)\left( (a+b+2s)^2-4ab \right)^{-1/2}
 \left(
2(a+b+2s)-4b
 \right)
 }{2b}
,\\
\frac{\partial}{\partial b}\bPsi(s;\btheta)
&=&
\frac{
\left\{
1-(1/2)\left( (a+b+2s)^2-4ab \right)^{-1/2}
 \left(
2(a+b+2s)-4a
 \right)
 \right\}
 2b
}{
4b^2
}
\\
&&
-
\frac{
 \left\{
 (a+b+2s)-\sqrt{(a+b+2s)^2-4ab}
\right\}2
}{
4b^2
}
,\\
\frac{\partial}{\partial a}
 {\bf D}(s;\btheta)
 &=&
 \frac{1-(1/2)\left( (a+b+2s)^2-4ab \right)^{-1/2}
 \left(
2(a+b+2s)-4b
 \right)
 }{2}
 ,\\
 \frac{\partial}{\partial b}
 {\bf D}(s;\btheta)
 &=&
 -1+
 \frac{1-(1/2)\left( (a+b+2s)^2-4ab \right)^{-1/2}
 \left(
2(a+b+2s)-4a
 \right)
 }{2}.
\end{eqnarray*}
To obtain $\frac{\partial}{\partial \btheta}
 {\bf p}(t;\btheta)$, we inverted the Laplace transform $\frac{\partial}{\partial \btheta}
\widetilde {\bf p}(s;\btheta)$ using the algorithm by Den Iseger~\cite{DenIseger_2006} as
coded in Toutain et al.~\cite{2011TB}. 
The output is presented in Figure~\ref{fig:ex1_transient1}.

\begin{Notes}
We have,
\begin{eqnarray*}
\frac{\partial}{\partial s}
\bPsi(s;\btheta)
&=&
\frac{\partial}{\partial s}
\frac{(a+b+2s)-\sqrt{(a+b+2s)^2-4ab}}{2b}
\\
&=&
\frac{
2
-
(1/2)
\left(
(a+b+2s)^2-4ab
\right)^{-1/2}
\times
2(a+b+2s)2
}
{2b}
,\\
\lim_{s\downarrow 0}
\left(
\frac{\partial}{\partial s}
\bPsi(s;\btheta)
\right)
&=&
\frac{
1-(a-b)^{-1}(a+b)
}{b}
=
\frac{(a-b)-(a+b)}{(a-b)b}
=\frac{-2}{a-b},
\end{eqnarray*}
and so the mean time of the busy period is $-\lim_{s\downarrow 0}
\left(
\frac{\partial}{\partial s}
\bPsi(s;\btheta)
\right)=\frac{2}{a-b}$, with
\begin{eqnarray*}
\frac{\partial}{\partial a}
\left(
\frac{2}{a-b}
\right)
&=&
\frac{-2}{(a-b)^2}
<0,
\\
\frac{\partial}{\partial b}
\left(
\frac{2}{a-b}
\right)
&=&
\frac{2}{(a-b)^2}
>0,
\end{eqnarray*}
as expected. \hfill $\square$
\end{Notes}

Further,
\begin{eqnarray*}
{\bf K}(s;\btheta)
&=&-a-s+\frac{(a+b+2s)-\sqrt{(a+b+2s)^2-4ab}}{2}
=
{\bf U}(s;\btheta)
,\\
{\bf J}(s;\btheta)
&=&-b-s+
\frac{(a+b+2s)-\sqrt{(a+b+2s)^2-4ab}}{2}
={\bf D}(s;\btheta)
,
\end{eqnarray*}
and, since here ${\bf C}_+=|{\bf C}_-|=1$, we have
\begin{eqnarray*}
\widetilde {\bf f}_+(x,s;\btheta)
&=&
\widetilde {\bf p}(s;\btheta)
{\bf T}_{\ominus +}(\btheta)
e^{{\bf K}(s;\btheta)x}
\nonumber\\
&&
\quad
+
1\{x\leq z\}\times
e^{{\bf D}(s;\btheta)(z-x)}
\left(
{\bf I}-
{\bf H}_{-+}^{(x,x)}(s;\btheta)
\bPsi(s;\btheta)
\right)^{-1}
{\bf H}_{-+}^{(x,x)}(s;\btheta)
\nonumber\\
&&
\quad
+
1\{x>z\}\times
{\bf H}_{-+}^{(z,z)}(s;\btheta)
{\bf H}_{++}^{(z,x)}(s;\btheta)
\left(
{\bf I}-\bPsi(s;\btheta){\bf H}_{-+}^{(x,x)}(s;\btheta)
\right)^{-1}
,\\
\widetilde {\bf f}_-(x,s;\btheta)
&=&
\widetilde {\bf p}(s;\btheta)
{\bf T}_{\ominus +}(\btheta)
e^{{\bf K}(s;\btheta)x}
\bPsi(s;\btheta)
\nonumber\\
&&
\quad
+
1\{x\leq z\}\times
e^{{\bf D}(s;\btheta)(z-x)}
\left(
{\bf I}-{\bf H}_{-+}^{(x,x)}(s;\btheta)\bPsi(s;\btheta)
\right)^{-1}
\nonumber\\
&&
\quad
+
1\{x>z\}\times
{\bf H}_{-+}^{(z,z)}(s;\btheta)
{\bf H}_{++}^{(z,x)}(s;\btheta)
\left(
{\bf I}-
\bPsi(s;\btheta)
{\bf H}_{-+}^{(x,x)}(s;\btheta)
\right)^{-1}
\bPsi(s;\btheta)
.
\end{eqnarray*}
To obtain $\frac{\partial}{\partial\btheta} {\bf f}_+(x,t;\btheta)$ and $\frac{\partial}{\partial\btheta} {\bf f}_-(x,t;\btheta)$, we evaluate the Laplace transforms $\frac{\partial}{\partial\btheta}\widetilde {\bf f}_+(x,s;\btheta)$ and $\frac{\partial}{\partial\btheta}\widetilde {\bf f}_-(x,s;\btheta)$ respectively, using Theorem~\ref{th:der_LT_transient}, and then invert them using the algorithm by 
Horv{\'a}th et al.~\cite{horvath2020numerical}.
The output is presented in Figures~\ref{fig:ex1_transientNEW} and \ref{fig:ex1_transientNEW2}.

In Figure \ref{fig:ex1_stationary1}, the left graph shows the stationary boundary probability $p_-(\theta)$. Increasing $a$, which corresponds to more frequent transitions into the decreasing phase, increases $p_-(\theta)$, while increasing $b$, which promotes transitions into the increasing phase, decreases it; this is reflected in the signs of the corresponding sensitivities shown in the middle and right panels. The magnitude of these effects increases as $b \uparrow a$ or $b \downarrow 0$, where the system approaches the critical regime.

\begin{figure}[h]
\centering
\includegraphics[scale=0.32]{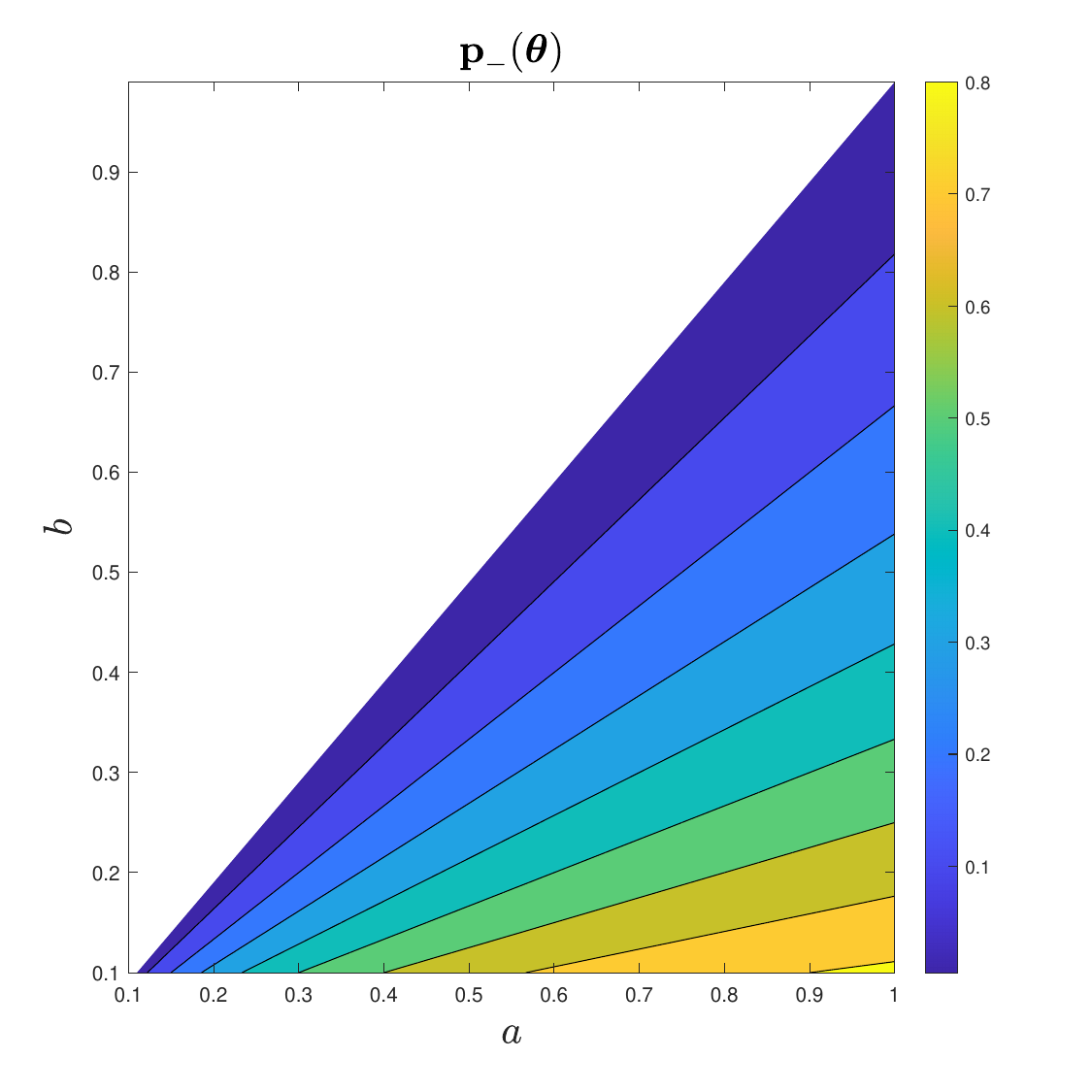}
\includegraphics[scale=0.32]{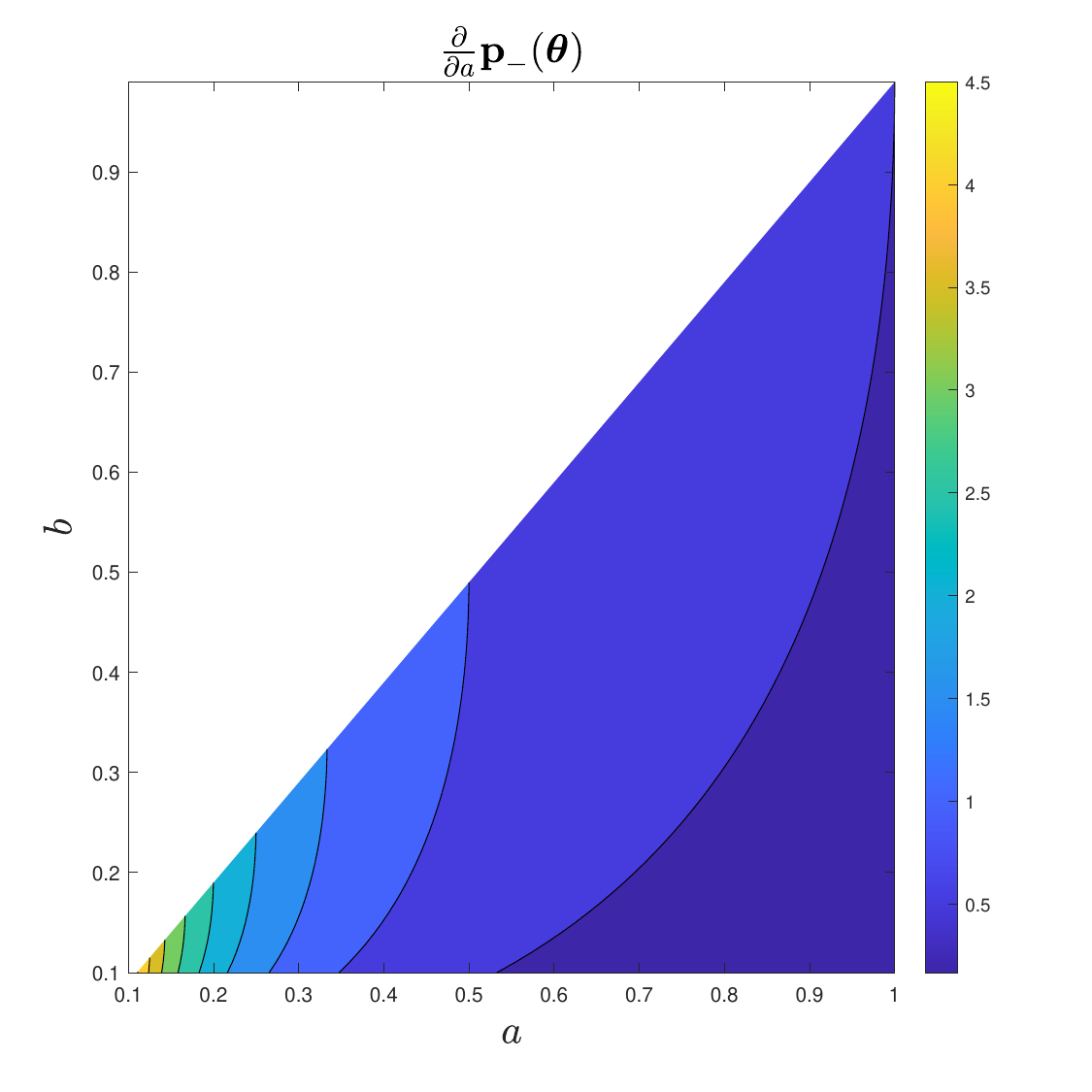}
\includegraphics[scale=0.32]{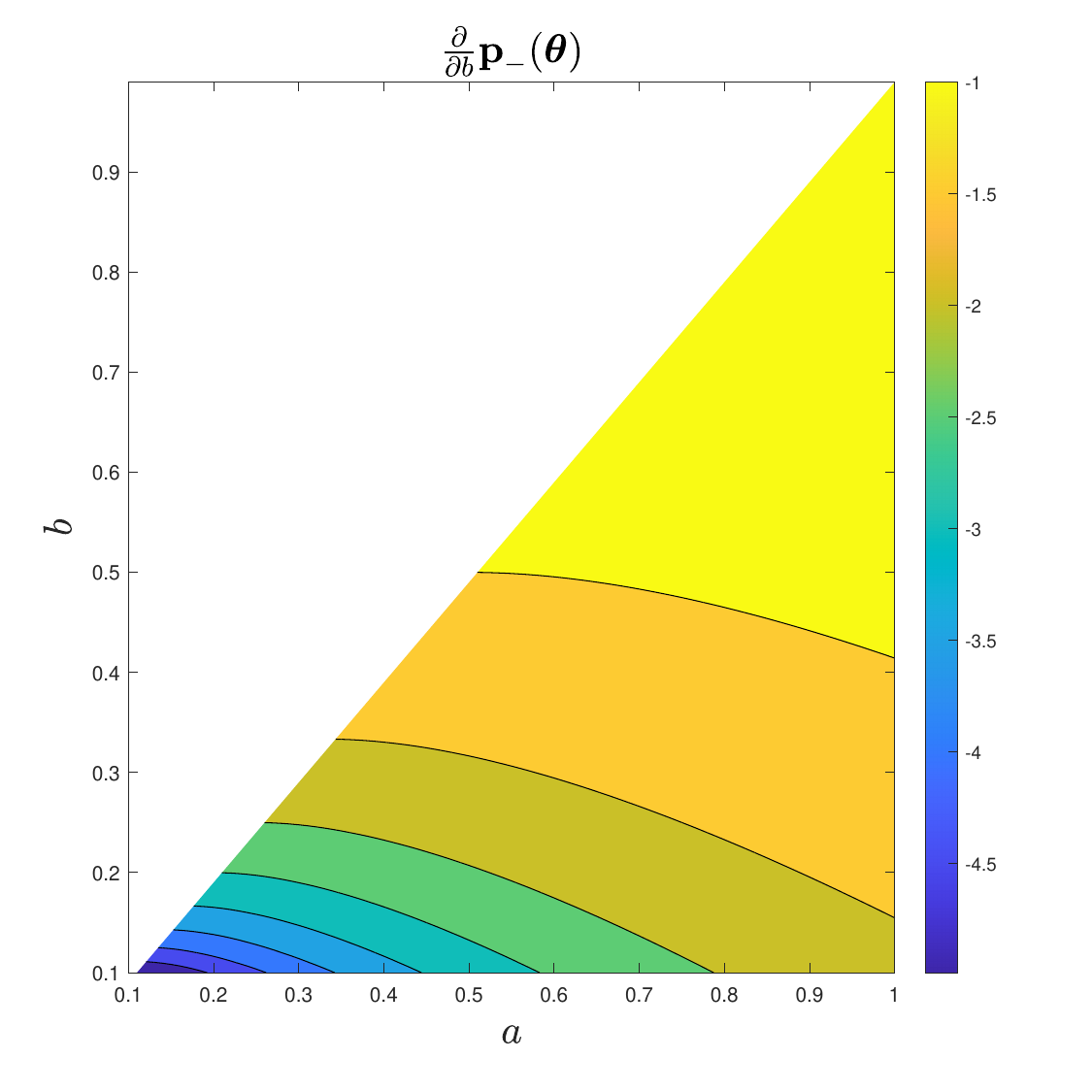}
\caption{Derivatives of the stationary quantities of the process $\{( X(t),\varphi(t)):t\geq 0\}$: ${\bf p}_-(\btheta)$, $\frac{\partial}{\partial a}{\bf p}_-(\btheta)$, and $\frac{\partial}{\partial b}{\bf p}_-(\btheta)$, for $0.1\leq b\leq a-0.01\leq 0.99$ (domain on each graph is limited to subdiagonal triangle).
}
\label{fig:ex1_stationary1}
\end{figure}

\begin{figure}[h]
\centering
\includegraphics[scale=0.32]{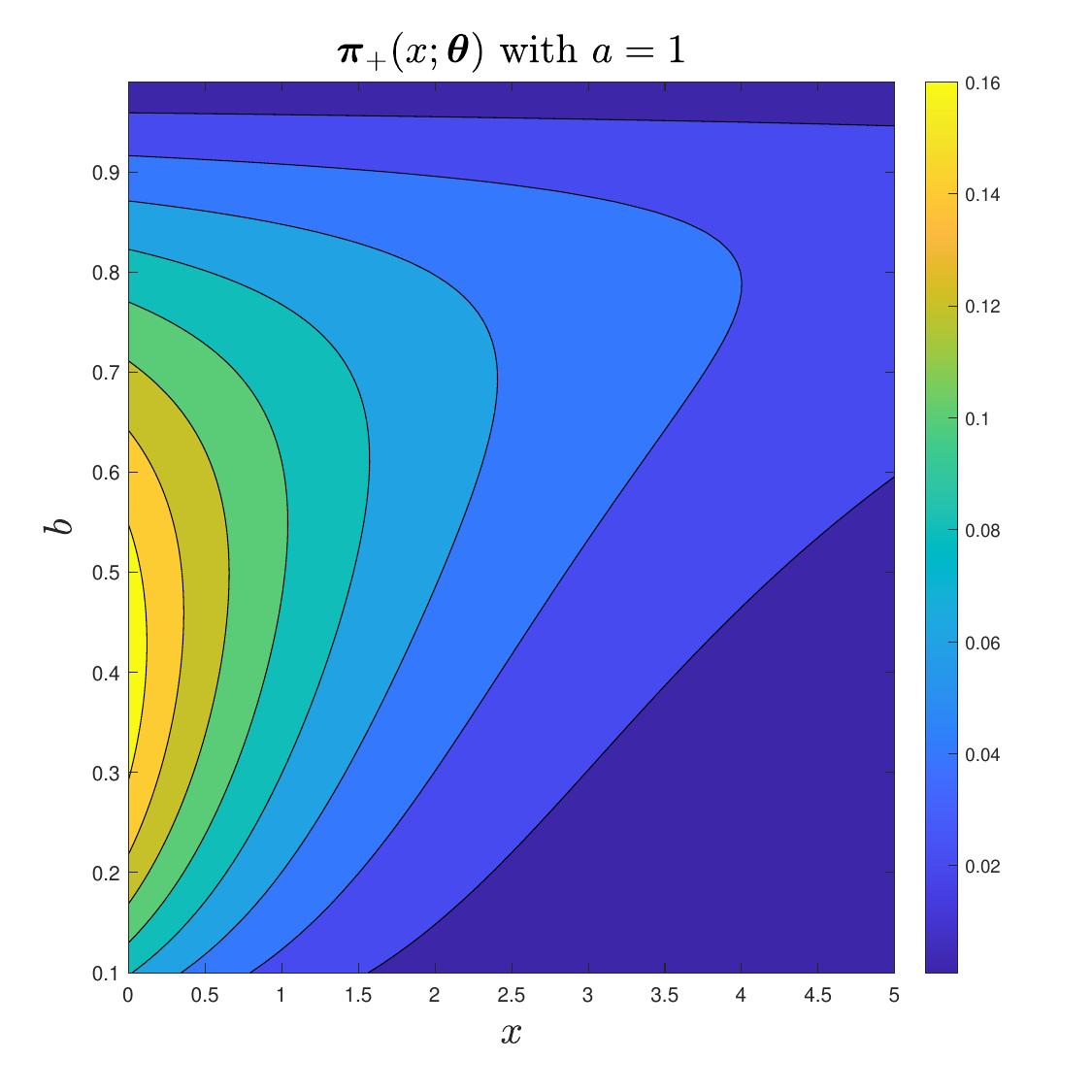}
\includegraphics[scale=0.32]{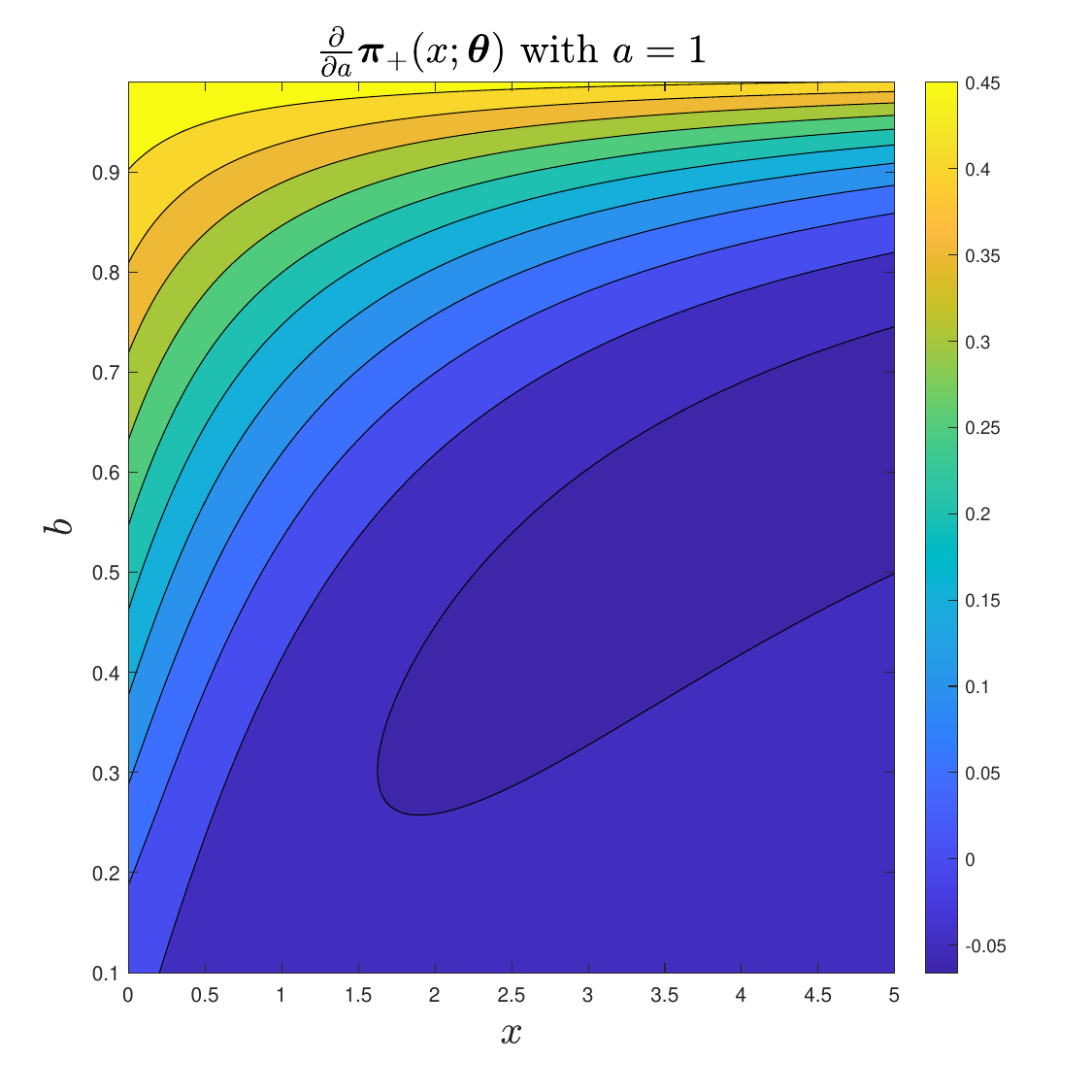}
\includegraphics[scale=0.32]{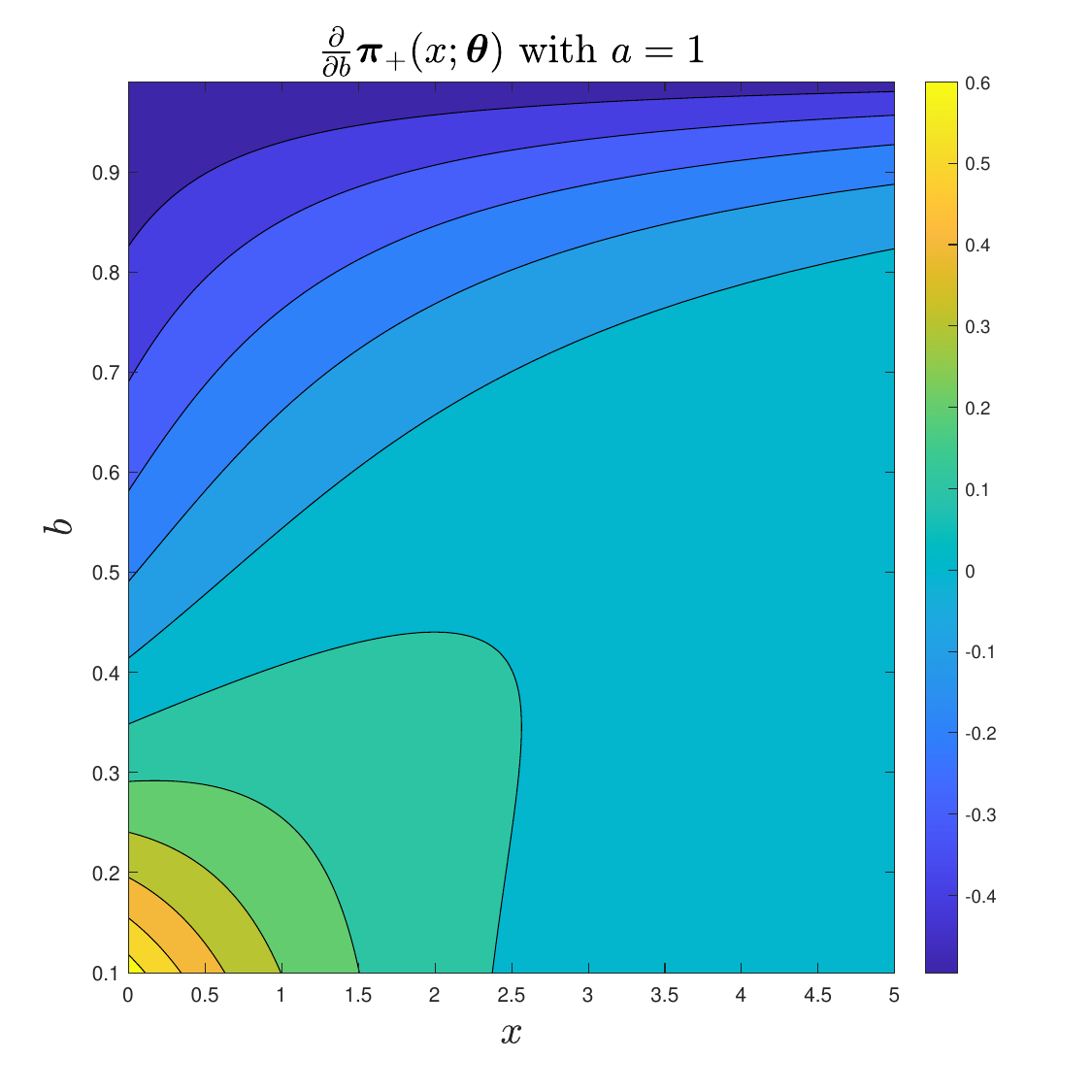}
\caption{Derivatives of the stationary quantities of the process $\{( X(t),\varphi(t)):t\geq 0\}$:
${\bpi}_+(x;\btheta)$, $\frac{\partial}{\partial a}{\bpi}_+(x;\btheta)$, and $\frac{\partial}{\partial b}{\bpi}_+(x;\btheta)$, for $0.1\leq b\leq 0.99$ and $0<x<5$.}
\label{fig:ex1_stationary2}
\end{figure}
In Figure \ref{fig:ex1_stationary2}, we observe that the stationary probability density in phase~1 is decreasing in $x$ for all values of $b$ with $a=1$. For small values of $b$, the process spends relatively little time in state~1, since transitions from state~2 to state~1 are infrequent, and therefore the density in phase~1 is lower. As $b$ increases, transitions to state~1 become more frequent, so the process spends more time in the increasing phase and the distribution shifts towards larger values of $x$.
For mid values of $b$ these two effects balance out and in this range $\bpi_+(x)$ is larger for smaller values of $x$.\\
The sensitivities are strongest for smaller values of $x$. The derivative with respect to $a$ is positive when $a=1$, indicating that increasing $a$ increases the stationary density in phase~1. The derivative with respect to $b$ has mixed sign: it is negative for small values of $b$ and positive for larger values of $b$. This reflects the two competing effects of increasing $b$: on the one hand, it increases the frequency of transitions into phase~1; on the other hand, it shifts probability mass towards larger values of $x$.

\begin{figure}[h]
\centering
\includegraphics[scale=0.32]{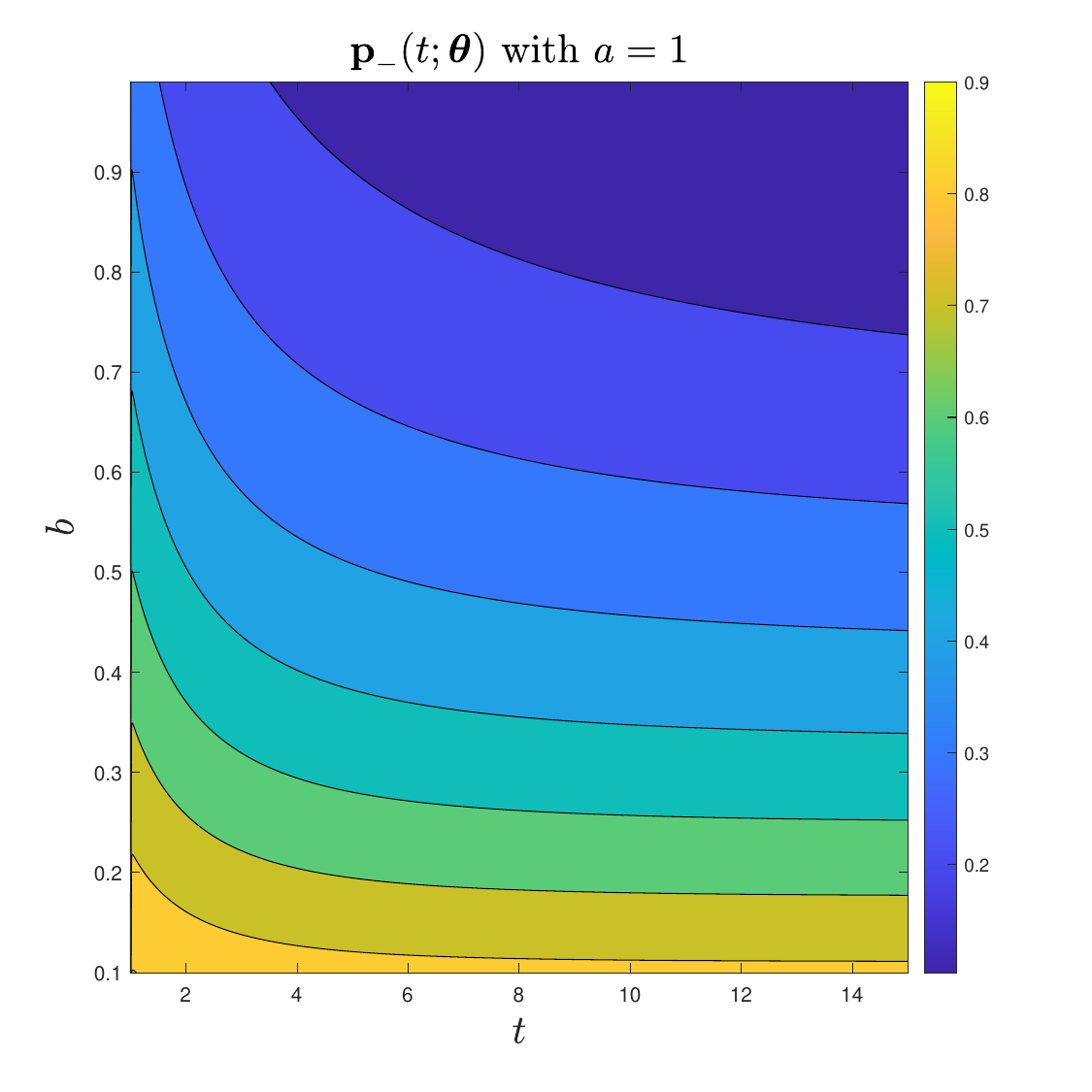}
\includegraphics[scale=0.32]{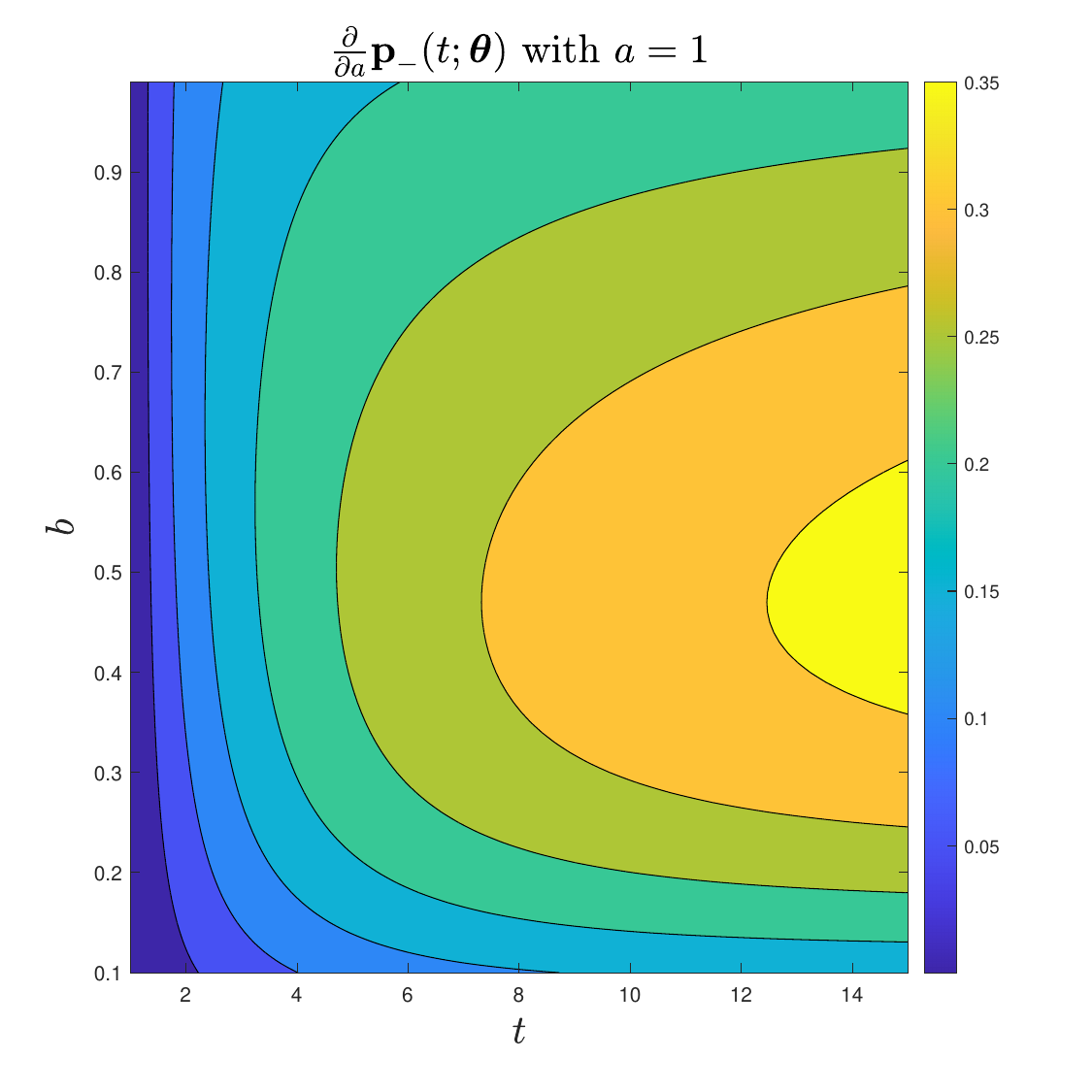}
\includegraphics[scale=0.32]{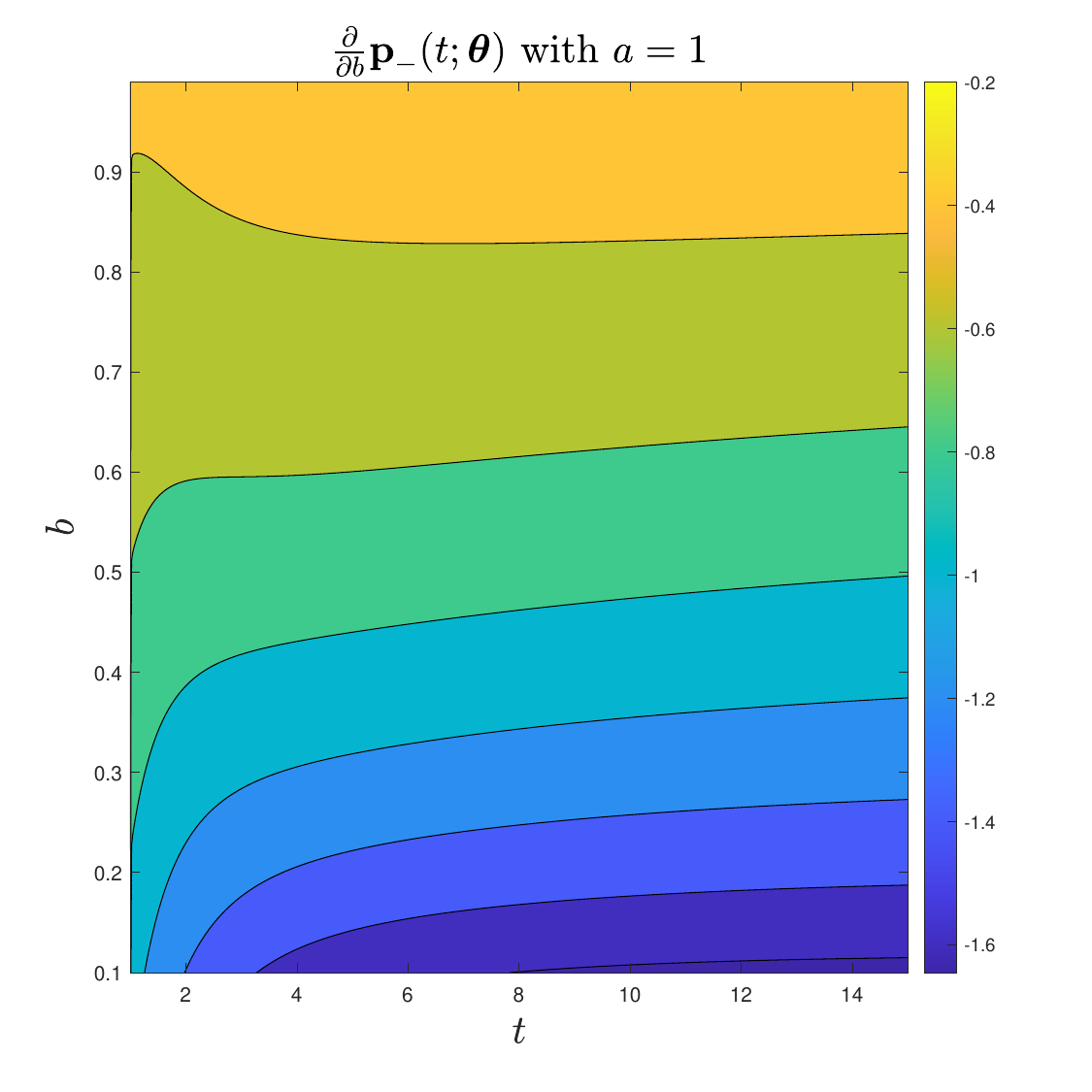}
\caption{Derivatives of the transient quantities of the process $\{( X(t),\varphi(t)):t\geq 0\}$: ${\bf p}_-(t;\btheta)$, $\frac{\partial}{\partial a}{\bf p}_-(t;\btheta)$, and $\frac{\partial}{\partial b}{\bf p}_-(t;\btheta)$, for $0.1\leq b\leq 0.99$ and $1\leq t\leq 15$, assuming $(X(0),\varphi(0))=(1,2)$, with the minimum time $t=1$ required to reach level $0$ given start in level $z=1$.
}
\label{fig:ex1_transient1}
\end{figure}

\begin{figure}[h]
\centering
\includegraphics[scale=0.32]{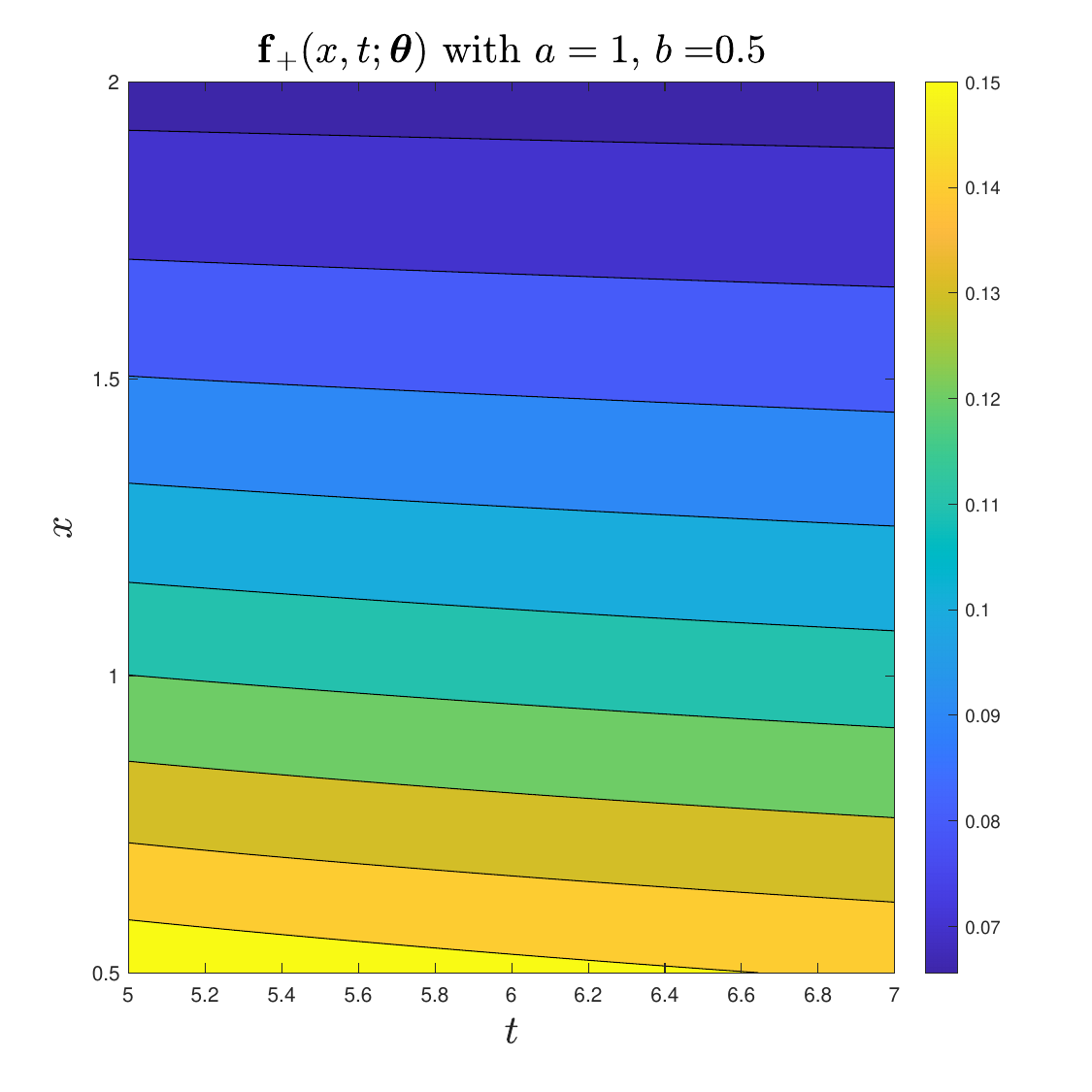}
\includegraphics[scale=0.32]{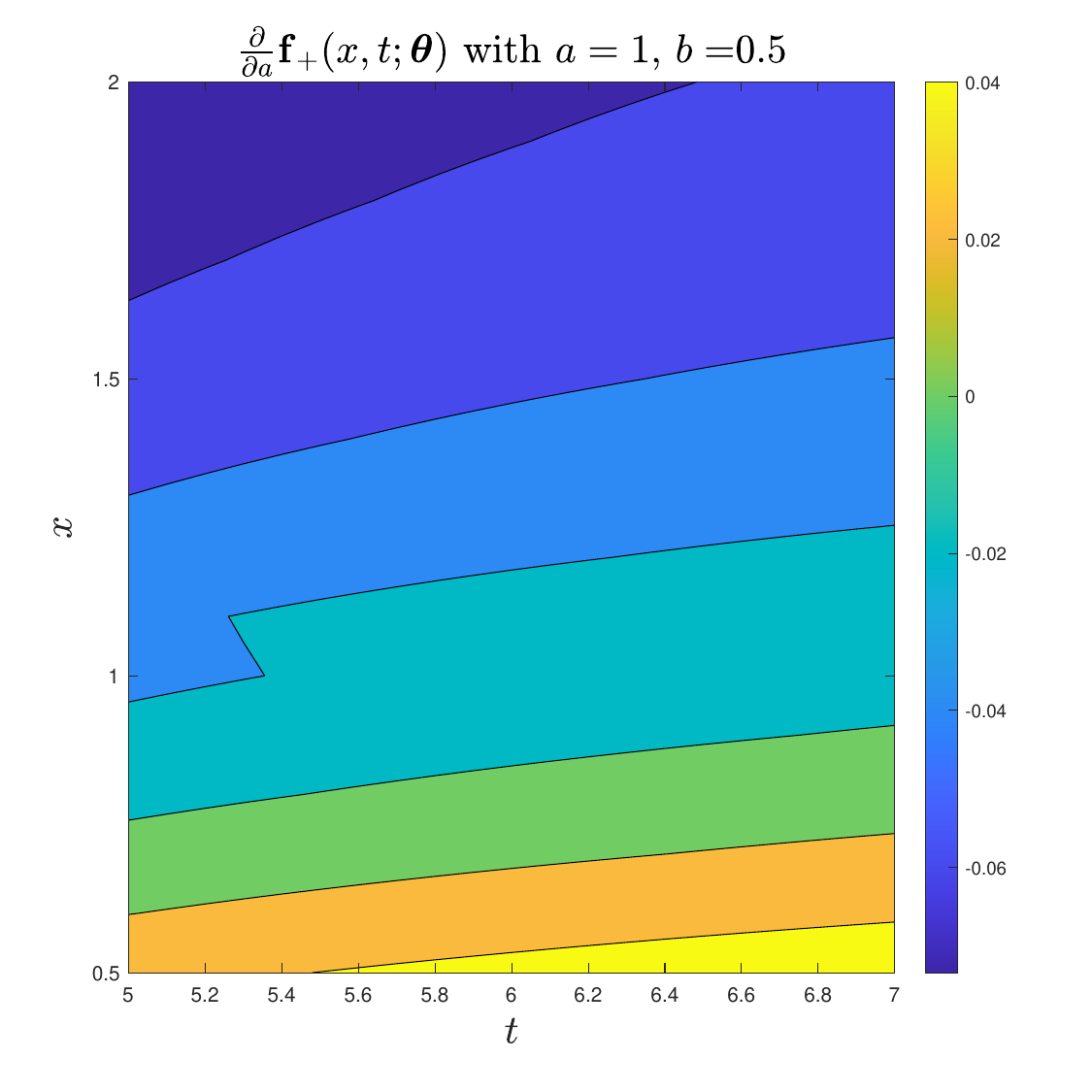}
\includegraphics[scale=0.32]{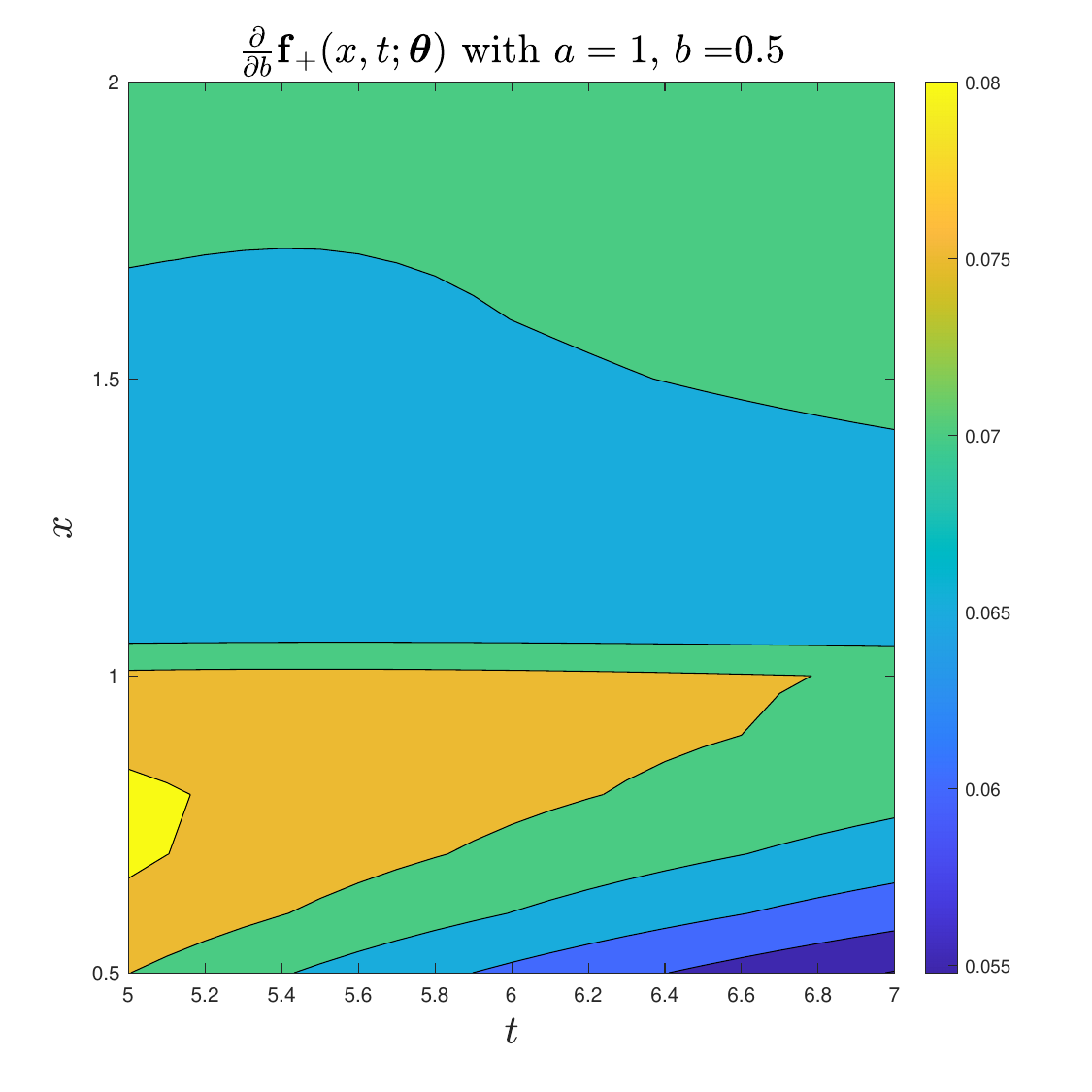}
\caption{Derivatives of the transient quantities of the process $\{( X(t),\varphi(t)):t\geq 0\}$: ${\bf f}_+(x,t;\btheta)$, $\frac{\partial}{\partial a}{\bf f}_+(x,t;\btheta)$, $\frac{\partial}{\partial b}{\bf f}_+(x,t;\btheta)$; assuming $(X(0),\varphi(0))=(1,2)$ with $a=1$, $b=0.5$, $x=[0.5, 2]$, $5<t<7$.
}
\label{fig:ex1_transientNEW}
\end{figure}

\begin{figure}[h]
\centering
\includegraphics[scale=0.32]{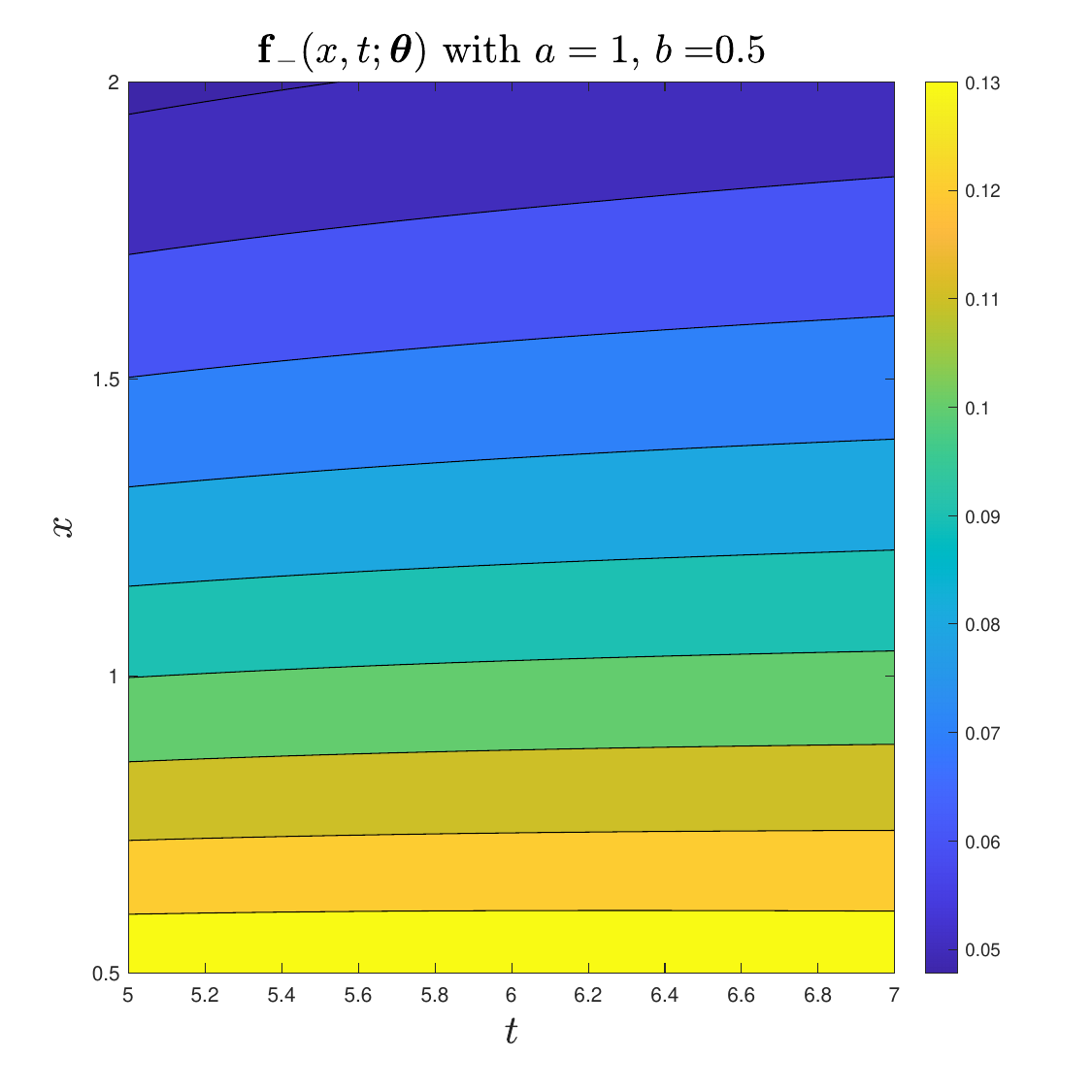}
\includegraphics[scale=0.32]{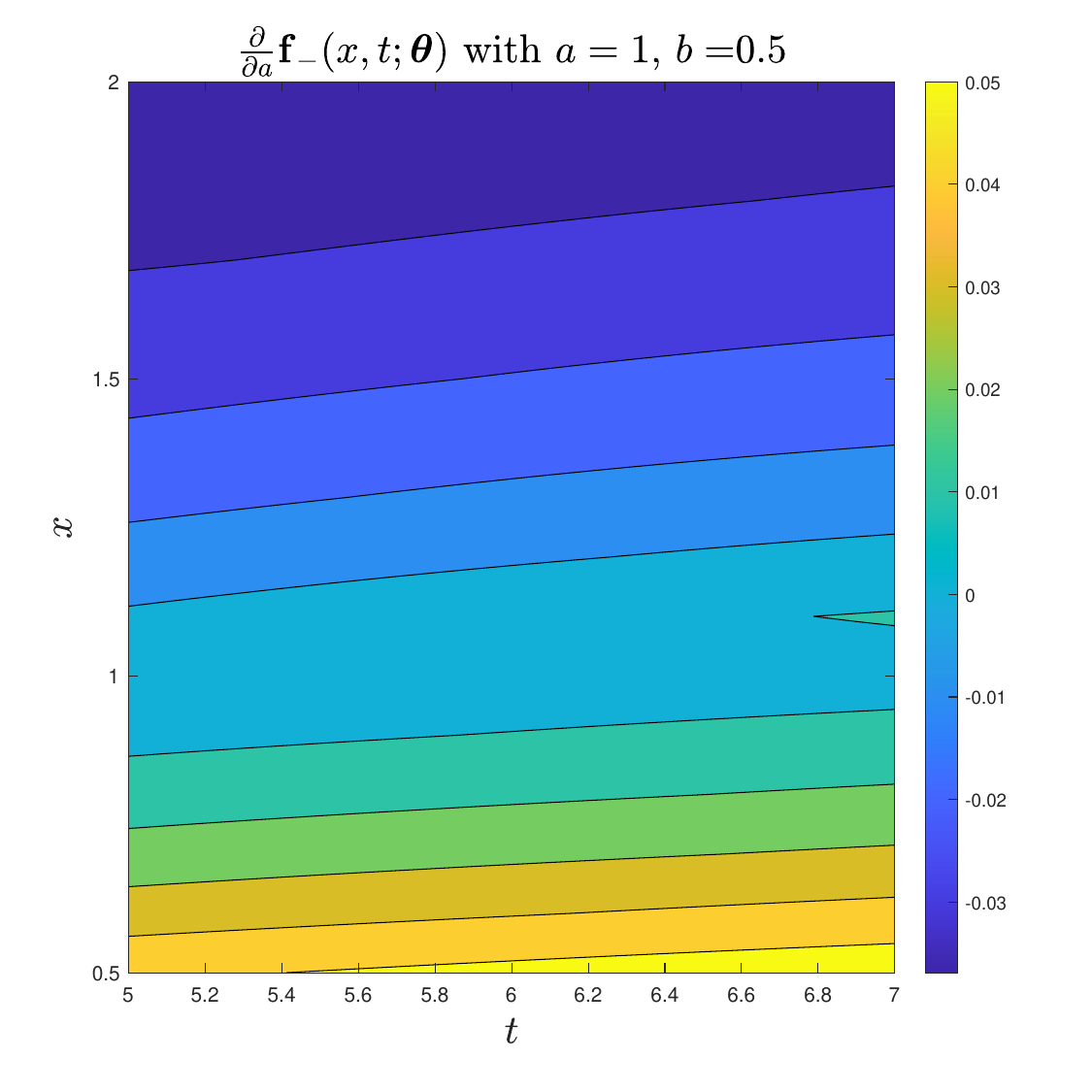}
\includegraphics[scale=0.32]{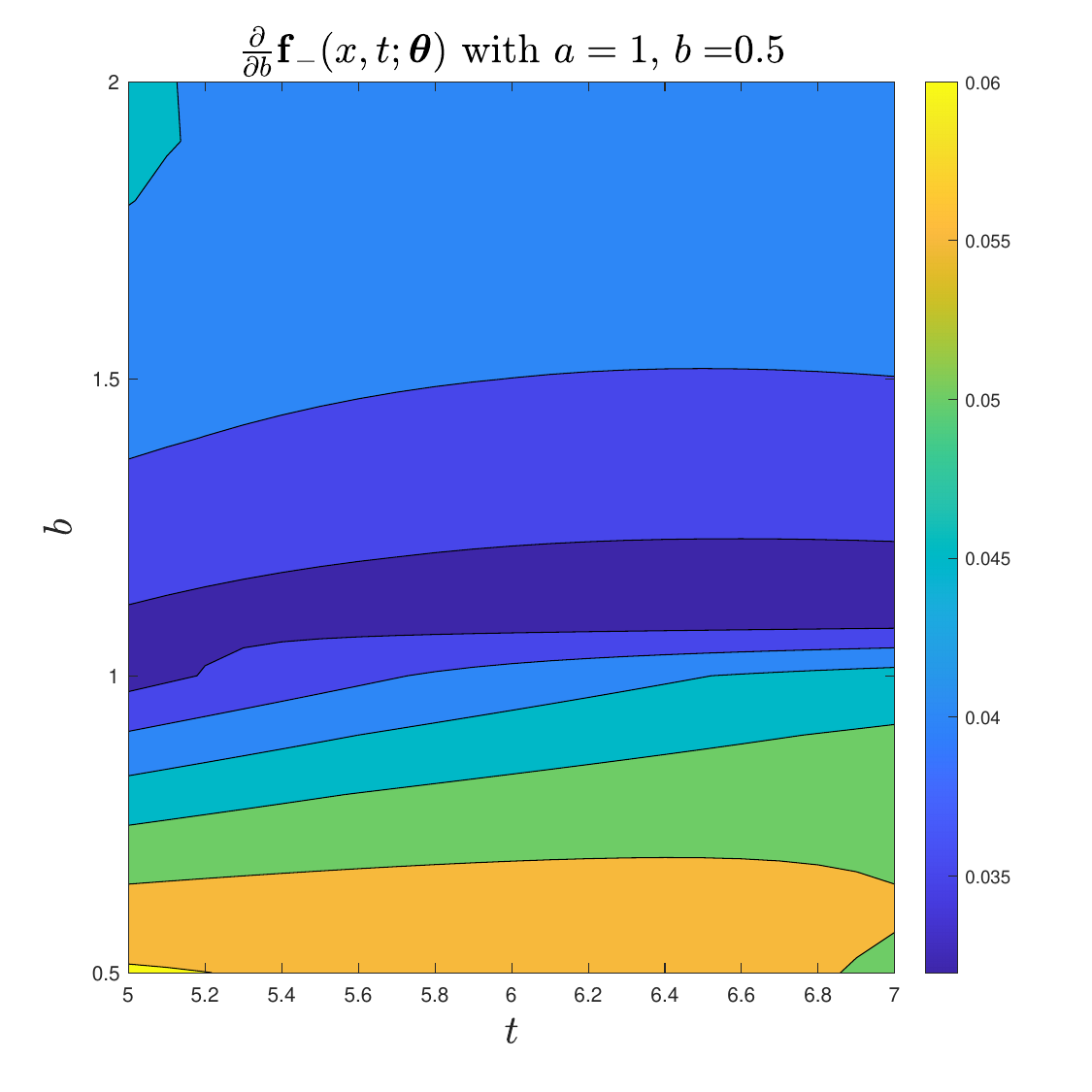}
\caption{Derivatives of the transient quantities of the process $\{( X(t),\varphi(t)):t\geq 0\}$: ${\bf f}_-(x,t;\btheta)$, $\frac{\partial}{\partial a}{\bf f}_-(x,t;\btheta)$, $\frac{\partial}{\partial b}{\bf f}_-(x,t;\btheta)$; assuming $(X(0),\varphi(0))=(1,2)$ with $a=1$, $b=0.5$, $x=[0.5, 2]$, $5<t<7$.
}
\label{fig:ex1_transientNEW2}
\end{figure}

In Figure \ref{fig:ex1_transient1}, we observe that the transient boundary probability at state $(0,2)$ stabilises rapidly, with very little change for $t \gtrsim 5$, indicating that the process is close to its stationary regime. The signs of the sensitivities are consistent with the stationary case shown in Figure \ref{fig:ex1_stationary1}: the derivative with respect to $a$ is positive, while the derivative with respect to $b$ is negative.
The effect of changes in $a$ is most pronounced for intermediate values of $b$ and during the transient phase, just before convergence to stationarity. This reflects the fact that parameter changes have the strongest impact while the distribution is still evolving, whereas their influence diminishes once the process has effectively reached equilibrium.

In Figures \ref{fig:ex1_transientNEW} and \ref{fig:ex1_transientNEW2}, the left graphs show that the transient density is higher for smaller values of $x$ and is already stable in time for the chosen parameters $a=1$ and $b=0.5$.\\
The derivative with respect to $a$ (middle graphs) is positive for $x$ below the initial $z=1$ and negative for $x$ above the initial $z=1$ and exhibits a monotone behaviour in both $x$ and $t$, with its magnitude being rather small. The derivative with respect to $b$ (right panels) is positive, decreases in time for most values of $x$, and displays rather low dependence on $x$, reflecting a uniform effect across different levels of $x$.

\subsection{Hydro-power generation system example}
\label{ex:appl}
Consider a SFM $\{(X(t),\varphi(t)):t\geq 0\}$ constructed by Bean, O'Reilly and Sargison in~\cite{hydropaper} for the management of hydro-power generation systems, with $X(t)\in[0,1]$ modelling the deterioration level where $0=$ brand new and $1=$ in need of replacement, and $\varphi(t)\in\mathcal{S}=\{1,2,3,4,5,6\}=\mathcal{S}_+\cup\mathcal{S}_-$ modelling the operational phases, where $\mathcal{S}_+=\{1,2,3,4,5\}$ contains on-design (energy is produced at a standard rate), off-design (energy is produced at a non-standard rate which causes the system to deteriorate more quickly), start (to commence the production of the energy), stop (to stop the system), and idle phases, respectively, and $\mathcal{S}_-=\{6\}$ contains the maintenance phase.

Further, generator ${\bf T}$ is used to model an operational strategy, so that different ${\bf T}$ would represent alternative strategies under consideration, with one of the strategies in~\cite{hydropaper},
 obtained from the operational data, corresponding to
\begin{equation}
\label{e-T}
{\bf T}=
\left[
\begin{array}{cccccc}\setlength{\arraycolsep}{0.1mm}
 -83.9&      40.7&          0&     43.2&          0&          0\\
     180.2&   -262.9&          0&     82.7&          0&          0\\
     1085.5&    314.9&    -1440.0&     39.6&          0&          0\\
          0&          0&    17.6&    -1440.0&     1422.4&          0\\
          0&         0&     62.3&          0&    -64.8&     2.5\\
          0&          0&          0&          0&     39.4&    -39.4
\end{array}
\right].
\end{equation}
We note that the transition rates recorded by the generator ${\bf T}$ correspond to transitions observed while the level variable is within the open interval $X(t)\in(0,1)$ (and so additional parameters, discussed below, are used to model the behaviour at the boundaries $0$ and $1$).

Finally, the rates $c_i$, $i\in\mathcal{S}$, in~\cite{hydropaper} are estimated to be
\begin{equation*}
[c_1,c_2,\ldots,c_6]=[0.004,0.017,0.020,0.020,0.001,-0.01]
\end{equation*}
based on expert, experience-based opinion provided by a hydro-power company for a typical power station, and the expectation that the deterioration level would increase at a rate of 0.1\% per month if left idle, and at a rate of 0.4\% per month if left in the on-design phase, and so on, and that the deterioration level would decrease at a rate of 1\% per month during the maintenance phase.

Denote parameters $\theta_i=|c_i|$ for $i=1,\ldots,6$, collected in vector $\btheta=[\theta_1,\theta_2,\ldots,\theta_6]$, here
\begin{eqnarray}\label{eq:ex2par}
\btheta=
[0.004,0.017,0.020,0.020,0.001,0.01].
\end{eqnarray}

By~\cite{hydropaper}, the Laplace transform of the distribution of the lifetime $\gamma$ of the system defined by
\begin{equation*}
[{\bf L}(s;\btheta)]_{i}=E[e^{-s{\gamma}}; \gamma<\infty \mid X(0)=0,\varphi (0)=i]
=
\int_{t=0}^{\infty}
e^{-st}
[{\bf h}(t;\btheta)]_i
dt,
\end{equation*}
is given by
\begin{eqnarray}\label{eq_L}
{\bf L}(s;\btheta)= {\bf W}(s;\btheta)
\bar{\bf P}_{++}(s) {\bf 1},
\end{eqnarray}
where
\begin{eqnarray*}
{\bf W}(s;\btheta)&=&({\bf I}-{\bf G}_{+-}^{(0,1)}(s;\btheta)\bar{\bf P}_{-+})^{-1}{\bf H}_{++}^{(0,1)}(s;\btheta)
\end{eqnarray*}
is the Laplace transform of the time taken to first reach level $1$, given start from level~$0$ in some phase in $\mathcal{S}_+$, with the probability matrix
$$\check{\bf P}_{-+}=\left[\begin{array}{ccccc}0&0&0&0&1\\\end{array}\right]$$
corresponding to the transition from maintenance to idle phase upon hitting level $0$; and $\bar{\bf P}_{++}(s)$ is the Laplace transform of the time spent at the upper boundary given by
\begin{eqnarray*}
\bar{\bf P}_{++}(s)&=&\hat{\bf P}_{++}+\hat{\bf P}_{+0}(s{\bf I}-\hat{{\bf T}}_{00})^{-1}\hat{\bf T}_{0+},
\end{eqnarray*}
with
\begin{displaymath}
\hat{\bf P}_{++}=\left[
\begin{array}{ccccc}
0&0&0&0&1\\
0&0&0&0&1\\
0&0&0&0&1\\
0&0&0&0&1\\
0&0&0&0&0\\
\end{array}\right],
\mbox{ }
\hat{\bf P}_{+0}=\left[
\begin{array}{ccccc}
0\\
0\\
0\\
0\\
1\\
\end{array}\right],
\end{displaymath}
and $\hat{\bf T}_{0-}=[0]$, $[\hat{\bf T}_{00}]=[-64.8]$ and $\hat{\bf T}_{0+}=[0,0,0,0,64.8]$.

Denote by
\begin{eqnarray*}
\balpha=[\alpha_i]_{i=1,\ldots,5},
\alpha_i=\mathbb{P}(\varphi(0)=i),i=1,\ldots,5,
\end{eqnarray*}
the distribution of the initial phase, and let $h(t;\btheta)=\balpha {\bf h}(t;\btheta)$ be the density corresponding to
\begin{equation*}
\balpha{\bf L}(s;\btheta)=E_{\balpha}[e^{-s{\gamma}}; \gamma<\infty \mid X(0)=0]
=
\int_{t=0}^{\infty}
e^{-st}
\balpha
{\bf h}(t;\btheta)
dt
=
\int_{t=0}^{\infty}
e^{-st}
h(t;\btheta)
dt,
\end{equation*}
interpreted as the density of the lifetime.

By above and the matrix derivative calculus~\eqref{eq:mproduct}-\eqref{eq:minverse} , it then follows that
\begin{eqnarray*}
\frac{\partial}{\partial \btheta}
\left(
\balpha {\bf L}(s;\btheta)
\right)&=&
\balpha \frac{\partial}{\partial \btheta}
{\bf L}(s;\btheta)
,\\
\frac{\partial}{\partial \btheta}
{\bf L}(s;\btheta)&=&
\frac{\partial {\bf W}(s;\btheta)}{\partial \btheta}
\times
\left(
{\bf I}\otimes
\bar{\bf P}_{++}(s) {\bf 1}
\right)
\end{eqnarray*}
where
\begin{eqnarray*}
\frac{\partial {\bf W}(s;\btheta)}{\partial \btheta}
&=&
\frac{\partial ({\bf I}-{\bf G}_{+-}^{(0,1)}(s;\btheta)\bar{\bf P}_{-+})^{-1}}{\partial \btheta}
\times
\left(
{\bf I}
\otimes
{\bf H}_{++}^{(0,1)}(s;\btheta)
\right)
\\
&&
+
({\bf I}-{\bf G}_{+-}^{(0,1)}(s;\btheta)\bar{\bf P}_{-+})^{-1}
\times
\frac{\partial {\bf H}_{++}^{(0,1)}(s;\btheta)}{\btheta}
,\\
\frac{\partial ({\bf I}-{\bf G}_{+-}^{(0,1)}(s;\btheta)\bar{\bf P}_{-+})^{-1}}{\partial \btheta}
&=&
({\bf I}-{\bf G}_{+-}^{(0,1)}(s;\btheta)
\bar{\bf P}_{-+})^{-1}
\times
\frac{\partial {\bf G}_{+-}^{(0,1)}(s;\btheta)}{ \partial \btheta}
\times\left(
{\bf I}\otimes \bar{\bf P}_{-+}
\right)
\\
&&
\times
\left(
{\bf I}
\otimes ({\bf I}-{\bf G}_{+-}^{(0,1)}(s;\btheta)\bar{\bf P}_{-+})^{-1}
\right)
,
\end{eqnarray*}
and the remaining quantities can be evaluated using Theorems~\ref{th:PsiDer}-\ref{th:KsiDer} and Lemma~\ref{lem:GHder}, where we apply
\begin{eqnarray*}
\frac{\partial}{\partial\btheta}{\bf Q}(s;\btheta)&=&
\left[
\begin{array}{cc}
\frac{\partial}{\partial\btheta}{\bf Q}_{++}(s;\btheta)&\frac{\partial}{\partial\btheta}{\bf Q}_{+-}(s;\btheta)\\
\frac{\partial}{\partial\btheta}{\bf Q}_{++}(s;\btheta)&\frac{\partial}{\partial\btheta}{\bf Q}_{+-}(s;\btheta)
\end{array}
\right]
,
\\
\frac{\partial}{\partial\btheta}{\bf Q}_{++}(s;\btheta)
&=&
\frac{\partial ({\bf C}_+(s;\btheta))^{-1}}{\partial\btheta}
\times
\left(
{\bf I}\otimes
\left(
{\bf T}_{++}(\btheta)-s{\bf I}
\right)
\right)
,\\
\frac{\partial}{\partial\btheta}{\bf Q}_{+-}(s;\btheta)
&=&
\frac{\partial ({\bf C}_+(s;\btheta))^{-1}}{\partial\btheta}
\times
\left(
{\bf I}\otimes
{\bf T}_{+-}(\btheta)
\right)
,\\
\frac{\partial}{\partial\btheta}{\bf Q}_{--}(s;\btheta)
&=&
\frac{\partial |{\bf C}_-(s;\btheta)|^{-1}}{\partial\btheta}
\times
\left(
{\bf I}\otimes
\left(
{\bf T}_{--}(\btheta)-s{\bf I}
\right)
\right)
,\\
\frac{\partial}{\partial\btheta}{\bf Q}_{-+}(s;\btheta)
&=&
\frac{\partial |{\bf C}_-(s;\btheta)|^{-1}}{\partial\btheta}
\times
\left(
{\bf I}\otimes
{\bf T}_{-+}(\btheta)
\right)
,\\
\frac{\partial ({\bf C}_+(s;\btheta))^{-1}}{\partial\btheta}
&=&
-({\bf C}_+(s;\btheta))^{-1}\times
\frac{\partial {\bf C}_+(s;\btheta)}{\partial\btheta}
\times
({\bf I}\otimes
({\bf C}_+(s;\btheta))^{-1} )
,\\
\frac{\partial |{\bf C}_-(s;\btheta)|^{-1}}{\partial\btheta}
&=&
-|{\bf C}_-(s;\btheta)|^{-1}\times
\frac{\partial |{\bf C}_-(s;\btheta)|}{\partial\btheta}
\times
({\bf I}\otimes
|{\bf C}_-(s;\btheta)|^{-1} )
,
\end{eqnarray*}
and note that here,
for $i=1,\ldots,5$,
\begin{eqnarray*}
\frac{\partial ({\bf C}_+(s;\btheta))^{-1}}{\partial\theta_i}
&=&
-({\bf C}_+(s;\btheta))^{-1}
\times
\mbox{diag}({\bf e}_i)
\times
({\bf C}_+(s;\btheta))^{-1}
=
-\mbox{diag}({\bf e}_i)
c_i^{-2},
\end{eqnarray*}
and, since $\theta_6=|c_6|$,
\begin{eqnarray*}
\frac{\partial |{\bf C}_-(s;\btheta)|^{-1}}{\partial\theta_6}
&=&
-|{\bf C}_-(s;\btheta)|^{-1}\times
\mbox{diag}({\bf e}_6)
\times
|{\bf C}_-(s;\btheta)|^{-1}
=
-\mbox{diag}({\bf e}_6)|c_6|^{-2}
=
-|c_6|^{-2}\
.
\end{eqnarray*}


Since
\begin{eqnarray*}
\balpha
\frac{\partial}{\partial \btheta}
{\bf L}(s;\btheta)
&=&
\frac{\partial}{\partial \btheta}
\int_{t=0}^{\infty}
e^{-st}
\balpha
{\bf h}(t;\btheta)
dt
=
\int_{t=0}^{\infty}
e^{-st}
\frac{\partial}{\partial \btheta}
\balpha
{\bf h}(t;\btheta)
dt
=
\int_{t=0}^{\infty}
e^{-st}
\frac{\partial}{\partial \btheta}
h(t;\btheta)
dt
,
\end{eqnarray*}
by numerically inverting $\balpha \frac{\partial}{\partial \btheta}
{\bf L}(s;\btheta)$ using the algorithm by Horv{\'a}th et al.~\cite{horvath2020numerical}, we then obtain the derivative of the density of the lifetime, $\frac{\partial}{\partial \btheta}h(t;\btheta)$. The output is presented in Figure~\ref{fig:ex2_lifetime}.

In Figure \ref{fig:ex2_lifetime}, on the first plot, we present the distribution of the first hitting time of level $X(t)=1$ for the SFM used in the management of a hydro-power generation system. The distribution is unimodal and symmetric, with median around $t=244$.
The effect on the lifetime is similar across all parameters associated with the increasing states $i=1,\ldots,5$. In particular, the derivatives are approximately linear and decreasing in $t$, with comparable magnitudes. This is also reflected in the semi-relative sensitivities, $\theta_i\times\partial h(t,\boldsymbol{\theta})/\partial \theta_i$, which are of similar size across $i$ and therefore indicate a comparable impact of each parameter on the lifetime. At the median $t=244$, the derivatives vanish, indicating no local sensitivity to parameter changes. Away from the median, the effect increases in magnitude: it is positive for $t<244$ and negative for $t>244$, reflecting a redistribution of probability mass from one side of the distribution to the other.
The maintenance parameter has the opposite effect on the lifetime, which is consistent with its role in the model.

\begin{figure}[H]
\centering
\includegraphics[scale=0.32]{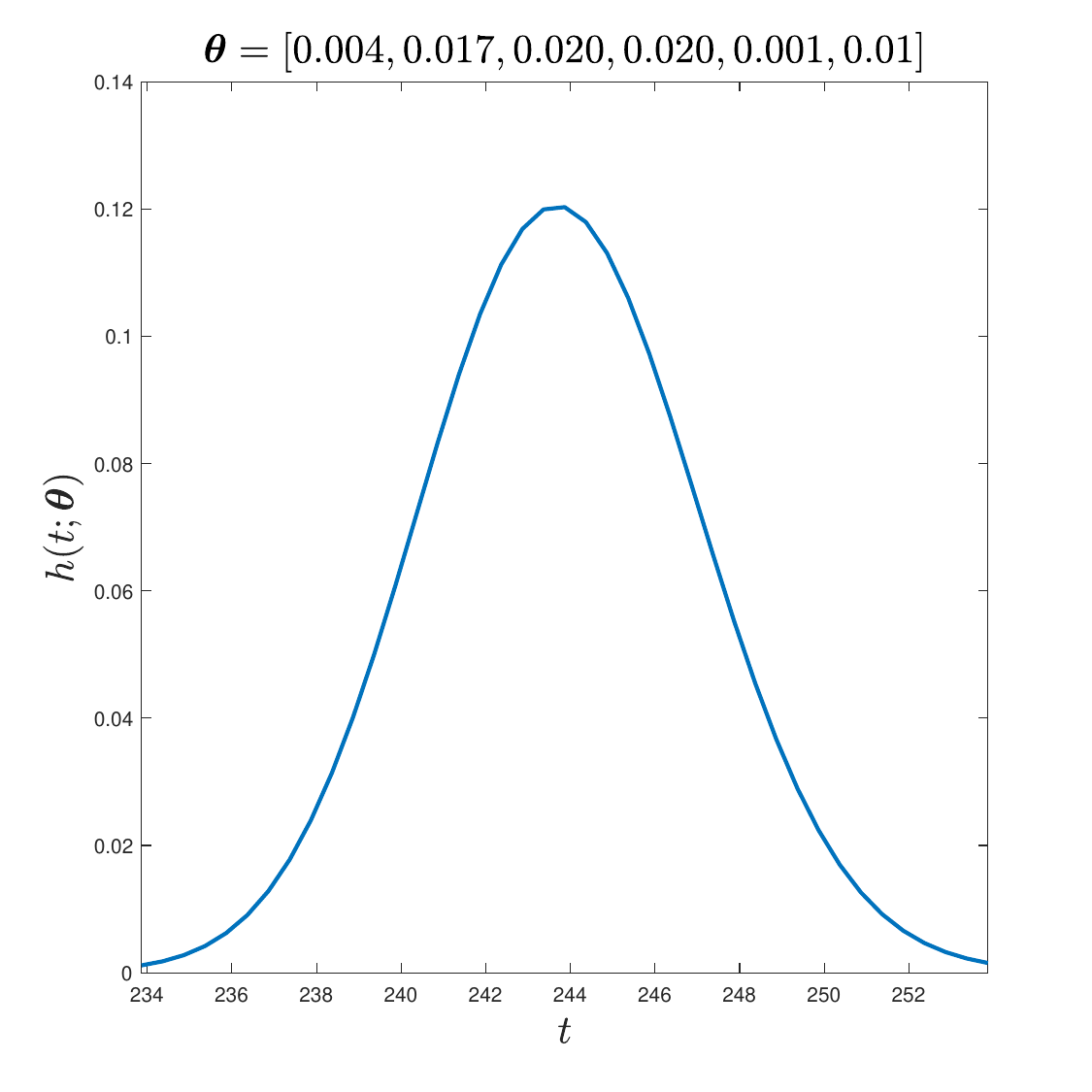}
\includegraphics[scale=0.32]{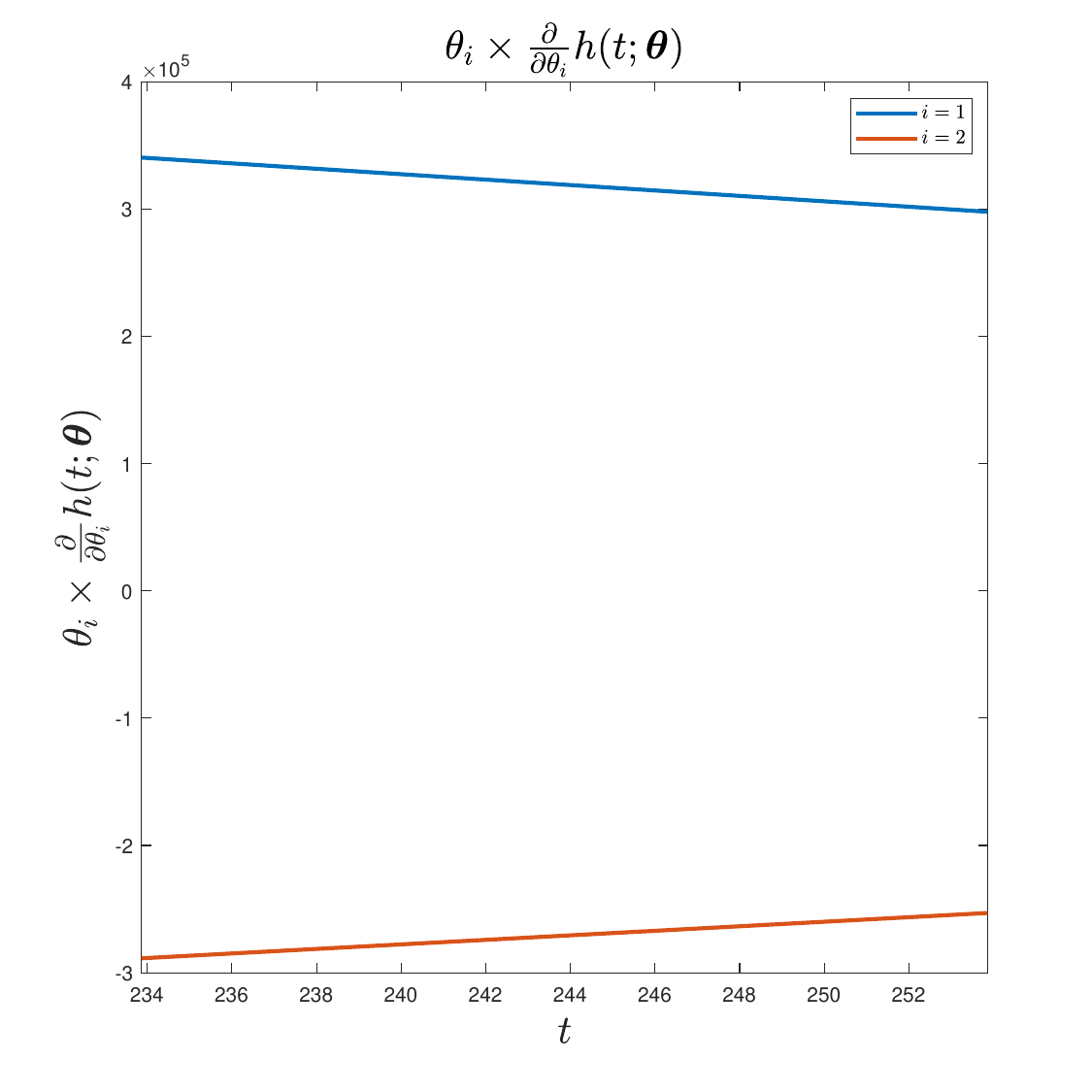}
\includegraphics[scale=0.32]{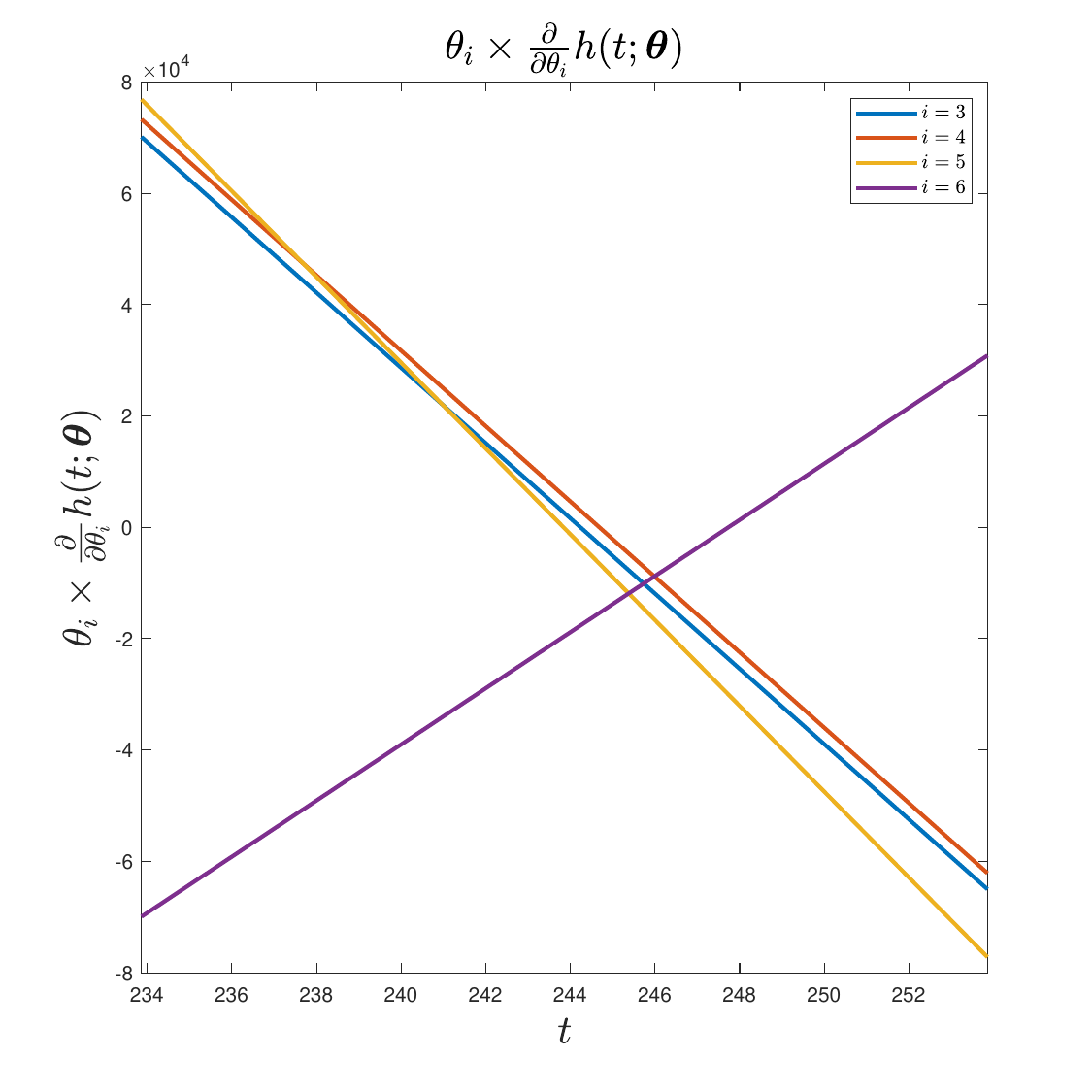}
\caption{Density $h(t,\btheta)$ of the lifetime and its derivatives in Example~\ref{ex:appl}: (top row) $h(t,\btheta)$ evaluated using parameters $\btheta$ in~\eqref{eq:ex2par}; and $\theta_i\times\frac{\partial}{\partial \theta_i}h(t;\btheta)$ for $i=1,\ldots,6$ assuming the process starts from the idle phase $\varphi(0)=5$ at time zero, according to $\balpha=[0,0,0,0,1]$.
}
\label{fig:ex2_lifetime}
\end{figure}


\subsection{Insurer ruin example}
\label{ex:insurance}
Gerber in \cite{gerber1998on} considered the classical insurance risk process
\begin{equation*}\label{Eq.Classical risk model}
  R_t = x+ ct - \sum_{k=1}^{N_t}U_k,
\end{equation*}
where $x\geq 0 $ is the initial surplus, $c$ is the premium rate, $N_t$, $t \geq 0$, is a Poisson arrival process, and $\{U_k\}_{k=1}^\infty$ is a sequence of i.i.d.\ claim sizes independent of $N_t$.
This process is a key process in a actuarial theory.
For this process, we usually try to find or estimate the ruin probability
\begin{equation*}\label{Eq.Gerber-Shiu definition}
\psi(x)= \mathbb{P} \left( \tau <\infty\right),
\end{equation*}
  where
\begin{equation*}\label{Eq.Ruin time}
\tau = \inf\{t\geq 0 : R_t <0\}
\end{equation*}
 is the ruin time for a more general process of the following form
\begin{equation*}\label{Eq.Levy process with positive phase-type jumps mixture model}
 R_t = x+ ct +  \sum_{k=1}^{N^{(+)}_t}U_k^{(+)} - \sum_{k=1}^{N^{(-)}_t}U_k^{(-)}
\end{equation*}
The sequence $\{U_k^{(+)}\}$ models the sudden gains of an insurance company and it is a sequence of i.i.d.\ random variables which is independent of the Poisson process $N_t^{(+)}$ with intensity $\lambda^{(+)}$.
As before, $\{U_k^{(-)}\}_{k=1}^\infty$ is a sequence of i.i.d.\ claim sizes independent of $N_t^{(-)}$ which is in our case the Poisson process $N_t^{(-)}$ with intensity $\lambda^{(-)}$.
Moreover, the claims and gains are independent of each other. Similarly, the gain arrival process and the claim arrival process do not depend on each other either.

In our model, we assume that $U_k^{(+)}$ and $U_k^{(-)}$
have phase-type distributions with parameters $(\boldsymbol{\alpha}^{(+)}, \mathbf{M}^{(+)})$ and $(\boldsymbol{\alpha}^{(-)}, \mathbf{M}^{(-)})$, respectively.

Phase-type distributions are a computational vehicle for much of modern applied probability.
Recall that a distribution $F$ is said to be of phase-type with parameters $(\boldsymbol{\alpha}, \mathbf{M})$ if it is the same distribution
as the lifetime of some terminating Markov process starting from the vector $\boldsymbol{\alpha}=[\alpha_k]$, having finitely many states, say $N$, and having time homogeneous transition rates with subintensity $\mathbf{M}=[M_{ij}]$.
Namely,
$F(y)=1-\boldsymbol{\alpha}\exp\{\mathbf{M}y\} \mathbf{1}$,
for $y\geq 0$, where $\mathbf{1}$ is a column\-/vector of ones.
This gives also the form of the density $f(y)$ of $F$ as
$f(y)=\boldsymbol{\alpha}\exp\{\mathbf{M}y\}\mathbf{t}$,
where
$\mathbf{t}=[t_k]=-\mathbf{M} \mathbf{1}$.
Many well\-/known distributions are of phase\-/type; for example, the exponential distribution with intensity $\rho$ is PH in this case with $N=1$, $\mathbf{M}=\rho$, and $\boldsymbol{\alpha}=1$. Other more involved distributions that belong to the class of phase\-/types are the hyperexponential, the Erlang, and the Coxian; see \cite{asmussen-RP} for a detailed description.

When a model the positive and negative jumps of phase\-/type, a common approach is to spread them out as a succession of linear pieces of unit slope, the positive jumps with slope $+1$
and negative jumps with slope $-1$; see e.g.\ in \cite{badescu2005surplus,breuer2011generalised,Ramaswami2006}.
This procedure, which is called fluid embedding, requires adding supplementary states to the background process as many as cumulative amount of the phases appearing in our model
plus one which related to the premium state. The 1st state corresponds to the premium only. The other states correspond to the upward and downward jumps transformed into lines.
In other words, by above procedure, we construct an auxiliary SFM with the state space $\{1,2,\dots,m=N^{(+)}+N^{(-)}+1~\}$.
In the case when $U_k^{(+)}=0$ with superscripts $(-)$ skipped for convenience, we have $m=N+1$,
$c_1=c$ and $c_2=c_3=\ldots =c_m=-1$ and generator ${\bf T}=[T_{ij}]$ such that
\begin{equation*}
T_{11}=-\lambda,\quad
T_{1 (k+1)}=\lambda  {\alpha}_k,\quad
T_{(i+1) (j+1)}= M_{ij}, \quad
T_{(k+1)1}= t_k\quad \text{for $k,i,j=1,2,\ldots, N$.}
\end{equation*}

The key observation is that
\begin{equation*}
\psi(x)=\sum_{j=1}^m
[{\bf G}^{(x)}(0;\btheta)]_{1j}
\end{equation*}
since SFM ever touches zero iff the ruin happens for the risk process $R_t$.

For numerical purposes we decided to consider
claim size distribution being
the symmetric mixture of Erlang$(2,\theta_1)$ and Erlang$(2,\theta_2)$ taken from Example 4.2 of \cite{DICKSON1998251}.
More formally, in our case
$\lambda=1$, $c=4$, $N=4$ and $\mathbf{\alpha}=[1/2,0, 1/2, 0]$ and
\[\mathbf{M} =
\left[\begin{array}{cccc}
-\theta_1& \theta_1& 0& 0\\
0 &-\theta_1& 0& 0\\
0 &0 &-\theta_2& \theta_2 \\
0 &0 &0 &-\theta_2
\end{array}
\right].
\]
In \cite{DICKSON1998251}, $\theta_1=1$ and $\theta_2=2$.
Hence our Markov chain $\varphi(t)$ lives on $5$-dimensional state space partitioned as $\mathcal{S}=\mathcal{S}_+\cup\mathcal{S}_-=\{1\}\cup\{2,3,4,5\}$ and generator
\begin{eqnarray*}
{\bf T}
&=&
\left[
\begin{array}{c|cccc}
-1&1/2 &0 &1/2 &0 \\
\hline
0& -\theta_1 &\theta_1 &0 &0 \\
\theta_1& 0& -\theta_1 &0 &0\\
0& 0& 0& -\theta_2 &\theta_2\\
\theta_2& 0& 0& 0& -\theta_2
\end{array}
\right]
=
\left[
\begin{array}{cc}
{\bf T}_{++}&
{\bf T}_{+-}\\
{\bf T}_{-+}&
{\bf T}_{--}
\end{array}
\right],
\end{eqnarray*}
with
\begin{eqnarray*}
\frac{\partial}{\partial \theta_1}
{\bf T}
=
\left[
\begin{array}{c|cccc}
0&0 &0 &0 &0 \\
\hline
0& -1 &1 &0 &0 \\
1& 0& -1 &0 &0\\
0& 0& 0& 0 &0\\
0& 0& 0& 0& 0
\end{array}
\right]
,
\
\frac{\partial}{\partial \theta_2}
{\bf T}
=
\left[
\begin{array}{c|cccc}
0&0 &0 &0 &0 \\
\hline
0& 0 &0 &0 &0 \\
0& 0& 0 &0 &0\\
0& 0& 0& -1 &1\\
1& 0& 0& 0& -1
\end{array}
\right]
.
\end{eqnarray*}

Next, to evaluate $\frac{\partial}{\partial\btheta}\psi(x)$ we apply
\begin{eqnarray}
\frac{\partial}{\partial\btheta}\psi(x)
&=&
\sum_{j=1}^m
\left[
\frac{\partial}{\partial\btheta}
\left(
\bPsi(\btheta)e^{{\bf D}(\btheta)x}
\right)
\right]_{1j}
\label{eq:psix_der}
\end{eqnarray}
where, by~\eqref{eq:mproduct},
\begin{eqnarray}
\frac{\partial}{\partial\btheta}
\left(
\bPsi(\btheta)\times e^{{\bf D}(\btheta)x}
\right)
&=
\frac{\partial }{ \partial \btheta}
\bPsi(\btheta)
\times\left(
{\bf I}\otimes e^{{\bf D}(\btheta)x}
\right)
+
\bPsi(\btheta)
\times
\frac{\partial }{\partial \btheta}
e^{{\bf D}(\btheta)x}
,
\label{eq:mproduct_insure}
\end{eqnarray}
with $\frac{\partial }{ \partial \btheta}
\bPsi(\btheta)$ following by Theorem~\ref{th:PsiDer}. Further, by~\eqref{eq:exptrick},
\begin{eqnarray}
\int_{x=0}^{\infty}e^{-vx}
\frac{\partial
}
{ \partial \btheta}
e^{
{\bf D}\left(\btheta\right)x
}
dx
&=&
({\bf D}(\btheta)-v{\bf I})^{-1}
\times
\frac{\partial {\bf D}(\btheta)}{\partial \btheta}
\times
\left(
{\bf I}
\otimes ({\bf D}(\btheta)-v{\bf I})^{-1}
\right),
\label{eq:exptrickD}
\end{eqnarray}
with
\begin{eqnarray*}
\frac{\partial {\bf D}(\btheta)}{\partial \btheta}
&=&
\frac{\partial}{\partial \btheta}
{\bf Q}_{--}(\btheta)
+
\left(
\frac{\partial}{\partial \btheta}
{\bf Q}_{-+}(\btheta)
\right)
\times
\left(
{\bf I}
\otimes
\bPsi(\btheta)
\right)
+
{\bf Q}_{-+}(\btheta)
\times
\frac{\partial}{\partial \btheta}
\bPsi(\btheta)
\end{eqnarray*}
and
\begin{eqnarray*}
\frac{\partial}{\partial\theta_1}
{\bf Q}_{-+}(\btheta)
=
\left[
\begin{array}{c}
0\\1\\0\\0
\end{array}
\right]
,\
\frac{\partial}{\partial\theta_1}
{\bf Q}_{--}(\btheta)
=
\left[
\begin{array}{cccc}
-1&1&0&0\\
0&-1&0&0\\
0&0&0&0\\
0&0&0&0
\end{array}
\right]
,\\
\frac{\partial}{\partial\theta_2}
{\bf Q}_{-+}(\btheta)
=
\left[
\begin{array}{c}
0\\0\\0\\1\\
\end{array}
\right]
,\
\frac{\partial}{\partial\theta_2}
{\bf Q}_{--}(\btheta)
=
\left[
\begin{array}{cccc}
0&0&0&0\\
0&0&0&0\\
0&0&-1&1\\
0&0&0&-1\\
\end{array}
\right]
,
\end{eqnarray*}
and so $\frac{ \partial}{ \partial \btheta}
e^{
{\bf D}\left(\btheta\right)x
}$ follows by numerically inverting the Laplace transform in~\eqref{eq:exptrickD} using the algorithm by Horv{\'a}th et al.~\cite{horvath2020numerical}.

\begin{figure}[h]
\centering
\includegraphics[scale=0.30]{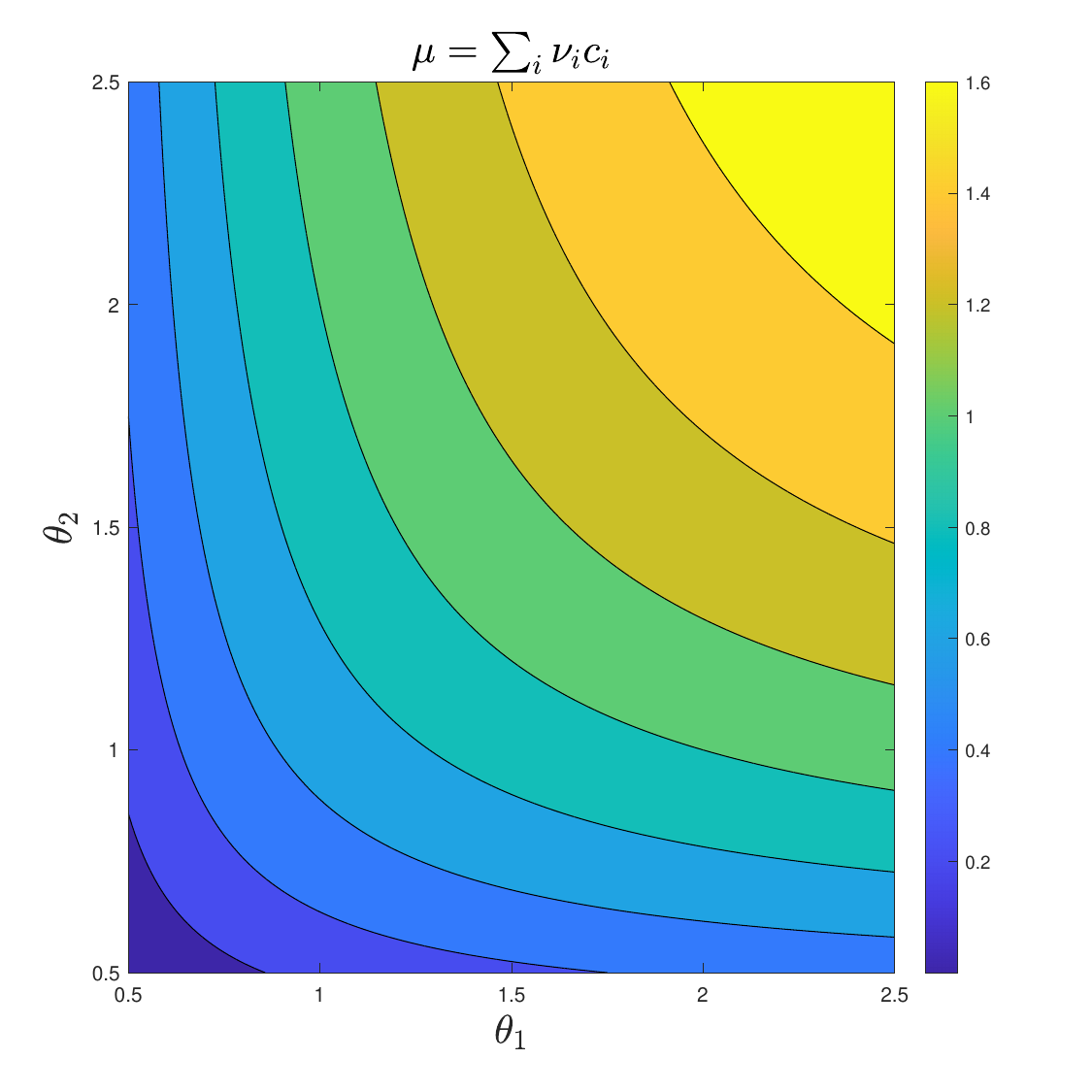}
\
\includegraphics[scale=0.30]{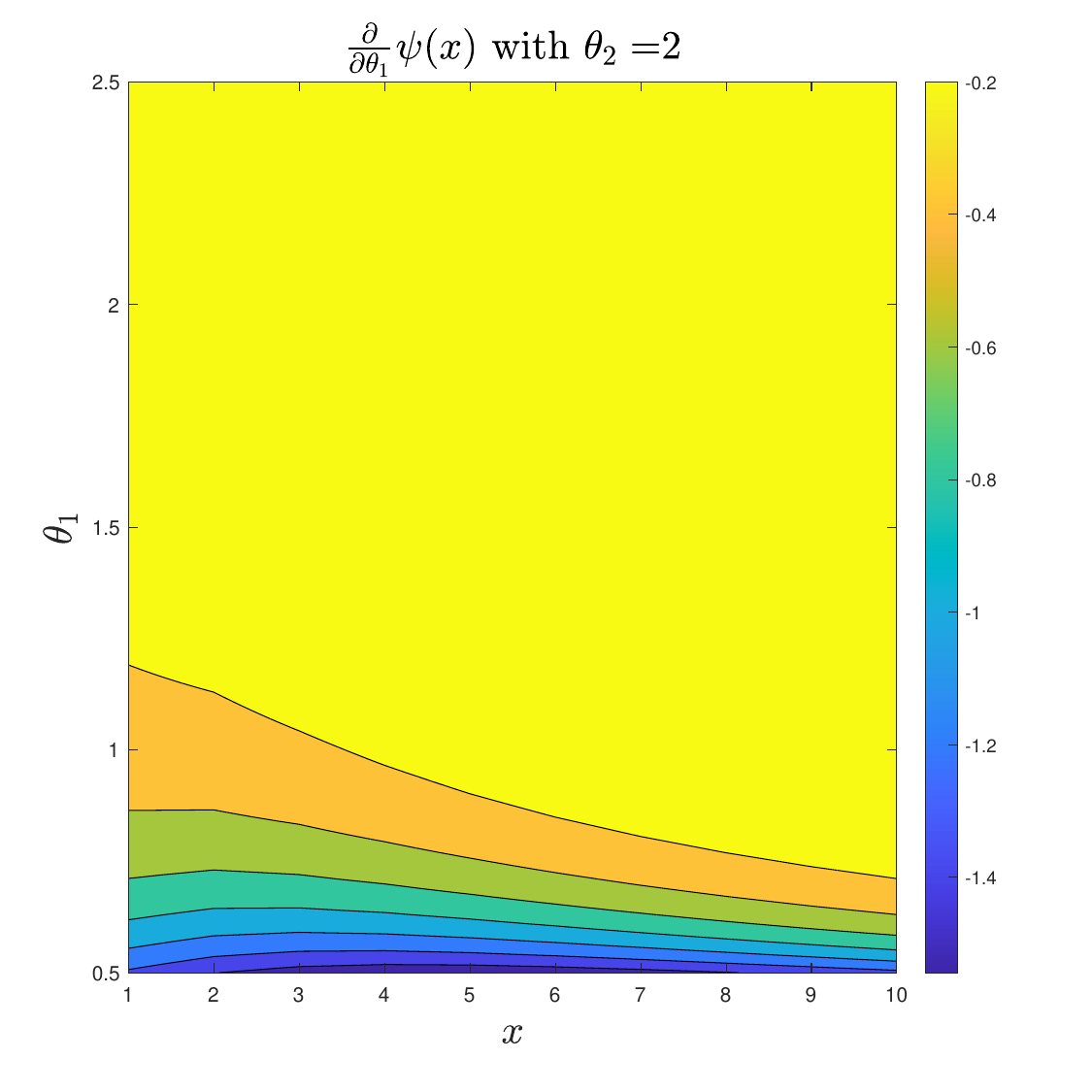}
\
\includegraphics[scale=0.30]{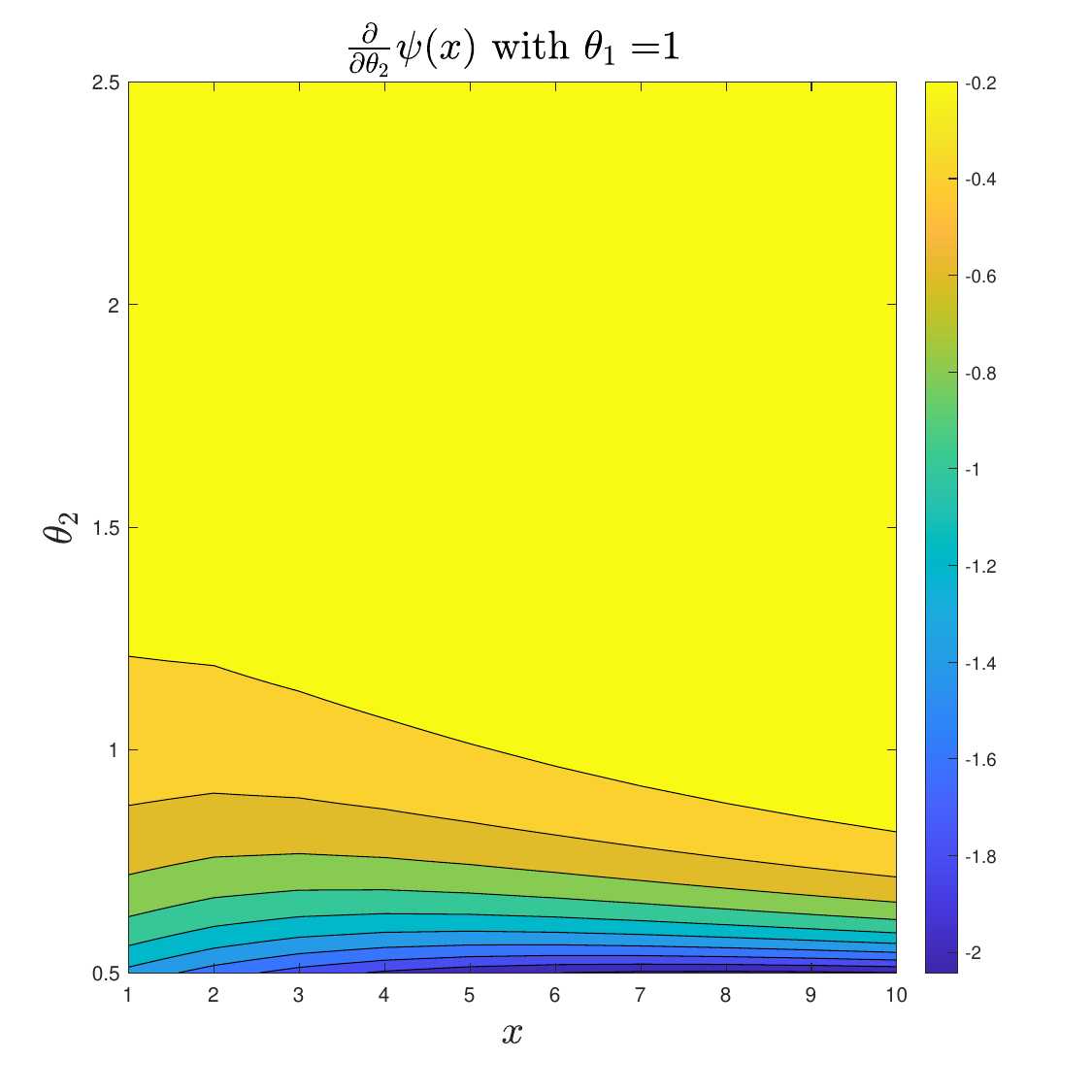}
\\
\includegraphics[scale=0.3]
{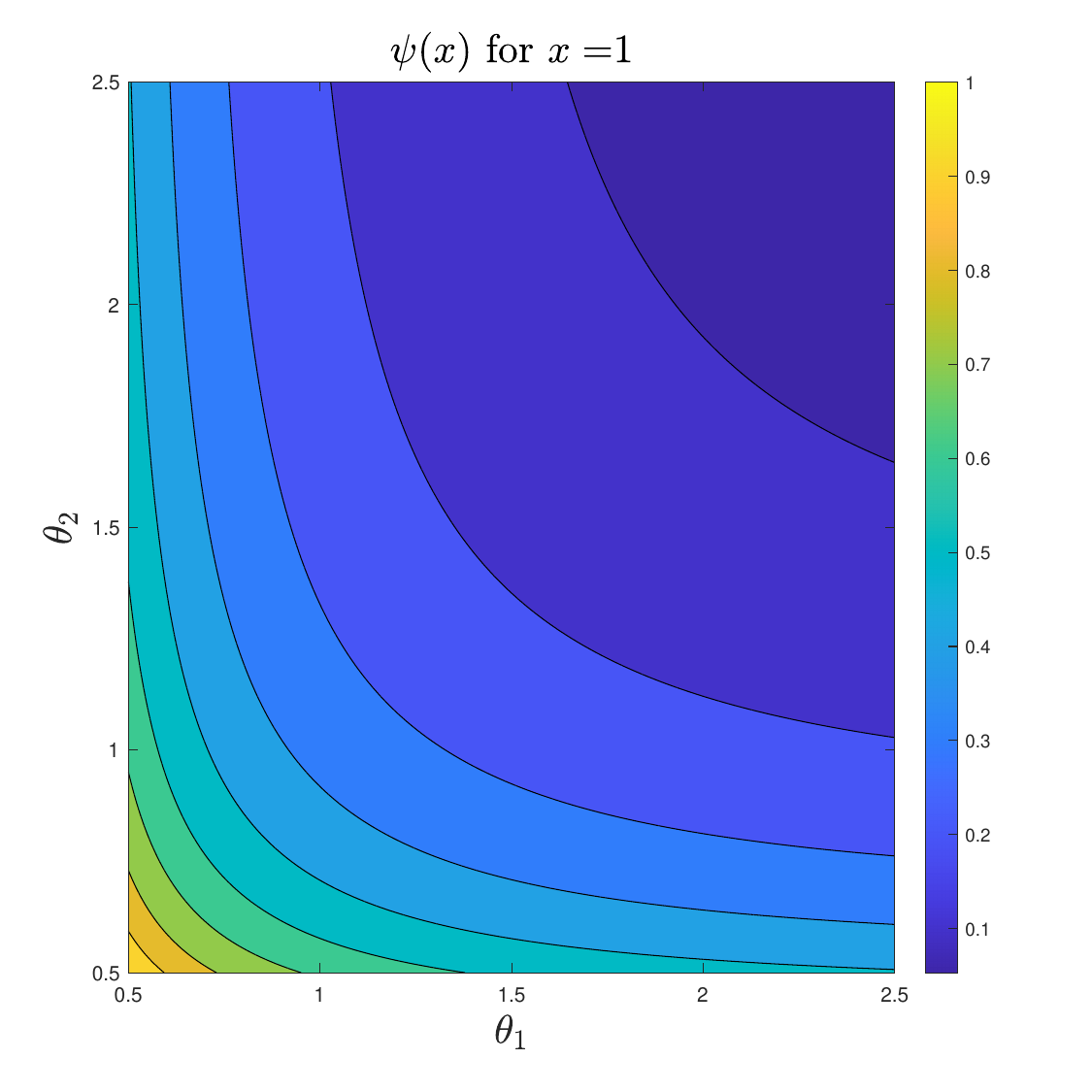}
\
\includegraphics[scale=0.3]
{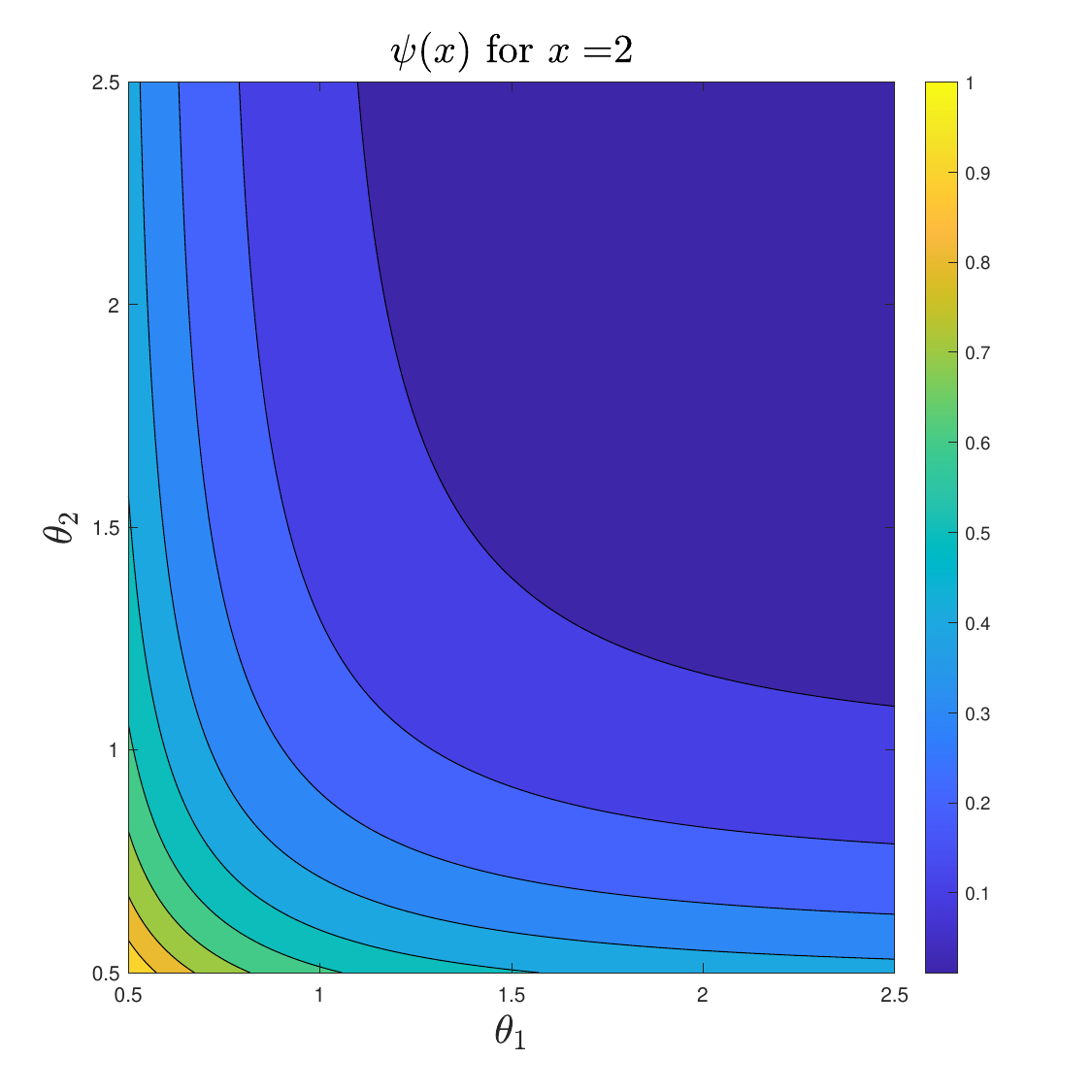}
\
\includegraphics[scale=0.3]
{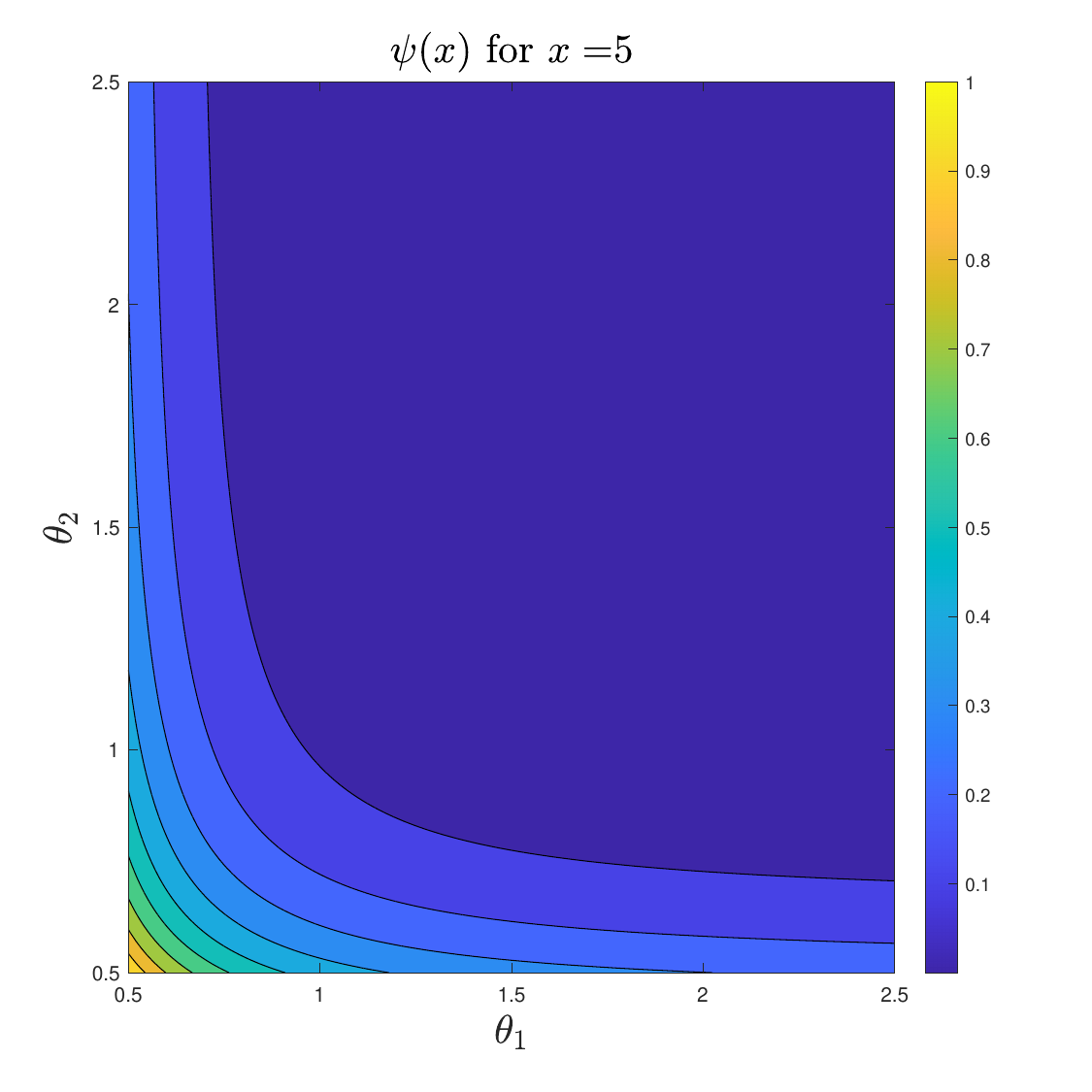}
\caption{Example~\ref{ex:insurance}: (top row) $\mu$, $\frac{\partial}{\partial\theta_1}\psi(x)$, and $\frac{\partial}{\partial\theta_2}\psi(x)$;  and (bottom row) $\psi(x)$.
}
\label{fig:ex3_insurance}
\end{figure}

The numerical results are presented in Figure~\ref{fig:ex3_insurance}. As expected, the drift $\mu$ (first plot in the top row) is positive, symmetric in $\theta_1$ and $\theta_2$, and increasing in each parameter. Indeed, a higher value of $\theta_i$ corresponds to a lower expected claim size, and therefore to a higher drift.
A similar mechanism is observed for the ruin probability $\psi(x)$, shown in the second row. It is decreasing in $x$ and symmetric in $\theta_1$ and $\theta_2$. Moreover, $\psi(x)$ decreases as either $\theta_i$ increases, since a higher $\theta_i$ again corresponds to a lower expected claim size and hence a lower probability of ruin.
Consequently, the derivatives $\partial \psi(x)/\partial \theta_i$ are negative, as illustrated in the second and third plots of the top row. Their magnitude is largest when both $\theta_1$ and $\theta_2$ are small, indicating that the sensitivity of the ruin probability to changes in $\theta_i$ is most pronounced in that region.



\section{Conclusion and Future Work}\label{sec:Conclusion}


We derived results for the sensitivity analysis of the stochastic fluid models for the first time. We illustrated the application of our methodology to stationary and transient (time-dependent) metrics of queueing systems through numerical examples. We constructed a simple example for which we derived explicit expressions. We also demonstrated application potential to deteriorating systems, and risk insurance processes.

The above analysis can be generalised to other classes of the SFMs, since the quantities considered here are the building blocks of the expressions within other SFMs. One relevant example is the class of SFMs referred to as the multi-layer SFMs (ML-SFMs), in which the fluid level is modulated by layer-dependent continuous-time Markov
chains, and which lends itself to many applications where the level variable may affect the transitions within the underlying environment or the rate of change of the level. For instance, in an overloaded telecommunication buffer, we may wish to change the operation of the on and off switches which control whether the data enters or leaves the buffer. In a queueing system with element of control, the volume of customers may affect the rate at which the system is emptying. In a deteriorating system, the level of deterioration may affect the deterioration rates, as discussed by Samuelson et al.~\cite{SAMUELSON20171169} in their analysis of a model related to our Example~\ref{ex:appl}.


Transient analysis of the general ML-SFMs was established by Bean and O'Reilly in~\cite{BOP2008}, where the process was first introduced in its full scale. Next, da Silva Soares and Latouche derived the results for the stationary analysis of ML-SFMs in~\cite{SL2009}. Further, Wu in~\cite{Wu2021}, and He and Wu in~\cite{HW2020,HW2019}, reviewed and refined the theory and algorithms for the analysis of ML-SFMs and developed ideas for the applications in queueing systems and risk models. For the comprehensive literature review of the ML-SFMs, the reader is referred to Wu~\cite{Wu2021}. Also, see Wu~\cite{Wu2021} for the details of the following three applications of the ML-SFMs approach to the analysis of queueing systems. The first is MAP/PH/K + GI queues, useful in the analysis of multi-server queueing systems with a moderately large number of servers. The second is the double-sided queues with MMAP and abandonment, useful in the analysis of ride-hailing platforms and organ transplantation systems. The third is the double-sided queues with BMAP and abandonment, useful in the analysis of perishable inventory systems and financial
markets.

Noting that the expressions for the stationary and transient analysis of these various classes of the ML-SFMs are written in terms of the key quantities for the standard SFMs, their sensitivity analysis can be derived as an extension of the results derived here, for a wide range of applications. These results will be reported in our future work.

\section{Statements and declarations}

\noindent{\bf Data}\\

No new data were created or analyzed during this study.\\ 

\noindent{\bf Declaration of competing interests}\\

The authors have no competing interests to declare that are relevant to the content of this article.\\


\bibliographystyle{abbrv}
\bibliography{refs_sens_sfms_2024_v000}

\end{document}